\documentclass[english]{article}
\usepackage[T1]{fontenc}
\usepackage[utf8]{inputenc}
\usepackage{babel}
\usepackage{float}
\usepackage{dsfont}
\usepackage{amsmath}
\usepackage{amsthm}
\usepackage{amssymb}
\usepackage[numbers]{natbib}
\usepackage[pdfusetitle,
 bookmarks=true,bookmarksnumbered=false,bookmarksopen=false,
 breaklinks=false,pdfborder={0 0 1},backref=false,colorlinks=false]
 {hyperref}

\makeatletter

\newcommand{\lyxmathsym}[1]{\ifmmode\begingroup\def\b@ld{bold}
  \text{\ifx\math@version\b@ld\bfseries\fi#1}\endgroup\else#1\fi}

\let\SF@@footnote\footnote
\def\footnote{\ifx\protect\@typeset@protect
    \expandafter\SF@@footnote
  \else
    \expandafter\SF@gobble@opt
  \fi
}
\expandafter\def\csname SF@gobble@opt \endcsname{\@ifnextchar[%]
  \SF@gobble@twobracket
  \@gobble
}
\edef\SF@gobble@opt{\noexpand\protect
  \expandafter\noexpand\csname SF@gobble@opt \endcsname}
\def\SF@gobble@twobracket[#1]#2{}
\floatstyle{ruled}
\newfloat{algorithm}{tbp}{loa}
\providecommand{\algorithmname}{Algorithm}
\floatname{algorithm}{\protect\algorithmname}

\theoremstyle{plain}
\newtheorem{assumption}{\protect\assumptionname}
\theoremstyle{plain}
\newtheorem{lem}{\protect\lemmaname}
\theoremstyle{remark}
\newtheorem{rem}{\protect\remarkname}
\theoremstyle{plain}
\newtheorem{thm}{\protect\theoremname}
\theoremstyle{plain}
\newtheorem{fact}{\protect\factname}

\@ifundefined{date}{}{\date{}}
\usepackage[main, final]{neurips_2026}
\usepackage[utf8]{inputenc} % allow utf-8 input
\usepackage[T1]{fontenc}    % use 8-bit T1 fonts
\usepackage{hyperref}       % hyperlinks
\usepackage{url}            % simple URL typesetting
\usepackage{booktabs}       % professional-quality tables
\usepackage{amsfonts}       % blackboard math symbols
\usepackage{nicefrac}       % compact symbols for 1/2, etc.
\usepackage{microtype}      % microtypography
\usepackage{xcolor}         % colors

\author{%
  Zijian Liu\\
  New York University\\
  \texttt{zl3067@stern.nyu.edu}
}

\usepackage{enumitem}
\usepackage{algorithmic}

\makeatother

\providecommand{\assumptionname}{Assumption}
\providecommand{\factname}{Fact}
\providecommand{\lemmaname}{Lemma}
\providecommand{\remarkname}{Remark}
\providecommand{\theoremname}{Theorem}

\begin{document}
\global\long\def\B{\mathbb{B}}%
\global\long\def\E{\mathbb{E}}%
\global\long\def\F{\mathcal{F}}%
\global\long\def\N{\mathbb{N}}%
\global\long\def\O{\mathcal{O}}%
\global\long\def\R{\mathbb{R}}%
\global\long\def\S{\mathbb{S}}%
\global\long\def\X{{\cal X}}%
\global\long\def\Z{{\cal Z}}%
\global\long\def\rb{\mathrm{b}}%
\global\long\def\rc{\mathrm{c}}%
\global\long\def\ru{\mathrm{u}}%
\global\long\def\bd{\boldsymbol{d}}%
\global\long\def\be{\boldsymbol{e}}%
\global\long\def\bg{\boldsymbol{g}}%
\global\long\def\bx{\boldsymbol{x}}%
\global\long\def\by{\boldsymbol{y}}%
\global\long\def\bz{\boldsymbol{z}}%
\global\long\def\p{\mathtt{p}}%
\global\long\def\q{\mathtt{q}}%
\global\long\def\bzero{\boldsymbol{0}}%
\global\long\def\1{\mathds{1}}%
\global\long\def\defeq{\triangleq}%
\global\long\def\argmin{\mathrm{argmin}}%
\global\long\def\dis{{\cal D}}%
\global\long\def\lar{\mathtt{l}}%
\global\long\def\sma{\mathtt{s}}%
\global\long\def\clip{\mathtt{clip}}%
\global\long\def\eff{\mathtt{eff}}%
\global\long\def\op{\mathtt{op}}%
\global\long\def\cvx{\mathrm{cvx}}%
\global\long\def\str{\mathrm{str}}%
\global\long\def\poly{\mathrm{poly}}%
\global\long\def\lip{\mathrm{lip}}%
\global\long\def\smo{\mathrm{smo}}%
\global\long\def\mydots{\dots}%
\global\long\def\cm{\tau}%
\global\long\def\hres{I}%
\global\long\def\hc{A}%
\global\long\def\eres{J}%
\global\long\def\ec{B}%
\global\long\def\lasthres{P}%
\global\long\def\lasthc{U}%
\global\long\def\constant{\mathtt{C}}%
\global\long\def\SGD{\mathtt{SGD}}%
\global\long\def\ClipSGD{\mathtt{Clipped\text{-}SGD}}%
\global\long\def\ProxClipSGD{\mathtt{Proximal\text{-}Clipped\text{-}SGD}}%

\title{On the Last-Iterate Convergence of\\Clipped Gradient Methods}
\maketitle
\begin{abstract}
In stochastic (strongly) convex optimization, establishing last-iterate
convergence guarantees for first-order methods has recently garnered
increasing attention, as this not only deepens the theoretical understanding
of these algorithms but also aligns more closely with practice. Notably,
people have demonstrated that one of the most famous, simple, and
popular methods, Stochastic Gradient Descent ($\mathtt{SGD}$), provably
achieves last-iterate convergence. However, the existing results may
have limited applicability, since they typically rely on strong conditions
on the gradient noise (e.g., an exponentially decaying tail) rather
than the more realistic heavy-tailed noise. Unfortunately, when faced
with heavy-tailed noise, $\mathtt{SGD}$ is known to exhibit undesirable
behavior or even fail to converge. To address the challenge of heavy-tailed
noise, people have proposed $\mathtt{Clipped}\text{-}\mathtt{SGD}$,
an algorithm that combines $\mathtt{SGD}$ with a simple mechanism,
gradient clipping. Although $\mathtt{Clipped}\text{-}\mathtt{SGD}$
performs well in practice, its last-iterate convergence property remains
highly underexplored. In this work, we prove the first optimal last-iterate
rates in high probability for $\mathtt{Clipped}\text{-}\mathtt{SGD}$
under heavy-tailed noise, thereby closing a gap in the literature.
\end{abstract}

\section{Introduction\protect\footnote{To save space, we only provide a simplified leading term in the rate
when discussing convergence results in this section.}\label{sec:introduction}}

In stochastic (strongly) convex optimization, establishing last-iterate
convergence guarantees\footnote{\label{fn:metric}Throughout the paper, given an objective function
$F$ and a point $\bx$ in the optimization domain, we measure the
convergence by the function value gap, i.e., $F(\bx)-F_{\star}$,
where $F_{\star}=F(\bx_{\star})$ denotes the function value at the
optimal solution $\bx_{\star}$. Note that for $\mu$-strongly convex
$F$ (not necessarily smooth), the function value gap always performs
as a stronger convergence criterion than the commonly used metric
$\mu\left\Vert \bx-\bx_{\star}\right\Vert ^{2}$, as discussed in
the literature (see, e.g., \citep{ICLR2024_b1239ceb}).} for first-order methods has attracted much attention in recent years.
This topic not only deepens our theoretical understanding of different
algorithms but also aligns more closely with reality, given that the
last iterate is often used as the final output in practice. Notably,
a particularly exciting line of work (see, e.g., \citep{NEURIPS2025_97335d45,garrigos2025last,pmlr-v99-harvey19a,doi:10.1137/19M128908X,ICLR2024_b1239ceb,NIPS2011_40008b9a,orabona2020blog,rakhlin2011making,pmlr-v28-shamir13})
has demonstrated that one of the most famous, simple, and popular
methods, Stochastic Gradient Descent ($\SGD$) \citep{robbins1951stochastic},
which serves as the prototype for many follow-up first-order algorithms,
provably achieves last-iterate convergence rates, even optimal ones,
for both Lipschitz and smooth objectives.

However, upon closer examination, one may realize that the existing
last-iterate convergence results for $\SGD$ may have limited applicability,
since they typically rely on strong conditions on the gradient noise
(which refers to the difference between the stochastic and the true
gradients). Especially, two commonly considered conditions, the finite-variance
noise assumption for convergence in expectation and the exponentially
decaying tail noise assumption (e.g., sub-Gaussian noise) for convergence
in high probability, do not accurately reflect practice, as they are
observed to be violated in many different scenarios \citep{ICLR2024_4661b552,pmlr-v238-battash24a,pmlr-v139-garg21b,pmlr-v97-simsekli19a,NEURIPS2020_b05b57f6}.
Given this, a more realistic assumption, known as heavy-tailed noise,
has been proposed to model the behavior of gradient noise. Specifically,
under the assumption of heavy-tailed noise, the gradient noise is
considered to have only a finite $\p$-th moment, where $\p\in\left(1,2\right]$
is called the tail index.

Unfortunately, when faced with heavy-tailed noise, $\SGD$ is known
to exhibit undesirable behavior, such as an inevitable $\poly(1/\delta)$
dependence on the failure probability $\delta\in\left(0,1\right]$
in high-probability convergence \citep{pmlr-v313-liu26a,pmlr-v202-sadiev23a},
or even to fail to converge in expectation \citep{NEURIPS2020_b05b57f6}.
To address this challenge, people have proposed $\ClipSGD$, an algorithm
that combines $\SGD$ with a simple mechanism known as gradient clipping.
Concretely, in contrast to $\SGD$, which directly applies the update
rule to the stochastic gradient $\bg$, $\ClipSGD$ first preprocesses
$\bg$ via the clipping function to obtain its clipped counterpart,
$\clip_{\cm}(\bg)\defeq\min\left\{ \cm,\left\Vert \bg\right\Vert \right\} \frac{\bg}{\left\Vert \bg\right\Vert }$,
where $\cm$ is called the clipping threshold, and then uses the clipped
stochastic gradient in the following procedure.

Although $\ClipSGD$ performs well in practice and has provable convergence
guarantees in various settings, its last-iterate convergence property
remains highly underexplored. As far as we are aware, only two studies
have established last-iterate rates under heavy-tailed noise. The
first is \citep{parletta2024improved}, which showed a last-iterate
rate of $\O(1/T^{1-\frac{1}{\p}})$ in expectation for Lipschitz convex
objectives; this rate is known to be optimal in $T$ \citep{ICLR2026_bd5b846a,nemirovskij1983problem,pmlr-v178-vural22a}.
The second is \citep{ICLR2026_03382840}, in which the authors developed
a last-iterate rate of $\O(\ln^{2}(T^{2}/\delta)/T^{\frac{2}{3}(1-\frac{1}{\p})})$
in high probability for smooth convex objectives. While, to the best
of our knowledge, no high-probability lower bounds in the corresponding
setting are known in the literature, the rate in \citep{ICLR2026_03382840}
appears to be suboptimal for the following two reasons. First, it
is worse than the best-known high-probability rate $\O(\ln^{1-\frac{1}{\p}}(1/\delta)/T^{1-\frac{1}{\p}})$
for the average iterate under the same setting \citep{NEURIPS2023_4c454d34}.
Second, when $\p=2$, the rate in \citep{ICLR2026_03382840} reduces
to only $\O(\ln^{2}(T^{2}/\delta)/T^{\frac{1}{3}})$, which is slower
than the best-known in-expectation rate $\O(1/\sqrt{T})$ for the
last iterate of $\SGD$ \citep{ICLR2024_b1239ceb}.

Given the above discussion, an important research question remains
to be investigated:
\begin{center}
\textit{Under heavy-tailed noise, can we establish high-probability
rates for the last iterate of $\ClipSGD$ that match its best-known
high-probability bounds for the average iterate?}
\par\end{center}

\subsection{Our Contributions}

This work answers the above question affirmatively by proving the
following results, thereby closing a gap in the literature.

\begin{itemize}[leftmargin=*]

\item For general convex optimization, we prove that the last iterate
of $\ClipSGD$ converges in high probability at a rate of $\O(\ln^{1-\frac{1}{\p}}(1/\delta)/T^{1-\frac{1}{\p}})$
for both Lipschitz and smooth objectives (cf. Theorems \ref{thm:main-cvx-lip-hp}
and \ref{thm:main-cvx-smo-hp}), matching the best-known average-iterate
results in \citep{ICLR2026_bd5b846a} for Lipschitz functions and
\citep{NEURIPS2023_4c454d34} for smooth functions.

\item For strongly convex optimization\footnote{In fact, our results extend beyond strongly convex functions. See
Assumption \ref{assu:obj} for details.}, we prove that the last iterate of $\ClipSGD$ converges in high
probability at a rate of $\O(\ln^{2-\frac{2}{\p}}(1/\delta)/T^{2-\frac{2}{\p}})$
for both Lipschitz and smooth objectives (cf. Theorems \ref{thm:main-str-lip-hp}
and \ref{thm:main-str-smo-hp}), matching the best-known average-iterate
result in \citep{ICLR2026_bd5b846a} for Lipschitz functions and improving
over the result in \citep{pmlr-v202-sadiev23a} for smooth functions
by removing extra logarithmic factors.

\item For Lipschitz objectives, our rates are optimal in all parameters
(including other problem-dependent parameters currently hidden in
the big-$\O$ notation) under the classical heavy-tailed noise assumption,
as implied by the high-probability minimax lower bounds in \citep{ICLR2026_bd5b846a}.
For smooth objectives, although no high-probability minimax lower
bounds are known, we believe that our rates are also optimal under
the classical heavy-tailed noise assumption.

\item Notably, our rates are derived under the fine-grained heavy-tailed
noise condition (cf. Assumption \ref{assu:heavy}) recently introduced
in \citep{ICLR2026_bd5b846a} and exhibit a refined dependence on
the noise parameter, as in \citep{ICLR2026_bd5b846a}.

\item We highlight that our analysis contains an interesting idea.
Specifically, we develop a new argument that maintains the optimal
$\poly(\ln(1/\delta))$ dependence on $\delta$. Our analysis is inspired
by induction-based proofs in the literature, which, in contrast, typically
lead to a suboptimal $\poly(\ln(T/\delta))$ dependence. For more
details, we refer the interested reader to Section \ref{sec:idea}.

\end{itemize}

\subsection{Related Work\label{subsec:related}}

First, we review the literature on high-probability convergence results
for $\ClipSGD$ in stochastic (strongly) convex optimization under
heavy-tailed noise. For high-probability convergence in stochastic
nonconvex optimization under heavy-tailed noise (not limited to the
$\ClipSGD$ algorithm), the interested reader could refer to, for
example, \citep{pmlr-v258-armacki25a,NEURIPS2021_26901deb,pmlr-v300-fradin26a,pmlr-v258-hubler25a,pmlr-v235-liu24bo,pmlr-v195-liu23c,NEURIPS2023_4c454d34,pmlr-v202-sadiev23a},
for recent progress. 

\paragraph{Upper bounds.}

For convex problems, the best-known high-probability rate for the
average iterate is $\O(\ln^{1-\frac{1}{\p}}(1/\delta)/T^{1-\frac{1}{\p}})$,
for both Lipschitz functions \citep{ICLR2026_bd5b846a} and smooth
functions \citep{NEURIPS2023_4c454d34}. For strongly convex problems,
the best-known high-probability bounds are $\O(\ln^{2-\frac{2}{\p}}(1/\delta)/T^{2-\frac{2}{\p}})$
for the weighted average iterate and the squared distance $\left\Vert \bx_{T+1}-\bx_{\star}\right\Vert ^{2}$
in the Lipschitz setting \citep{ICLR2026_bd5b846a} and $\widetilde{\O}(\ln^{2-\frac{2}{\p}}(T/\delta)/T^{2-\frac{2}{\p}})$
for $\left\Vert \bx_{T+1}-\bx_{\star}\right\Vert ^{2}$ in the smooth
setting \citep{pmlr-v202-sadiev23a}.

Regarding last-iterate convergence, the only work we are aware of
is \citep{ICLR2026_03382840}, which proved the first and only rate
in the order of $\O(\ln^{2}(T^{2}/\delta)/T^{\frac{2}{3}(1-\frac{1}{\p})})$
for smooth convex objectives. However, such a bound is worse than
the best-known average-iterate rate $\O(\ln^{1-\frac{1}{\p}}(1/\delta)/T^{1-\frac{1}{\p}})$
mentioned above for the same setting. Note that in the strongly convex
case, as discussed in Footnote \ref{fn:metric}, the existing results
for $\left\Vert \bx-\bx_{\star}\right\Vert ^{2}$ do not align with
the metric $F(\bx)-F_{\star}$ considered in this work.

Second, we discuss lower bounds for stochastic (strongly) convex optimization
under heavy-tailed noise (or similar conditions). Given the limited
number of results in the literature, we include both high-probability
and in-expectation lower bounds.

\paragraph{Lower bounds.}

To the best of our knowledge, only a few results of high-probability
lower bounds are known, and all of them are established for Lipschitz
functions. In detail, \citep{5394945} proved the first high-probability
minimax lower bound $\Omega(1/T^{1-\frac{1}{\p}})$ for convex optimization.
However, it considered the heavy-tailed stochastic gradient assumption
rather than the heavy-tailed noise assumption, and, more critically,
did not provide a tight dependence on parameters other than $T$,
such as the failure probability $\delta$. Under the same heavy-tailed
stochastic gradient assumption, but only for $\p=2$, \citep{pmlr-v247-carmon24a}
showed an improved minimax lower bound $\Omega(\sqrt{\ln(1/\delta)/T})$;
for general $\p\in\left(1,2\right]$, \citep{fatkhullin2025can} gave
an algorithm-dependent lower bound $\Omega((1/\delta)^{\frac{1}{\p}}/T^{1-\frac{1}{\p}})$.
Under the heavy-tailed noise assumption, \citep{ICLR2026_bd5b846a}
recently characterized the first high-probability minimax lower bounds,
namely, $\Omega(\ln^{1-\frac{1}{\p}}(1/\delta)/T^{1-\frac{1}{\p}})$
for convex optimization and $\Omega(\ln^{2-\frac{2}{\p}}(1/\delta)/T^{2-\frac{2}{\p}})$
for strongly convex optimization. 

Next, we turn our attention to in-expectation lower bounds. Under
the heavy-tailed stochastic gradient assumption, \citep{nemirovskij1983problem}
and \citep{pmlr-v178-vural22a} provided a minimax lower bound $\Omega(1/T^{1-\frac{1}{\p}})$
for Lipschitz convex functions, and \citep{NEURIPS2020_b05b57f6}
established the first $\Omega(1/T^{2-\frac{2}{\p}})$ result for smooth
strongly convex objectives. Under the heavy-tailed noise assumption,
\citep{ICLR2026_bd5b846a} recently showed similar lower bounds $\Omega(1/T^{1-\frac{1}{\p}})$
and $\Omega(1/T^{2-\frac{2}{\p}})$, for Lipschitz convex and strongly
convex optimization, respectively.

Finally, we remark that all lower bounds in \citep{ICLR2026_bd5b846a}
can apply to the fine-grained heavy-tailed noise assumption beyond
the classical one. We will elaborate further on this point later when
discussing our theorems in Section \ref{sec:algorithm}.

\section{Preliminary\label{sec:preliminary}}

\paragraph{Notation.}

$\N$ is the set of natural numbers (excluding $0$). $\left[T\right]\triangleq\left\{ 1,\mydots,T\right\} ,\forall T\in\N$.
$\left\lceil \cdot\right\rceil $ and $\left\lfloor \cdot\right\rfloor $
denote the ceiling and floor functions, respectively. $\left\langle \cdot,\cdot\right\rangle $
is the standard Euclidean inner product, and $\left\Vert \cdot\right\Vert \defeq\sqrt{\left\langle \cdot,\cdot\right\rangle }$
is the Euclidean norm. $\left\Vert \cdot\right\Vert _{\op}$ denotes
the operator norm. $\S^{d-1}$ is the unit sphere in $\R^{d}$. $\B^{d}(\bx,r)$
is the Euclidean ball centered at $\bx\in\R^{d}$ with radius $r>0$.
Given a convex $h:\R^{d}\to\R$, $\nabla h(\bx)$ denotes an arbitrary
element in $\partial h(\bx)$, where $\partial h(\bx)$ is the subgradient
set of $h$ at $\bx$.

\paragraph{Objective.}

We study the composite optimization problem in the form of
\[
\inf_{\bx\in\X}F(\bx)\defeq f(\bx)+r(\bx),
\]
where $\X\subseteq\R^{d}$ is a nonempty closed convex set.

\paragraph{Assumptions.}

Our analysis relies on the following assumptions.
\begin{assumption}
\label{assu:minimizer}There exists $\bx_{\star}\in\X$ such that
$F_{\star}\defeq F(\bx_{\star})=\inf_{\bx\in\X}F(\bx)$.
\end{assumption}
\begin{assumption}
\label{assu:obj}Both $f:\R^{d}\to\R$ and $r:\R^{d}\to\R$ are convex.
Moreover, $F$ is $\mu$-strongly star-convex on $\X$ for some $\mu\geq0$,
i.e., $F(\lambda\bx_{\star}+(1-\lambda)\bx)+\frac{\mu\lambda(1-\lambda)}{2}\left\Vert \bx_{\star}-\bx\right\Vert ^{2}\leq\lambda F_{\star}+(1-\lambda)F(\bx),\forall\bx\in\X,\lambda\in\left[0,1\right]$,
where $\bx_{\star}$ is the minimizer specified in Assumption \ref{assu:minimizer}.
\end{assumption}
\begin{assumption}
\label{assu:lip}$f$ is $G$-Lipschitz on $\X$, i.e., $\left\Vert \nabla f(\bx)\right\Vert \leq G,\forall\bx\in\X$.
\end{assumption}
\begin{assumption}
\label{assu:smo}$f$ is $L$-smooth on $\X$, i.e., $\left\Vert \nabla f(\bx)-\nabla f(\by)\right\Vert \leq L\left\Vert \bx-\by\right\Vert ,\forall\bx,\by\in\X$.
\end{assumption}
Assumptions \ref{assu:minimizer}, \ref{assu:lip}, and \ref{assu:smo}
are standard in the optimization literature \citep{doi:10.1137/16M1080173,lan2020first,nesterov2018lectures}.
The convexity requirement in Assumption \ref{assu:obj} is also fundamental.
We remark that the $\mu$-strong star-convexity condition \citep{necoara2019linear}
in Assumption \ref{assu:obj} is strictly weaker than the commonly
assumed $\mu$-strong convexity condition in the literature.

Before moving on, we recall a useful property of $\mu$-strongly star-convex
functions (see, e.g., \citep{pmlr-v125-hinder20a}) in the following
Lemma \ref{lem:str-star-cvx}. For completeness, we reproduce its
proof in Appendix \ref{sec:proof-of-lemma}.
\begin{lem}
\label{lem:str-star-cvx}Under Assumptions \ref{assu:minimizer} and
\ref{assu:obj}, we have $\frac{\mu}{2}\left\Vert \bx-\bx_{\star}\right\Vert ^{2}\leq F(\bx)-F_{\star}\leq\frac{\left\Vert \nabla F(\bx)\right\Vert ^{2}}{2\mu},\forall\bx\in\X$.
\end{lem}
Next, we consider a fine-grained heavy-tailed noise assumption recently
introduced in \citep{ICLR2026_bd5b846a}. For the plausibility of
Assumption \ref{assu:heavy}, we refer the reader to the discussions
in Section 2 of \citep{ICLR2026_bd5b846a}.
\begin{assumption}
\label{assu:heavy}There exists a function $\bg:\X\times\Z\to\R^{d}$
and a probability distribution $\dis$ on $\Z$ such that $\E_{Z\sim\dis}\left[\bg(\bx;Z)\right]=\nabla f(\bx),\forall\bx\in\X$.
In addition, for some $\p\in\left(1,2\right]$, we have
\begin{eqnarray*}
\E_{Z\sim\dis}\left[\left|\left\langle \be,\bg(\bx;Z)-\nabla f(\bx)\right\rangle \right|^{\p}\right]\leq\sigma_{\sma}^{\p}, & \E_{Z\sim\dis}\left[\left\Vert \bg(\bx;Z)-\nabla f(\bx)\right\Vert ^{\p}\right]\leq\sigma_{\lar}^{\p}, & \forall\be\in\S^{d-1},\bx\in\X,
\end{eqnarray*}
where $\sigma_{\sma}$ and $\sigma_{\lar}$ are two constants satisfying
$0\leq\sigma_{\sma}\leq\sigma_{\lar}\leq\sqrt{\pi d/2}\sigma_{\sma}.$
\end{assumption}
\begin{rem}
Hereafter, we omit $Z\sim\dis$ from $\E_{Z\sim\dis}$ when the context
is clear.
\end{rem}
\begin{rem}
Assumption \ref{assu:heavy} recovers the classical heavy-tailed noise
assumption by setting both $\sigma_{\sma}$ and $\sigma_{\lar}$ equal
to a common parameter $\sigma$ (or by simply ignoring the existence
of $\sigma_{\sma}$ and viewing $\sigma_{\lar}$ as the single parameter
of gradient noise).
\end{rem}
As in \citep{ICLR2026_bd5b846a}, we also define the generalized effective
dimension $d_{\eff}$ as follows:
\begin{equation}
d_{\eff}\defeq\sigma_{\lar}^{2}/\sigma_{\sma}^{2}\in\left\{ 0\right\} \cup\left[1,\pi d/2\right],\label{eq:main-d-eff}
\end{equation}
where $d_{\eff}=0$ if and only if $\sigma_{\sma}=\sigma_{\lar}=0$,
i.e., the deterministic case. Note that $d_{\eff}$ can attain order
$\Theta(d)$ \citep{ICLR2026_bd5b846a} and recovers the effective
dimension used in \citep{NEURIPS2024_10bf9689} when $\p=2$.

\section{Last-Iterate Convergence of Clipped Stochastic Gradient Descent\label{sec:algorithm}}

\begin{algorithm}[H]
\caption{\label{alg:Clipped-SGD}Clipped Stochastic Gradient Descent ($\protect\ClipSGD$)}

\begin{algorithmic}[1]

\STATE \textbf{Input:} initial point $\bx_{1}\in\X$, time horizon
$T\in\N$

\FOR{$t=1$ \textbf{to} $T$}

\STATE Decide a clipping threshold $\cm_{t}>0$

\STATE $\bg_{t}^{\rc}=\clip_{\cm_{t}}(\bg_{t})$ where $\bg_{t}=\bg(\bx_{t};Z_{t})$
and $Z_{t}\sim\dis$ is independent from the history

\STATE Decide a stepsize $\eta_{t}>0$

\STATE $\bx_{t+1}=\argmin_{\bx\in\X}r(\bx)+\left\langle \bg_{t}^{\rc},\bx\right\rangle +\frac{\left\Vert \bx-\bx_{t}\right\Vert ^{2}}{2\eta_{t}}$

\ENDFOR

\STATE \textbf{Output:} $\bx_{T+1}$

\end{algorithmic}
\end{algorithm}

\begin{rem}
In the same manner as \citep{ICLR2026_bd5b846a}, we refer to Algorithm
\ref{alg:Clipped-SGD} simply as $\ClipSGD$, rather than $\ProxClipSGD$.
\end{rem}
\begin{rem}
We treat the time horizon $T$ in Algorithm \ref{alg:Clipped-SGD}
as known to avoid any potential logarithmic dependence on $T$.
\end{rem}
This work studies the Clipped Stochastic Gradient Descent ($\ClipSGD$)
method, as given in Algorithm \ref{alg:Clipped-SGD}, and will establish
the first optimal last-iterate convergence rates in high probability. 

To ease the notation in the upcoming theorems, we denote by $D\defeq\left\Vert \bx_{\star}-\bx_{1}\right\Vert $
the distance between the optimal solution $\bx_{\star}$ and the initial
point $\bx_{1}$. Given a failure probability $\delta\in\left(0,1\right]$,
we adopt an important quantity $\cm_{\star}$ proposed in \citep{ICLR2026_bd5b846a},
defined as
\begin{equation}
\cm_{\star}\defeq\left(\min\left\{ \frac{\sigma_{\sma}\sigma_{\lar}^{\p-1}}{\ln\frac{3}{\delta}},\frac{\sigma_{\sma}^{2}}{\sigma_{\lar}^{2-\p}\1\left[\p<2\right]}\right\} \right)^{\frac{1}{\p}}.\label{eq:main-hp-tau-star}
\end{equation}
By the definition of the generalized effective dimension $d_{\eff}\overset{(\ref{eq:main-d-eff})}{=}\sigma_{\lar}^{2}/\sigma_{\sma}^{2}$,
$\cm_{\star}$ can also be written as
\begin{eqnarray}
\cm_{\star}=\sigma_{\lar}/\varphi_{\star}^{1/\p} & \text{where} & \varphi_{\star}\defeq\max\left\{ \sqrt{d_{\eff}}\ln\frac{3}{\delta},d_{\eff}\1\left[\p<2\right]\right\} .\label{eq:main-hp-varphi-star}
\end{eqnarray}

\subsection{General Convex Optimization\label{subsec:cvx-hp}}

This subsection focuses on general convex optimization, i.e., $\mu=0$
in Assumption \ref{assu:obj}.

\subsubsection{Lipschitz Objectives}

We first consider the case where $f$ is $G$-Lipschitz on $\X$.
\begin{thm}
\label{thm:main-cvx-lip-hp}Under Assumptions \ref{assu:minimizer},
\ref{assu:obj} (with $\mu=0$), \ref{assu:lip}, and \ref{assu:heavy},
for any $T\in\N$ and $\delta\in\left(0,1\right]$, setting $\eta_{t}=\eta_{\star}(T-t+1),\cm_{t}=\max\left\{ 2G,\cm_{\star}T^{\frac{1}{\p}}\right\} ,\forall t\in\left[T\right]$,
where $\eta_{\star}$ is a properly picked parameter defined in (\ref{eq:cvx-lip-hp-eta-star}),
then $\ClipSGD$ (Algorithm \ref{alg:Clipped-SGD}) guarantees that,
with probability at least $1-\delta$,
\[
F(\bx_{T+1})-F_{\star}=\O\left(\frac{(\varphi_{\lip}+\ln\frac{3}{\delta})GD}{T}+\frac{GD}{\sqrt{T}}+\frac{(\sigma_{\sma}^{\frac{2}{\p}-1}\sigma_{\lar}^{2-\frac{2}{\p}}+\sigma_{\sma}^{\frac{1}{\p}}\sigma_{\lar}^{1-\frac{1}{\p}}\ln^{1-\frac{1}{\p}}\frac{3}{\delta})D}{T^{1-\frac{1}{\p}}}\right),
\]
in which $\varphi_{\lip}\leq\varphi_{\star}$ is a constant defined
in (\ref{eq:lip-hp-varphi}) and equals $\varphi_{\star}$ when $T=\Omega\left(\frac{G^{\p}}{\sigma_{\lar}^{\p}}\varphi_{\star}\right)$.
\end{thm}
To the best of our knowledge, for Lipschitz objectives, Theorem \ref{thm:main-cvx-lip-hp}
provides the first and only high-probability convergence rate for
the last iterate of $\ClipSGD$. Even in the special case $\p=2$,
i.e., finite-variance noise, no such high-probability last-iterate
convergence results exist in the literature.

We highlight that our Theorem \ref{thm:main-cvx-lip-hp} exactly matches
the best-known average-iterate convergence rate of $\ClipSGD$ in
high probability \citep{ICLR2026_bd5b846a} and hence has three implications.
First, under the classical heavy-tailed noise condition in the literature
(i.e., $\sigma_{\sma}=\sigma_{\lar}=\sigma$ for some $\sigma\geq0$),
the rate in Theorem \ref{thm:main-cvx-lip-hp} reduces to $\O\left(\frac{GD\ln\frac{1}{\delta}}{T}+\frac{GD}{\sqrt{T}}+\frac{\sigma D\ln^{1-\frac{1}{\p}}\frac{1}{\delta}}{T^{1-\frac{1}{\p}}}\right)$
by noting that $d_{\eff}\overset{(\ref{eq:main-d-eff})}{=}\sigma_{\lar}^{2}/\sigma_{\sma}^{2}=1\Rightarrow\varphi_{\lip}\leq\varphi_{\star}\overset{(\ref{eq:main-hp-varphi-star})}{=}\ln\frac{3}{\delta}$.
Remarkably, the leading term $\O\left(\frac{\sigma D\ln^{1-\frac{1}{\p}}\frac{1}{\delta}}{T^{1-\frac{1}{\p}}}\right)$
in the resulting bound is known to be fully optimal \citep{ICLR2026_bd5b846a}.
Second, under the fine-grained heavy-tailed noise assumption (i.e.,
$\sigma_{\sma}$ and $\sigma_{\lar}$ may differ), the rate established
in Theorem \ref{thm:main-cvx-lip-hp} is optimal in $T$, yet may
be suboptimal in terms of the problem-dependent parameters $\sigma_{\sma}$
and $\sigma_{\lar}$. This is because, under Assumption \ref{assu:heavy},
the only existing lower bound is $\Omega\left(\frac{(\sigma_{\sma}^{\frac{2}{\p}-1}\sigma_{\lar}^{2-\frac{2}{\p}}+\sigma_{\sma}\ln^{1-\frac{1}{\p}}\frac{1}{\delta})D}{T^{1-\frac{1}{\p}}}\right)$
\citep{ICLR2026_bd5b846a}, which suggests that there could still
be room to further improve the dependence on $\sigma_{\sma}$ and
$\sigma_{\lar}$. Third, when $\sigma_{\lar}$, equivalently $\sigma_{\sma}$,
tends to $0$, Theorem \ref{thm:main-cvx-lip-hp} recovers the standard
$\O\left(\frac{GD}{\sqrt{T}}\right)$ rate for deterministic Lipschitz
convex optimization.

Moreover, we remark that the linearly decaying stepsize employed in
Theorem \ref{thm:main-cvx-lip-hp} is critical to our proof. This
kind of stepsize was originally introduced in \citep{doi:10.1137/24M1717762}
to achieve an optimal convergence rate for Gradient Descent on Lipschitz
objectives in deterministic convex optimization. Subsequently, \citep{ICLR2024_b1239ceb}
demonstrated that it can also be applied to stochastic (strongly)
convex optimization and obtained optimal last-iterate convergence
rates in various settings. Our work further investigates the potential
of such a stepsize in $\ClipSGD$ and, consequently, establishes the
first optimal last-iterate convergence rates in high probability under
heavy-tailed noise.

\subsubsection{Smooth Objectives}

Next, we consider the setting where $f$ is $L$-smooth on $\X$.
In this case, we need an extra assumption.
\begin{assumption}
\label{assu:G1}We have access to a constant $G_{1}$ satisfying $G_{1}\geq\left\Vert \nabla f(\bx_{1})\right\Vert $.
\end{assumption}

\paragraph{Discussion on Assumption \ref{assu:G1}.}

We note that the same assumption was also imposed before in \citep{NEURIPS2023_4c454d34}
under the setting $r=0$ (see their Assumption (5)). As pointed out
by \citep{NEURIPS2023_4c454d34}, this is a mild condition, since
there are several ways to estimate $G_{1}$. For example, when $\X=\R^{d}$
and $r=0$ as in most prior work, we have $\nabla f(\bx_{\star})=\bzero$
and hence can set $G_{1}=LD$ by noticing that $\left\Vert \nabla f(\bx_{1})\right\Vert =\left\Vert \nabla f(\bx_{1})-\nabla f(\bx_{\star})\right\Vert \leq LD$.
For more details on the justification of Assumption \ref{assu:G1},
we kindly refer the interested reader to \citep{NEURIPS2023_4c454d34}.

Equipped with Assumption \ref{assu:G1}, we can prove the following
Theorem \ref{thm:main-cvx-smo-hp}.
\begin{thm}
\label{thm:main-cvx-smo-hp}Under Assumptions \ref{assu:minimizer},
\ref{assu:obj} (with $\mu=0$), \ref{assu:smo}, \ref{assu:heavy},
and \ref{assu:G1}, for any $T\in\N$ and $\delta\in\left(0,\frac{1}{2}\right]$,
setting $\eta_{t}=\eta_{\star}(T-t+1),\cm_{t}=\max\left\{ 2L\left\Vert \bx_{t}-\bx_{1}\right\Vert +2G_{1},\cm_{\star}T^{\frac{1}{\p}}\right\} ,\forall t\in\left[T\right]$,
where $\eta_{\star}$ is a properly picked parameter defined in (\ref{eq:cvx-smo-hp-eta-star}),
then $\ClipSGD$ (Algorithm \ref{alg:Clipped-SGD}) guarantees that,
with probability at least $1-2\delta$,
\[
F(\bx_{T+1})-F_{\star}=\O\left(\frac{(\varphi_{\smo}+\ln\frac{3}{\delta})(LD+G_{1})D}{T}+\frac{(\sigma_{\sma}^{\frac{2}{\p}-1}\sigma_{\lar}^{2-\frac{2}{\p}}+\sigma_{\sma}^{\frac{1}{\p}}\sigma_{\lar}^{1-\frac{1}{\p}}\ln^{1-\frac{1}{\p}}\frac{3}{\delta})D}{T^{1-\frac{1}{\p}}}\right),
\]
in which $\varphi_{\smo}\leq\varphi_{\star}$ is a constant defined
in (\ref{eq:smo-hp-varphi}) and equals $\varphi_{\star}$ when $T=\Omega\left(\frac{(LD+G_{1})^{\p}}{\sigma_{\lar}^{\p}}\varphi_{\star}\right)$.
\end{thm}
As far as we know, for smooth objectives, a recent work \citep{ICLR2026_03382840}
obtained the first and currently only high-probability convergence
result for the last iterate of $\ClipSGD$ in the setting where $r=0$,
$\X=\R^{d}$, and $\sigma_{\sma}=\sigma_{\lar}=\sigma$ (i.e., the
classical heavy-tailed noise assumption) with a rate $\O\left(\frac{LD^{2}\ln^{2}\frac{T^{2}}{\delta}}{T}+\frac{\max\left\{ \sigma D,L^{\frac{\p-1}{3\p-1}}\sigma^{\frac{2\p}{3\p-1}}D^{\frac{4\p-2}{3\p-1}},L^{\frac{1}{3}}\sigma^{\frac{2}{3}}D^{\frac{4}{3}}\right\} \ln^{2}\frac{T^{2}}{\delta}}{T^{\frac{2}{3}\left(1-\frac{1}{\p}\right)}}\right)$.
However, this rate is much worse than the best-known average-iterate
rate $\O\left(\frac{LD^{2}\ln\frac{1}{\delta}}{T}+\frac{\sigma D\ln^{1-\frac{1}{\p}}\frac{1}{\delta}}{T^{1-\frac{1}{\p}}}\right)$
\citep{NEURIPS2023_4c454d34}. In contrast, by taking $G_{1}=LD$
and using $d_{\eff}\overset{(\ref{eq:main-d-eff})}{=}\sigma_{\lar}^{2}/\sigma_{\sma}^{2}=1\Rightarrow\varphi_{\smo}\leq\varphi_{\star}\overset{(\ref{eq:main-hp-varphi-star})}{=}\ln\frac{3}{\delta}$
in this special case, the last-iterate rate in Theorem \ref{thm:main-cvx-smo-hp}
can be simplified into $\O\left(\frac{LD^{2}\ln\frac{1}{\delta}}{T}+\frac{\sigma D\ln^{1-\frac{1}{\p}}\frac{1}{\delta}}{T^{1-\frac{1}{\p}}}\right)$,
which matches the bound in \citep{NEURIPS2023_4c454d34} and substantially
improves over \citep{ICLR2026_03382840}. Moreover, we believe that
the leading term in this simplified rate is optimal, although, as
previously noted, no high-probability lower bound exists in this setting.

Notably, our Theorem \ref{thm:main-cvx-smo-hp} holds in a more general
setting. In particular, it is the first provable result applicable
to smooth objectives under the fine-grained Assumption \ref{assu:heavy}.
Although it is currently unclear whether the bound is tight or not
(since the lower bounds in \citep{ICLR2026_bd5b846a} were proved
only for Lipschitz $f$), we believe that a conclusion similar to
the previous one for the Lipschitz case holds, that is, the leading
term in Theorem \ref{thm:main-cvx-smo-hp} is tight in $T$, whereas
the term $\sigma_{\sma}^{1/\p}\sigma_{\lar}^{1-1/\p}$ may not be
optimal (and we conjecture that the optimal dependence should be $\sigma_{\sma}$).

In addition, we would like to briefly discuss the first component,
$2L\left\Vert \bx_{t}-\bx_{1}\right\Vert +2G_{1}$, in our clipping
threshold $\cm_{t}$. The reader familiar with the literature may
immediately recognize a similar choice, $2L\max_{s\in\left[t\right]}\left\Vert \bx_{s}-\bx_{1}\right\Vert +2G_{1}$,
proposed in \citep{NEURIPS2023_4c454d34}, which directly inspired
our choice. However, as one can see, ours made a simplification by
replacing $\max_{s\in\left[t\right]}\left\Vert \bx_{s}-\bx_{1}\right\Vert $
with $\left\Vert \bx_{t}-\bx_{1}\right\Vert $.

\subsection{Strongly Star-Convex Optimization\label{subsec:str-hp}}

This subsection focuses on strongly star-convex optimization, i.e.,
$\mu>0$ in Assumption \ref{assu:obj}. To avoid algebraic issues
in the analysis, we consider only $T\geq2$.

\subsubsection{Lipschitz Objectives}

As before, we begin with the case where $f$ is $G$-Lipschitz on
$\X$.
\begin{thm}
\label{thm:main-str-lip-hp}Under Assumptions \ref{assu:minimizer},
\ref{assu:obj} (with $\mu>0$), \ref{assu:lip}, and \ref{assu:heavy},
for any $T\geq2$ and $\delta\in\left(0,\frac{1}{2}\right]$, setting
$\eta_{t}=\begin{cases}
\frac{6}{\mu t} & t\leq S\\
\frac{T-t+1}{\mu T^{2}} & t>S
\end{cases},\cm_{t}=\begin{cases}
\max\left\{ 2G,\cm_{\star}t^{\frac{1}{\p}}\right\}  & t\leq S\\
\max\left\{ 2G,\cm_{\star}\left\lfloor \frac{T}{2}\right\rfloor ^{\frac{1}{\p}}\right\}  & t>S
\end{cases},\forall t\in\left[T\right]$, where $S\defeq\left\lceil \frac{T}{2}\right\rceil $, then $\ClipSGD$
(Algorithm \ref{alg:Clipped-SGD}) guarantees that, with probability
at least $1-2\delta$,
\[
F(\bx_{T+1})-F_{\star}=\O\left(\frac{\mu D^{2}}{T^{3}}+\frac{(\varphi_{\lip}^{2}+\ln^{2}\frac{3}{\delta})G^{2}}{\mu T^{2}}+\frac{G^{2}}{\mu T}+\frac{\sigma_{\sma}^{\frac{4}{\p}-2}\sigma_{\lar}^{4-\frac{4}{\p}}+\sigma_{\sma}^{\frac{2}{\p}}\sigma_{\lar}^{2-\frac{2}{\p}}\ln^{2-\frac{2}{\p}}\frac{3}{\delta}}{\mu T^{2-\frac{2}{\p}}}\right),
\]
in which $\varphi_{\lip}\leq\varphi_{\star}$ is a constant defined
in (\ref{eq:lip-hp-varphi}) and equals $\varphi_{\star}$ when $T=\Omega\left(\frac{G^{\p}}{\sigma_{\lar}^{\p}}\varphi_{\star}\right)$.
\end{thm}
Theorem \ref{thm:main-str-lip-hp} presents the first last-iterate
convergence rate for $\ClipSGD$ in stochastic strongly star-convex
optimization under heavy-tailed noise, which fills out the missing
pieces in the literature, as no such result was known even in the
simple case $\p=2$ and $\sigma_{\sma}=\sigma_{\lar}=\sigma$, i.e.,
classical finite-variance noise with noise level $\sigma\geq0$.

As one can check, Theorem \ref{thm:main-str-lip-hp} matches the previously
fastest rate proved for $F(\bar{\bx}_{T+1})-F_{\star}$ (where $\bar{\bx}_{T+1}$
is a weighted average of $\left\{ \bx_{t}\right\} _{t=2}^{T+1}$)
and $\mu\left\Vert \bx_{T+1}-\bx_{\star}\right\Vert ^{2}$ \citep{ICLR2026_bd5b846a}.
We remind the reader that, as mentioned earlier in Footnote \ref{fn:metric},
$F(\bx_{T+1})-F_{\star}$ always serves as a stronger convergence
criterion than $\mu\left\Vert \bx_{T+1}-\bx_{\star}\right\Vert ^{2}$,
implying that our new result is more favorable than \citep{ICLR2026_bd5b846a}.
However, similar to the case of $\mu=0$ discussed previously, there
is also a gap between the leading term $\O\left(\frac{\sigma_{\sma}^{\frac{4}{\p}-2}\sigma_{\lar}^{4-\frac{4}{\p}}+\sigma_{\sma}^{\frac{2}{\p}}\sigma_{\lar}^{2-\frac{2}{\p}}\ln^{2-\frac{2}{\p}}\frac{3}{\delta}}{\mu T^{2-\frac{2}{\p}}}\right)$
in Theorem \ref{thm:main-str-lip-hp} and the lower bound $\Omega\left(\frac{\sigma_{\sma}^{\frac{4}{\p}-2}\sigma_{\lar}^{4-\frac{4}{\p}}+\sigma_{\sma}^{2}\ln^{2-\frac{2}{\p}}\frac{1}{\delta}}{\mu T^{2-\frac{2}{\p}}}\right)$
proved by \citep{ICLR2026_bd5b846a}, closing which could be an interesting
task that we leave in the future.

One may observe that our choices of $\left\{ \eta_{t}\right\} _{t=1}^{T}$
and $\left\{ \cm_{t}\right\} _{t=1}^{T}$ have a two-stage structure,
and especially, they coincide with those chosen in \citep{ICLR2026_bd5b846a}
during the first stage, i.e., when $t\leq\left\lceil \frac{T}{2}\right\rceil $.
We note that this is done intentionally as part of our proof. Intuitively,
we use the first stage to establish convergence in squared distance,
which is finally converted into convergence in function value. The
interested reader could refer to the proof of Theorem \ref{thm:str-lip-hp}
for details.

\subsubsection{Smooth Objectives}

Lastly, we turn to the setting where $f$ is $L$-smooth on $\X$.
In this case, we require another assumption different from the previous
Assumption \ref{assu:G1} for smooth convex optimization.
\begin{assumption}
\label{assu:compact}We have access to a constant $G_{1}$ satisfying
$G_{1}\geq\max\left\{ \left\Vert \nabla F(\bx_{1})\right\Vert ,\left\Vert \nabla f(\bx_{1})\right\Vert \right\} $.
In addition, $\X$ is compact with diameter at most $2G_{1}/\mu$.
\end{assumption}
\begin{rem}
Strictly speaking, $\left\Vert \nabla F(\bx_{1})\right\Vert $ should
be replaced by $\max_{\bg\in\partial r(\bx_{1})}\left\Vert \nabla f(\bx_{1})+\bg\right\Vert $,
given that $r$ may be non-differentiable, e.g., $r(\bx)=\left\Vert \bx\right\Vert _{1}$.
However, we simply write it as $\left\Vert \nabla F(\bx_{1})\right\Vert $
to simplify the expression.
\end{rem}

\paragraph{Discussion on Assumption \ref{assu:compact}.}

One might consider the compactness requirement on $\X$ too restrictive.
However, we would like to argue that this entails no loss of generality
due to the following Fact \ref{fact:compact}.
\begin{fact}
\label{fact:compact}Under Assumptions \ref{assu:minimizer} and \ref{assu:obj}
(with $\mu>0$), we have $\bx_{\star}\in\mathbb{B}^{d}\left(\bx_{1},\left\Vert \nabla F(\bx_{1})\right\Vert /\mu\right)$.
\end{fact}
\begin{proof}
By Lemma \ref{lem:str-star-cvx}, $\frac{\mu}{2}\left\Vert \bx_{1}-\bx_{\star}\right\Vert ^{2}\leq F(\bx_{1})-F_{\star}\leq\frac{\left\Vert \nabla F(\bx_{1})\right\Vert ^{2}}{2\mu}\Rightarrow\bx_{\star}\in\B^{d}\left(\bx_{1},\left\Vert \nabla F(\bx_{1})\right\Vert /\mu\right)$.
\end{proof}

By Fact \ref{fact:compact} and the definition of $G_{1}$ in Assumption
\ref{assu:compact}, we know $\bx_{\star}\in\mathbb{B}^{d}\left(\bx_{1},G_{1}/\mu\right)$,
which further implies that 
\[
F_{\star}=F(\bx_{\star})=\inf_{\bx\in\X}F(\bx)=\inf_{\bx\in\X\cap\mathbb{B}^{d}\left(\bx_{1},G_{1}/\mu\right)}F(\bx).
\]
Hence, even if $\X$ is not compact, we can replace $\X$ with $\X\cap\mathbb{B}^{d}\left(\bx_{1},G_{1}/\mu\right)$,
which preserves the same optimal solution and restricts the problem
to a compact domain with diameter at most $2G_{1}/\mu$.

With Assumption \ref{assu:compact} in hand, we can prove the following
Theorem \ref{thm:main-str-smo-hp}.
\begin{thm}
\label{thm:main-str-smo-hp}Under Assumptions \ref{assu:minimizer},
\ref{assu:obj} (with $\mu>0$), \ref{assu:smo}, \ref{assu:heavy},
and \ref{assu:compact}, for any $T\geq2$ and $\delta\in\left(0,\frac{1}{2}\right]$,
setting $\eta_{t}=\begin{cases}
\frac{6}{\mu(t+12\kappa)} & t\leq S\\
\frac{T-t+1}{\mu T(T+\kappa)} & t>S
\end{cases},\cm_{t}=\begin{cases}
\max\left\{ 2L\left\Vert \bx_{t}-\bx_{1}\right\Vert +2G_{1},\cm_{\star}\left[T(T+\kappa)\right]^{\frac{1}{2\p}}\right\}  & t\leq S\\
\max\left\{ 2(1+2\kappa)G_{1},\cm_{\star}\left\lfloor \frac{T}{2}\right\rfloor ^{\frac{1}{\p}}\right\}  & t>S
\end{cases},\forall t\in\left[T\right]$, where $S\defeq\left\lceil \frac{T}{2}\right\rceil $ and $\kappa\defeq\frac{L}{\mu}$,
then $\ClipSGD$ (Algorithm \ref{alg:Clipped-SGD}) guarantees that,
with probability at least $1-2\delta$,
\[
F(\bx_{T+1})-F_{\star}=\O\left(\frac{(\varphi_{\smo}^{2}+\ln^{2}\frac{3}{\delta})(1+\kappa)^{2}G_{1}^{2}}{\mu T^{2}}+\frac{\sigma_{\sma}^{\frac{4}{\p}-2}\sigma_{\lar}^{4-\frac{4}{\p}}+\sigma_{\sma}^{\frac{2}{\p}}\sigma_{\lar}^{2-\frac{2}{\p}}\ln^{2-\frac{2}{\p}}\frac{3}{\delta}}{\mu\left[T(T+\kappa)\right]^{1-\frac{1}{\p}}}\right),
\]
in which $\varphi_{\smo}\leq\varphi_{\star}$ is a constant defined
in (\ref{eq:smo-hp-varphi}) and equals $\varphi_{\star}$ when $T=\Omega\left(\frac{(1+\kappa)^{\p}G_{1}^{\p}}{\sigma_{\lar}^{\p}}\varphi_{\star}\right)$.
\end{thm}
Theorem \ref{thm:main-str-smo-hp} gives not only the first (and,
in fact, the only) high-probability convergence result in terms of
$F(\bx_{T+1})-F_{\star}$ for $\ClipSGD$ in stochastic strongly star-convex
optimization, but also the first upper bound for smooth objectives
under the fine-grained Assumption \ref{assu:heavy}, exhibiting a
finer dependence on the noise parameter similar to Theorem \ref{thm:main-str-lip-hp}.

Under the classical heavy-tailed noise assumption with noise level
$\sigma$, Theorem \ref{thm:main-str-smo-hp} results in a leading
term $\O\left(\frac{\sigma^{2}\ln^{2-\frac{2}{\p}}\frac{1}{\delta}}{\mu\left[T(T+\kappa)\right]^{1-\frac{1}{\p}}}\right)$
for $F(\bx_{T+1})-F_{\star}$, which improves over the best-known
but suboptimal one $\widetilde{\O}\left(\frac{\sigma^{2}\ln^{2-\frac{2}{\p}}\frac{T}{\delta}}{\mu T^{2-\frac{2}{\p}}}\right)$
for $\mu\left\Vert \bx_{T+1}-\bx_{\star}\right\Vert ^{2}$ in the
same setting \citep{pmlr-v202-sadiev23a} (which additionally requires
$r=0$). However, we note that the higher-order term in \citep{pmlr-v202-sadiev23a}
decays exponentially and is therefore better than ours, which decays
only polynomially. It will be interesting to investigate whether a
more sophisticated design of $\left\{ \eta_{t}\right\} _{t=1}^{T}$
and $\left\{ \cm_{t}\right\} _{t=1}^{T}$, such as a three-stage schedule
\citep{ICLR2024_b1239ceb}, could refine Theorem \ref{thm:main-str-smo-hp}
and yield an exponentially decaying higher-order term.

Regarding the optimality of Theorem \ref{thm:main-str-smo-hp}, even
though no high-probability lower bounds currently exist under either
the classical or the fine-grained heavy-tailed noise assumption, we
believe that the leading term in Theorem \ref{thm:main-str-smo-hp}
is always optimal in $T$, $\delta$, and $\mu$, while the term $\sigma_{\sma}^{2/\p}\sigma_{\lar}^{2-2/\p}$
may be suboptimal under the fine-grained heavy-tailed noise assumption.

Finally, as in Theorem \ref{thm:main-str-lip-hp}, Theorem \ref{thm:main-str-smo-hp}
also adopts a two-stage structure for $\left\{ \eta_{t}\right\} _{t=1}^{T}$
and $\left\{ \cm_{t}\right\} _{t=1}^{T}$. The reason for this choice
will become clear in the proof of Theorem \ref{thm:str-smo-hp}.

\section{New Idea in Theoretical Analysis\label{sec:idea}}

In this section, we describe an interesting idea developed in our
theoretical analysis. Before beginning, we introduce some notation
that will be used later. In the following, let $\F_{t}\defeq\sigma\left(Z_{1},\mydots,Z_{t}\right)$
denote the natural filtration, and $\E_{t}\left[\cdot\right]\defeq\E\left[\cdot\mid\F_{t}\right]$
denote the conditional expectation.

Our contribution is a novel argument that goes beyond standard induction
while preserving the optimal $\poly(\ln(1/\delta))$ dependence on
$\delta$. To illustrate this point, we use smooth convex optimization
as an example. In the discussion below, we assume that both $\left\{ \eta_{t}\right\} _{t=1}^{T}$
and $\left\{ \cm_{t}\right\} _{t=1}^{T}$ are deterministic to simplify
the presentation.

To prove convergence of $\ClipSGD$ in this case (even for the average
iterate), one typically needs to show that $\left\Vert \bx_{\star}-\bx_{t+1}\right\Vert ^{2}=\O(D^{2}),\forall t\in\left[T\right]$
holds in high probability. A common strategy in prior work is induction
(e.g., \citep{pmlr-v235-gorbunov24a,pmlr-v202-sadiev23a}), which
consists of the following two parts.

\paragraph{Part I.}

One establishes a collection of inequalities of the form $\left\Vert \bx_{\star}-\bx_{t+1}\right\Vert ^{2}=\O(D^{2}+(\sum_{s=1}^{t}X_{s})^{2}+\sum_{s=1}^{t}Y_{s}^{2}),\forall t\in\left[T\right]$,
where $\left\{ X_{t}\right\} _{t=1}^{T}$ is a martingale difference
sequence (MDS) adapted to the filtration $\left\{ \F_{t}\right\} _{t=0}^{T}$,
and $\left\{ Y_{t}\right\} _{t=1}^{T}$ is a nonnegative predictable
sequence with respect to $\left\{ \F_{t}\right\} _{t=0}^{T}$. In
particular, they satisfy $\left|X_{t}\right|\leq h_{1}(\left\Vert \bx_{\star}-\bx_{t}\right\Vert ,\eta_{t},\cm_{t})$,
$\E_{t-1}\left[X_{t}^{2}\right]\leq h_{2}(\left\Vert \bx_{\star}-\bx_{t}\right\Vert ,\eta_{t},\cm_{t})$,
and $Y_{t}\leq h_{3}(\left\Vert \bx_{\star}-\bx_{t}\right\Vert ,\eta_{t},\cm_{t})$,
for three nonnegative functions $h_{1}$, $h_{2}$, and $h_{3}$,
each of which is nondecreasing in its first argument.

\paragraph{Part II.}

Let $\zeta_{t}\defeq\1\left[\left\Vert \bx_{\star}-\bx_{t}\right\Vert =\O(D)\right],\forall t\in\left[T+1\right]$
and $E_{t}\defeq\left\{ \zeta_{s}=1,\forall s\leq t\right\} ,\forall t\in\left[T+1\right]$.
Existing induction-based analysis assumes that, for a given $t\in\left[T\right]$,
$\Pr\left[E_{t}\right]\geq1-\frac{t-1}{T}\delta$. Then, on the event
$E_{t}$, we have $\left\Vert \bx_{\star}-\bx_{t+1}\right\Vert ^{2}=\O(D^{2}+(\sum_{s=1}^{t}X_{s}\zeta_{s})^{2}+\sum_{s=1}^{t}Y_{s}^{2}\zeta_{s}^{2})$.
Since $\zeta_{t}\in\F_{t-1}$, we know $\left|X_{t}\zeta_{t}\right|\leq h_{1,t}=\O(h_{1}(D,\eta_{t},\cm_{t}))$,
$\E_{t-1}\left[X_{t}^{2}\zeta_{t}^{2}\right]\leq h_{2,t}=\O(h_{2}(D,\eta_{t},\cm_{t}))$,
and $Y_{t}\zeta_{t}\leq h_{3,t}=\O(h_{3}(D,\eta_{t},\cm_{t}))$, which
further implies $\Pr\left[e_{t}\right]\geq1-\frac{\delta}{T}$ by
Freedman's inequality (cf. Lemma \ref{lem:Freedman}), where $e_{t}\defeq\left\{ \left|\sum_{s=1}^{t}X_{s}\zeta_{s}\right|\leq\O\left(\max_{s\in\left[t\right]}h_{1,s}\ln\frac{T}{\delta}+\sqrt{\sum_{s=1}^{t}h_{2,s}\ln\frac{T}{\delta}}\right)\right\} ,\forall t\in\left[T\right]$.
Note that, on the event $e_{t}\cap E_{t}$, $\left\Vert \bx_{\star}-\bx_{t+1}\right\Vert ^{2}=\O(D^{2}+\max_{s\in\left[t\right]}h_{1,s}^{2}\ln^{2}\frac{T}{\delta}+\sum_{s=1}^{t}h_{2,s}\ln\frac{T}{\delta}+\sum_{s=1}^{t}h_{3,s}^{2})\overset{\text{proper }\left\{ \eta_{t}\right\} _{t=1}^{T}\text{ and }\left\{ \cm_{t}\right\} _{t=1}^{T}}{=}\O(D^{2})\Rightarrow\zeta_{t+1}=1$.
Thus, $E_{t+1}=E_{t}\cap\left\{ \zeta_{t+1}=1\right\} \supseteq e_{t}\cap E_{t}\cap\left\{ \zeta_{t+1}=1\right\} =e_{t}\cap E_{t}$
and hence $\Pr\left[E_{t+1}\right]\geq1-\frac{t}{T}\delta$. This
completes the induction.

As can be seen, the induction argument in \textbf{Part II} inevitably
introduces a suboptimal $\poly(\ln(T/\delta))$ factor into the choices
of $\left\{ \eta_{t}\right\} _{t=1}^{T}$ and $\left\{ \cm_{t}\right\} _{t=1}^{T}$,
and consequently into the final convergence rate. However, we observe
that this issue can be solved by a new refined argument, introduced
below.

\paragraph{Refined Part II (Ours).}

We observe that $\Pr\left[E\right]\geq1-\delta$ by Freedman's inequality
(cf. Lemma \ref{lem:Freedman}), where $E\defeq\left\{ \max_{t\in\left[T\right]}\left|\sum_{s=1}^{t}X_{s}\zeta_{s}\right|\leq\O\left(\max_{t\in\left[T\right]}h_{1,t}\ln\frac{1}{\delta}+\sqrt{\sum_{t=1}^{T}h_{2,t}\ln\frac{1}{\delta}}\right)\right\} $.
Suppose we pick $\left\{ \eta_{t}\right\} _{t=1}^{T}$ and $\left\{ \cm_{t}\right\} _{t=1}^{T}$
to satisfy $\O(D^{2}+\max_{t\in\left[T\right]}h_{1,t}^{2}\ln^{2}\frac{1}{\delta}+\sum_{t=1}^{T}h_{2,t}\ln\frac{1}{\delta}+\sum_{t=1}^{T}h_{3,t}^{2})=\O(D^{2})$,
then we claim $\left\{ \zeta_{t+1}=1\right\} \supseteq E\cap E_{t},\forall t\in\left[T\right]$.
Given $t\in\left[T\right]$, under $E\cap E_{t}$, we observe
\begin{align*}
\left\Vert \bx_{\star}-\bx_{t+1}\right\Vert ^{2} & =\O(D^{2}+(\sum_{s=1}^{t}X_{s}\zeta_{s})^{2}+\sum_{s=1}^{t}Y_{s}^{2}\zeta_{s}^{2})\leq\O(D^{2}+\max_{t\in\left[T\right]}(\sum_{s=1}^{t}X_{s}\zeta_{s})^{2}+\sum_{t=1}^{T}Y_{t}^{2}\zeta_{t}^{2})\\
 & =\O(D^{2}+\max_{t\in\left[T\right]}h_{1,t}^{2}\ln^{2}\frac{1}{\delta}+\sum_{t=1}^{T}h_{2,t}\ln\frac{1}{\delta}+\sum_{t=1}^{T}h_{3,t}^{2})=\O(D^{2})\Rightarrow\zeta_{t+1}=1.
\end{align*}
Hence, our claim holds, which further implies that $E\cap E_{t+1}=E\cap E_{t}\cap\left\{ \zeta_{t+1}=1\right\} =E\cap E_{t},\forall t\in\left[T\right]\Rightarrow E_{T+1}\supseteq E\cap E_{T+1}=E\cap E_{1}=E\Rightarrow\Pr\left[E_{T+1}\right]\geq1-\delta$.
As one can see, our new argument maintains the optimal $\poly(\ln(1/\delta))$
dependence on $\delta$. For the formal application of the above idea,
we refer the reader to Lemma \ref{lem:cvx-smo-distance}.

\section{Conclusion and Limitations\label{sec:conclusion}}

This work proves the first last-iterate convergence rates optimal
in the time horizon $T$ for $\ClipSGD$ in high probability under
heavy-tailed noise across various settings, thereby closing an existing
gap in the literature. Specifically, we provide last-iterate convergence
guarantees for both convex and strongly star-convex objectives, covering
smooth and nonsmooth settings. Under the classical heavy-tailed noise
assumption (i.e., $\sigma_{\sma}=\sigma_{\lar}=\sigma$ in Assumption
\ref{assu:heavy}), these are the first minimax-optimal high-probability
last-iterate bounds for Lipschitz objectives; matching lower bounds
for smooth objectives remain open. Notably, our analysis also accommodates
composite objectives and constrained optimization, and captures a
refined dependence on noise parameters under the fine-grained Assumption
\ref{assu:heavy}.

Our work nevertheless has two major limitations. First, as in prior
work, the stepsize and clipping threshold depend on knowledge of the
tail index $\p$, which is generally unknown and difficult to estimate
reliably in practice. Removing this requirement would broaden the
applicability of our results. Second, our analysis assumes that the
time horizon $T$ is known in advance, and thus our current guarantees
do not directly extend to the anytime setting. Developing methods
that adapt to an unknown tail index while providing anytime last-iterate
guarantees is therefore an interesting direction for future research.

\section*{Acknowledgments}

The author thanks the anonymous reviewers for their valuable feedback.

\bibliographystyle{plainnat}
\bibliography{ref}

\clearpage

\appendix

\section{Proof of Lemma \ref{lem:str-star-cvx}\label{sec:proof-of-lemma}}

\begin{proof}
If $\mu=0$, the inequality holds automatically. Thus, we can focus
only on the case $\mu>0$. By Assumption \ref{assu:obj}, for any
$\bx\in\X$ and $\lambda\in\left[0,1\right]$,
\begin{equation}
F(\lambda\bx_{\star}+(1-\lambda)\bx)+\frac{\mu\lambda(1-\lambda)}{2}\left\Vert \bx_{\star}-\bx\right\Vert ^{2}\leq\lambda F_{\star}+(1-\lambda)F(\bx).\label{eq:str-star-cvx-1}
\end{equation}

Combining $F(\lambda\bx_{\star}+(1-\lambda)\bx)\geq F_{\star}$ and
(\ref{eq:str-star-cvx-1}), we obtain that
\[
F_{\star}+\frac{\mu\lambda(1-\lambda)}{2}\left\Vert \bx_{\star}-\bx\right\Vert ^{2}\leq\lambda F_{\star}+(1-\lambda)F(\bx),\forall\bx\in\X,\lambda\in\left[0,1\right].
\]
Therefore,
\begin{equation}
\frac{\mu\lambda}{2}\left\Vert \bx_{\star}-\bx\right\Vert ^{2}\leq F(\bx)-F_{\star},\forall\bx\in\X,\lambda\in\left[0,1\right)\Rightarrow\frac{\mu}{2}\left\Vert \bx-\bx_{\star}\right\Vert ^{2}\leq F(\bx)-F_{\star},\forall\bx\in\X.\label{eq:str-star-cvx-2}
\end{equation}

Combining $F(\lambda\bx_{\star}+(1-\lambda)\bx)\geq F(\bx)+\left\langle \nabla F(\bx),\lambda(\bx_{\star}-\bx)\right\rangle $,
due to the convexity of $F$, and (\ref{eq:str-star-cvx-1}), we obtain
that
\[
F(\bx)+\left\langle \nabla F(\bx),\lambda(\bx_{\star}-\bx)\right\rangle +\frac{\mu\lambda(1-\lambda)}{2}\left\Vert \bx_{\star}-\bx\right\Vert ^{2}\leq\lambda F_{\star}+(1-\lambda)F(\bx),\forall\bx\in\X,\lambda\in\left[0,1\right].
\]
Therefore,
\[
F(\bx)+\left\langle \nabla F(\bx),\bx_{\star}-\bx\right\rangle +\frac{\mu(1-\lambda)}{2}\left\Vert \bx_{\star}-\bx\right\Vert ^{2}\leq F_{\star},\forall\bx\in\X,\lambda\in\left(0,1\right],
\]
which implies that
\[
F(\bx)-F_{\star}\leq\left\langle \nabla F(\bx),\bx-\bx_{\star}\right\rangle -\frac{\mu}{2}\left\Vert \bx_{\star}-\bx\right\Vert ^{2},\forall\bx\in\X.
\]
Since
\[
\left\langle \nabla F(\bx),\bx-\bx_{\star}\right\rangle \leq\left\Vert \nabla F(\bx)\right\Vert \left\Vert \bx_{\star}-\bx\right\Vert \leq\frac{\left\Vert \nabla F(\bx)\right\Vert ^{2}}{2\mu}+\frac{\mu}{2}\left\Vert \bx_{\star}-\bx\right\Vert ^{2},
\]
we conclude that
\begin{equation}
F(\bx)-F_{\star}\leq\frac{\left\Vert \nabla F(\bx)\right\Vert ^{2}}{2\mu},\forall\bx\in\X.\label{eq:str-star-cvx-3}
\end{equation}

We combine (\ref{eq:str-star-cvx-2}) and (\ref{eq:str-star-cvx-3})
to complete the proof.
\end{proof}

\section{Missing Proofs for Theorems\label{sec:theorems}}

This section provides the missing proofs for theorems presented in
the main paper. All intermediate lemmas employed in the analysis will
be proved in Appendix \ref{sec:analysis}.

Before we begin, let us review the notation introduced previously
and also define some new ones. We first recall the term generalized
effective dimension $d_{\eff}$, defined in (\ref{eq:main-d-eff})
before, as follows:
\begin{equation}
d_{\eff}=\sigma_{\lar}^{2}/\sigma_{\sma}^{2}\in\left\{ 0\right\} \cup\left[1,\pi d/2\right],\label{eq:d-eff}
\end{equation}
in which $d_{\eff}=0$ if and only if $\sigma_{\sma}=\sigma_{\lar}=0$,
where $\sigma_{\sma}$ and $\sigma_{\lar}$ are two problem-dependent
parameters introduced in Assumption \ref{assu:heavy}. Next, a particularly
useful quantity given in (\ref{eq:main-hp-tau-star}) is
\begin{equation}
\cm_{\star}=\left(\min\left\{ \frac{\sigma_{\sma}\sigma_{\lar}^{\p-1}}{\ln\frac{3}{\delta}},\frac{\sigma_{\sma}^{2}}{\sigma_{\lar}^{2-\p}\1\left[\p<2\right]}\right\} \right)^{\frac{1}{\p}},\label{eq:hp-tau-star}
\end{equation}
where $\delta\in\left(0,1\right]$ is the failure probability and
$\p\in\left(1,2\right]$ is the tail index (see also Assumption \ref{assu:heavy}).
Another useful term mentioned before in (\ref{eq:main-hp-varphi-star})
is
\begin{equation}
\varphi_{\star}=\max\left\{ \sqrt{d_{\eff}}\ln\frac{3}{\delta},d_{\eff}\1\left[\p<2\right]\right\} .\label{eq:hp-varphi-star}
\end{equation}
In particular, by (\ref{eq:d-eff}), (\ref{eq:hp-tau-star}), and
(\ref{eq:hp-varphi-star}), the following equation holds
\begin{equation}
\varphi_{\star}=\frac{\sigma_{\lar}^{\p}}{\cm_{\star}^{\p}}.\label{eq:hp-varphi-tau-equation}
\end{equation}
Moreover, $D=\left\Vert \bx_{\star}-\bx_{1}\right\Vert $ always denotes
the distance between the optimal solution $\bx_{\star}$ and the initial
point $\bx_{1}$. Finally, we introduce two new quantities in the
following.

\begin{itemize}[leftmargin=*]

\item For Lipschitz objectives (i.e., Assumption \ref{assu:lip}),
let
\begin{equation}
\varphi_{\lip}\defeq\min_{\beta\in\left[0,\frac{1}{2}\right]}\frac{\sigma_{\lar}^{\p}T^{\beta}}{G^{\beta\p}\cm_{\star}^{(1-\beta)\p}}\overset{(\ref{eq:hp-varphi-tau-equation})}{=}\min\left\{ \varphi_{\star},\sqrt{\frac{\sigma_{\lar}^{\p}}{G^{\p}}\varphi_{\star}T}\right\} .\label{eq:lip-hp-varphi}
\end{equation}
Note that we have $\varphi_{\lip}=\varphi_{\star}$ when $T=\Omega\left(\frac{G^{\p}}{\sigma_{\lar}^{\p}}\varphi_{\star}\right)$.

\item For smooth objectives (i.e., Assumption \ref{assu:smo}), let
\begin{equation}
\varphi_{\smo}\defeq\min_{\beta\in\left[0,\frac{1}{2}\right]}\frac{\sigma_{\lar}^{\p}T^{\beta}}{\bar{G}^{\beta\p}\cm_{\star}^{(1-\beta)\p}}\overset{(\ref{eq:hp-varphi-tau-equation})}{=}\min\left\{ \varphi_{\star},\sqrt{\frac{\sigma_{\lar}^{\p}}{\bar{G}^{\p}}\varphi_{\star}T}\right\} ,\label{eq:smo-hp-varphi}
\end{equation}
in which $\bar{G}\defeq LD+G_{1}$ (resp. $\bar{G}\defeq(1+\kappa)G_{1}$)
and $G_{1}$ is defined in Assumption \ref{assu:G1} (resp. Assumption
\ref{assu:compact}) for general convex optimization (resp. strongly
star-convex optimization). Note that we have $\varphi_{\smo}=\varphi_{\star}$
when $T=\Omega\left(\frac{\bar{G}^{\p}}{\sigma_{\lar}^{\p}}\varphi_{\star}\right)$.

\end{itemize}

\subsection{General Convex Optimization}

We first consider general convex optimization, i.e., $\mu=0$ in Assumption
\ref{assu:obj}.

\subsubsection{Lipschitz Objectives}
\begin{thm}[Full statement of Theorem \ref{thm:main-cvx-lip-hp}]
\label{thm:cvx-lip-hp}Under Assumptions \ref{assu:minimizer}, \ref{assu:obj}
(with $\mu=0$), \ref{assu:lip}, and \ref{assu:heavy}, for any $T\in\N$
and $\delta\in\left(0,1\right]$, setting $\eta_{t}=\eta_{\star}(T-t+1),\cm_{t}=\max\left\{ 2G,\cm_{\star}T^{\frac{1}{\p}}\right\} ,\forall t\in\left[T\right]$,
where
\begin{equation}
\eta_{\star}=\min\left\{ \frac{D}{\left(\varphi_{\lip}+\ln\frac{3}{\delta}\right)GT},\frac{D}{GT^{\frac{3}{2}}},\frac{D}{\left(\sigma_{\sma}^{\frac{2}{\p}-1}\sigma_{\lar}^{2-\frac{2}{\p}}+\sigma_{\sma}^{\frac{1}{\p}}\sigma_{\lar}^{1-\frac{1}{\p}}\ln^{1-\frac{1}{\p}}\frac{3}{\delta}\right)T^{1+\frac{1}{\p}}}\right\} ,\label{eq:cvx-lip-hp-eta-star}
\end{equation}
and $\varphi_{\lip}\leq\varphi_{\star}$ is a constant defined in
(\ref{eq:lip-hp-varphi}), then $\ClipSGD$ (Algorithm \ref{alg:Clipped-SGD})
guarantees that, with probability at least $1-\delta$,
\[
F(\bx_{T+1})-F_{\star}=\O\left(\frac{\left(\varphi_{\lip}+\ln\frac{3}{\delta}\right)GD}{T}+\frac{GD}{\sqrt{T}}+\frac{\left(\sigma_{\sma}^{\frac{2}{\p}-1}\sigma_{\lar}^{2-\frac{2}{\p}}+\sigma_{\sma}^{\frac{1}{\p}}\sigma_{\lar}^{1-\frac{1}{\p}}\ln^{1-\frac{1}{\p}}\frac{3}{\delta}\right)D}{T^{1-\frac{1}{\p}}}\right).
\]
\end{thm}
\begin{proof}
Given $T\in\N$ and $\delta\in\left(0,1\right]$, note that $\left\{ \eta_{t}\right\} _{t=1}^{T}$
and $\left\{ \cm_{t}\right\} _{t=1}^{T}$ satisfy Conditions \ref{enu:cvx-lip-res-1}
and \ref{enu:cvx-lip-res-2} in Lemma \ref{lem:cvx-lip-res}. Therefore,
Lemmas \ref{lem:cvx-core} ($\omega=1$ now) and \ref{lem:cvx-lip-res}
together yield that, with probability at least $1-\delta$,
\begin{equation}
F(\bx_{T+1})-F_{\star}\leq\frac{v_{0}^{2}D^{2}+2\lasthc_{T}^{\cvx}+8\sum_{t=1}^{T}v_{t}^{2}\eta_{t}^{2}G^{2}}{\eta_{T}},\label{eq:cvx-lip-hp-1}
\end{equation}
where $\lasthc_{T}^{\cvx}$ is a constant satisfying
\begin{equation}
\lasthc_{T}^{\cvx}=\O\left(\max_{t\in\left[T\right]}v_{t}^{2}\eta_{t}^{2}\cm_{t}^{2}\ln^{2}\frac{3}{\delta}+\sum_{t=1}^{T}\sigma_{\lar}^{\p}v_{t}^{2}\eta_{t}^{2}\cm_{t}^{2-\p}+\left(\sum_{t=1}^{T}\frac{\sigma_{\sma}\sigma_{\lar}^{\p-1}v_{t}\eta_{t}}{\cm_{t}^{\p-1}}+\frac{\sigma_{\lar}^{\p}Gv_{t}\eta_{t}}{\cm_{t}^{\p}}\right)^{2}\right),\label{eq:cvx-lip-hp-C}
\end{equation}
and $\left\{ v_{t}^{2}\right\} _{t=0}^{T}$ is a sequence defined
in (\ref{eq:cvx-v-def}) as
\begin{eqnarray*}
v_{t}^{2}\defeq\frac{\eta_{T}}{\sum_{s=t}^{T}\eta_{s}},\forall t\in\left[T\right] & \text{and} & v_{0}^{2}\defeq v_{1}^{2}.
\end{eqnarray*}

We observe that $v_{t}^{2}=\frac{2}{(T-t+2)(T-t+1)},\forall t\in\left[T\right]$
under the stepsize $\eta_{t}=\eta(T-t+1),\forall t\in\left[T\right]$
for any $\eta>0$. As such, we can simplify (\ref{eq:cvx-lip-hp-1})
and (\ref{eq:cvx-lip-hp-C}) into
\begin{align}
F(\bx_{T+1})-F_{\star} & \leq\frac{2D^{2}}{\eta T^{2}}+\frac{2\lasthc_{T}^{\cvx}}{\eta}+16\eta G^{2}T,\label{eq:cvx-lip-hp-2}\\
\lasthc_{T}^{\cvx} & =\O\left(\eta^{2}\left[\max_{t\in\left[T\right]}\cm_{t}^{2}\ln^{2}\frac{3}{\delta}+\sum_{t=1}^{T}\sigma_{\lar}^{\p}\cm_{t}^{2-\p}+\left(\sum_{t=1}^{T}\frac{\sigma_{\sma}\sigma_{\lar}^{\p-1}}{\cm_{t}^{\p-1}}+\frac{\sigma_{\lar}^{\p}G}{\cm_{t}^{\p}}\right)^{2}\right]\right).\nonumber 
\end{align}
Furthermore, when $\cm_{t}=\max\left\{ 2G,\cm\right\} ,\forall t\in\left[T\right]$
for $\cm\geq0$, $\lasthc_{T}^{\cvx}$ can be further reduced into
\begin{align*}
\lasthc_{T}^{\cvx} & =\O\left(\eta^{2}\left[\left(\frac{\sigma_{\lar}^{2\p}T^{2}}{\max\left\{ 2G,\cm\right\} ^{2\p}}+\ln^{2}\frac{3}{\delta}\right)G^{2}+\cm^{2}\ln^{2}\frac{3}{\delta}+\sigma_{\lar}^{\p}\cm^{2-\p}T+\frac{\sigma_{\sma}^{2}\sigma_{\lar}^{2\p-2}}{\cm^{2\p-2}}T^{2}\right]\right)\\
 & =\O\left(\eta^{2}\left[\left(\min_{\beta\in\left[0,\frac{1}{2}\right]}\frac{\sigma_{\lar}^{2\p}T^{2}}{G^{2\beta\p}\cm^{2(1-\beta)\p}}+\ln^{2}\frac{3}{\delta}\right)G^{2}+\cm^{2}\ln^{2}\frac{3}{\delta}+\sigma_{\lar}^{\p}\cm^{2-\p}T+\frac{\sigma_{\sma}^{2}\sigma_{\lar}^{2\p-2}}{\cm^{2\p-2}}T^{2}\right]\right).
\end{align*}
One more step, we replace $\cm$ with $\cm T^{\frac{1}{\p}}$ to have
\begin{equation}
\lasthc_{T}^{\cvx}=\O\left(\eta^{2}\left[\left(\min_{\beta\in\left[0,\frac{1}{2}\right]}\frac{\sigma_{\lar}^{2\p}T^{2\beta}}{G^{2\beta\p}\cm^{2(1-\beta)\p}}+\ln^{2}\frac{3}{\delta}\right)G^{2}+\left(\cm^{2}\ln^{2}\frac{3}{\delta}+\sigma_{\lar}^{\p}\cm^{2-\p}+\frac{\sigma_{\sma}^{2}\sigma_{\lar}^{2\p-2}}{\cm^{2\p-2}}\right)T^{\frac{2}{\p}}\right]\right).\label{eq:cvx-lip-hp-C-simple}
\end{equation}

Finally, we plug in $\cm=\cm_{\star}$ (see (\ref{eq:hp-tau-star}))
into (\ref{eq:cvx-lip-hp-C-simple}) and use the definition of $\varphi_{\lip}$
(see (\ref{eq:lip-hp-varphi})) to have
\[
\lasthc_{T}^{\cvx}=\O\left(\eta^{2}\left[\left(\varphi_{\lip}^{2}+\ln^{2}\frac{3}{\delta}\right)G^{2}+\left(\sigma_{\sma}^{\frac{4}{\p}-2}\sigma_{\lar}^{4-\frac{4}{\p}}+\sigma_{\sma}^{\frac{2}{\p}}\sigma_{\lar}^{2-\frac{2}{\p}}\ln^{2-\frac{2}{\p}}\frac{3}{\delta}\right)T^{\frac{2}{\p}}\right]\right).
\]
Combine the above result and (\ref{eq:cvx-lip-hp-2}) to obtain
\begin{align}
 & F(\bx_{T+1})-F_{\star}\nonumber \\
= & \O\left(\frac{D^{2}}{\eta T^{2}}+\eta\left(\varphi_{\lip}^{2}+\ln^{2}\frac{3}{\delta}\right)G^{2}+\eta G^{2}T+\eta\left(\sigma_{\sma}^{\frac{4}{\p}-2}\sigma_{\lar}^{4-\frac{4}{\p}}+\sigma_{\sma}^{\frac{2}{\p}}\sigma_{\lar}^{2-\frac{2}{\p}}\ln^{2-\frac{2}{\p}}\frac{3}{\delta}\right)T^{\frac{2}{\p}}\right).\label{eq:cvx-lip-hp-3}
\end{align}
Plug in $\eta=\eta_{\star}$ (see (\ref{eq:cvx-lip-hp-eta-star}))
to conclude that
\[
F(\bx_{T+1})-F_{\star}=\O\left(\frac{\left(\varphi_{\lip}+\ln\frac{3}{\delta}\right)GD}{T}+\frac{GD}{\sqrt{T}}+\frac{\left(\sigma_{\sma}^{\frac{2}{\p}-1}\sigma_{\lar}^{2-\frac{2}{\p}}+\sigma_{\sma}^{\frac{1}{\p}}\sigma_{\lar}^{1-\frac{1}{\p}}\ln^{1-\frac{1}{\p}}\frac{3}{\delta}\right)D}{T^{1-\frac{1}{\p}}}\right).
\]
\end{proof}

\subsubsection{Smooth Objectives}
\begin{thm}[Full statement of Theorem \ref{thm:main-cvx-smo-hp}]
\label{thm:cvx-smo-hp}Under Assumptions \ref{assu:minimizer}, \ref{assu:obj}
(with $\mu=0$), \ref{assu:smo}, \ref{assu:heavy}, and \ref{assu:G1},
for any $T\in\N$ and $\delta\in\left(0,\frac{1}{2}\right]$, setting
$\eta_{t}=\eta_{\star}(T-t+1),\cm_{t}=\max\left\{ 2L\left\Vert \bx_{t}-\bx_{1}\right\Vert +2G_{1},\cm_{\star}T^{\frac{1}{\p}}\right\} ,\forall t\in\left[T\right]$,
where
\begin{equation}
\eta_{\star}=\min\left\{ \frac{D}{8\sqrt{\constant}\left(\varphi_{\smo}+\ln\frac{3}{\delta}\right)\left(LD+G_{1}\right)T},\frac{D}{\left(\sigma_{\sma}^{\frac{2}{\p}-1}\sigma_{\lar}^{2-\frac{2}{\p}}+\sigma_{\sma}^{\frac{1}{\p}}\sigma_{\lar}^{1-\frac{1}{\p}}\ln^{1-\frac{1}{\p}}\frac{3}{\delta}\right)T^{1+\frac{1}{\p}}}\right\} ,\label{eq:cvx-smo-hp-eta-star}
\end{equation}
$\constant\geq1$ is a numerical constant, and $\varphi_{\smo}\leq\varphi_{\star}$
is a constant defined in (\ref{eq:smo-hp-varphi}), then $\ClipSGD$
(Algorithm \ref{alg:Clipped-SGD}) guarantees that, with probability
at least $1-2\delta$,
\[
F(\bx_{T+1})-F_{\star}=\O\left(\frac{\left(\varphi_{\smo}+\ln\frac{3}{\delta}\right)\left(LD+G_{1}\right)D}{T}+\frac{\left(\sigma_{\sma}^{\frac{2}{\p}-1}\sigma_{\lar}^{2-\frac{2}{\p}}+\sigma_{\sma}^{\frac{1}{\p}}\sigma_{\lar}^{1-\frac{1}{\p}}\ln^{1-\frac{1}{\p}}\frac{3}{\delta}\right)D}{T^{1-\frac{1}{\p}}}\right).
\]
\end{thm}
\begin{proof}
Given $T\in\N$ and $\delta\in\left(0,\frac{1}{2}\right]$, note that
$\left\{ \eta_{t}\right\} _{t=1}^{T}$ and $\left\{ \cm_{t}\right\} _{t=1}^{T}$
satisfy Conditions \ref{enu:cvx-smo-distance-1} and \ref{enu:cvx-smo-distance-2}
in Lemma \ref{lem:cvx-smo-distance}. Therefore, by Lemma \ref{lem:cvx-smo-distance},
for any $\bar{D}>0$ such that (where $\cm=\cm_{\star}T^{\frac{1}{\p}}$,
$\bar{G}=L(\bar{D}+D)+G_{1}$, $\rho=\max\left\{ 2\bar{G},\cm\right\} $,
and $\constant\geq1$ is a numerical constant)
\begin{equation}
2D^{2}+4\constant\left[\max_{t\in\left[T\right]}\eta_{t}^{2}\rho^{2}\ln^{2}\frac{3}{\delta}+\sum_{t=1}^{T}\sigma_{\lar}^{\p}\eta_{t}^{2}\cm^{2-\p}+\left(\sum_{t=1}^{T}\frac{\sigma_{\sma}\sigma_{\lar}^{\p-1}\eta_{t}}{\cm^{\p-1}}+\min_{\beta\in\left[0,\frac{1}{2}\right]}\frac{\sigma_{\lar}^{\p}\bar{G}^{1-\beta\p}\eta_{t}}{\cm^{(1-\beta)\p}}\right)^{2}\right]\leq\bar{D}^{2},\label{eq:cvx-smo-hp-Dbar-ineq}
\end{equation}
there is
\[
\Pr\left[\left\Vert \bx_{\star}-\bx_{t}\right\Vert \leq\bar{D},\forall t\in\left[T+1\right]\right]\geq1-\delta.
\]

Now, we claim that
\begin{equation}
\bar{D}\defeq\sqrt{5+32\constant}D=\Theta(D)\label{eq:cvx-smo-hp-Dbar-value}
\end{equation}
 satisfies (\ref{eq:cvx-smo-hp-Dbar-ineq}) under the current choices
of $\left\{ \eta_{t}\right\} _{t=1}^{T}$ and $\left\{ \cm_{t}\right\} _{t=1}^{T}$.
To prove the claim, we first upper bound
\begin{align}
 & \max_{t\in\left[T\right]}\eta_{t}^{2}\rho^{2}\ln^{2}\frac{3}{\delta}+\sum_{t=1}^{T}\sigma_{\lar}^{\p}\eta_{t}^{2}\cm^{2-\p}+\left(\sum_{t=1}^{T}\frac{\sigma_{\sma}\sigma_{\lar}^{\p-1}\eta_{t}}{\cm^{\p-1}}+\min_{\beta\in\left[0,\frac{1}{2}\right]}\frac{\sigma_{\lar}^{\p}\bar{G}^{1-\beta\p}\eta_{t}}{\cm^{(1-\beta)\p}}\right)^{2}\nonumber \\
\overset{(a)}{\leq} & \eta_{\star}^{2}T^{2}\left(\rho^{2}\ln^{2}\frac{3}{\delta}+\sigma_{\lar}^{\p}\cm^{2-\p}T+2\left(\frac{\sigma_{\sma}^{2}\sigma_{\lar}^{2\p-2}}{\cm^{2\p-2}}T^{2}+\min_{\beta\in\left[0,\frac{1}{2}\right]}\frac{\sigma_{\lar}^{2\p}T^{2}}{\bar{G}^{2\beta\p}\cm^{2(1-\beta)\p}}\bar{G}^{2}\right)\right)\nonumber \\
\overset{(b)}{\leq} & \eta_{\star}^{2}T^{2}\left(\left(\min_{\beta\in\left[0,\frac{1}{2}\right]}\frac{2\sigma_{\lar}^{2\p}T^{2}}{\bar{G}^{2\beta\p}\cm^{2(1-\beta)\p}}+4\ln^{2}\frac{3}{\delta}\right)\bar{G}^{2}+\cm^{2}\ln^{2}\frac{3}{\delta}+\sigma_{\lar}^{\p}\cm^{2-\p}T+2\frac{\sigma_{\sma}^{2}\sigma_{\lar}^{2\p-2}}{\cm^{2\p-2}}T^{2}\right)\nonumber \\
\overset{(c)}{=} & \eta_{\star}^{2}T^{2}\left(\left(\min_{\beta\in\left[0,\frac{1}{2}\right]}\frac{2\sigma_{\lar}^{2\p}T^{2\beta}}{\bar{G}^{2\beta\p}\cm_{\star}^{2(1-\beta)\p}}+4\ln^{2}\frac{3}{\delta}\right)\bar{G}^{2}+\left(\cm_{\star}^{2}\ln^{2}\frac{3}{\delta}+\sigma_{\lar}^{\p}\cm_{\star}^{2-\p}+2\frac{\sigma_{\sma}^{2}\sigma_{\lar}^{2\p-2}}{\cm_{\star}^{2\p-2}}\right)T^{\frac{2}{\p}}\right)\nonumber \\
\overset{(d)}{\leq} & \eta_{\star}^{2}T^{2}\left(\left(2\varphi_{\smo}^{2}+4\ln^{2}\frac{3}{\delta}\right)\left(L(\bar{D}+D)+G_{1}\right)^{2}+\left(\cm_{\star}^{2}\ln^{2}\frac{3}{\delta}+\sigma_{\lar}^{\p}\cm_{\star}^{2-\p}+2\frac{\sigma_{\sma}^{2}\sigma_{\lar}^{2\p-2}}{\cm_{\star}^{2\p-2}}\right)T^{\frac{2}{\p}}\right)\nonumber \\
\overset{(e)}{\leq} & \frac{\bar{D}^{2}+D^{2}}{8\constant}+4D^{2},\label{eq:cvx-smo-hp-Dbar-simple}
\end{align}
where $(a)$ is by $\eta_{t}=\eta_{\star}(T-t+1)\leq\eta_{\star}T,\forall t\in\left[T\right]$
and $(x+y)^{2}\leq2x^{2}+2y^{2}$, $(b)$ is due to $\rho=\max\left\{ 2\bar{G},\cm\right\} \Rightarrow\rho^{2}\leq4\bar{G}^{2}+\cm^{2}$,
$(c)$ holds after plugging in $\cm=\cm_{\star}T^{\frac{1}{\p}}$,
$(d)$ is due to $\bar{G}=L(\bar{D}+D)+G_{1}\overset{(\ref{eq:cvx-smo-hp-Dbar-value})}{=}\Theta(LD+G_{1})$
and the definition of $\varphi_{\smo}$ (see (\ref{eq:smo-hp-varphi})),
and $(e)$ is true by noticing that
\begin{align*}
\eta_{\star}^{2}T^{2}\cdot\left(2\varphi_{\smo}^{2}+4\ln^{2}\frac{3}{\delta}\right)\left(L(\bar{D}+D)+G_{1}\right)^{2} & \overset{(\ref{eq:cvx-smo-hp-eta-star})}{\leq}\frac{\bar{D}^{2}+D^{2}}{8\constant},\\
\eta_{\star}^{2}T^{2}\cdot\left(\cm_{\star}^{2}\ln^{2}\frac{3}{\delta}+\sigma_{\lar}^{\p}\cm_{\star}^{2-\p}+2\frac{\sigma_{\sma}^{2}\sigma_{\lar}^{2\p-2}}{\cm_{\star}^{2\p-2}}\right)T^{\frac{2}{\p}} & \overset{(\ref{eq:hp-tau-star}),(\ref{eq:cvx-smo-hp-eta-star})}{\leq}4D^{2}.
\end{align*}
Combining (\ref{eq:cvx-smo-hp-Dbar-ineq}) and (\ref{eq:cvx-smo-hp-Dbar-simple}),
we only need to show that $\bar{D}=\sqrt{5+32\constant}D$ satisfies
the inequality $2D^{2}+4\constant\left(\frac{\bar{D}^{2}+D^{2}}{8\constant}+4D^{2}\right)\leq\bar{D}^{2}$,
which holds after a direct verification. Hence, we obtain that
\begin{equation}
\Pr\left[\left\Vert \bx_{\star}-\bx_{t}\right\Vert \leq\bar{D},\forall t\in\left[T+1\right]\right]\geq1-\delta.\label{eq:cvx-smo-hp-distance}
\end{equation}

Next, Lemma \ref{lem:cvx-core} ($\omega=0$ now) yields that
\begin{equation}
F(\bx_{T+1})-F_{\star}\leq\frac{v_{0}^{2}D^{2}+2\lasthres_{T}^{\cvx}}{\eta_{T}},\label{eq:cvx-smo-hp-1}
\end{equation}
where
\[
\lasthres_{T}^{\cvx}=8\max_{t\in\left[T\right]}\left(\sum_{s=1}^{t}M_{s}\right)^{2}+2\sum_{t=1}^{T}v_{t}^{2}\eta_{t}^{2}\left\Vert \bd_{t}^{\ru}\right\Vert ^{2}+4\left(\sum_{t=1}^{T}v_{t}\eta_{t}\left\Vert \bd_{t}^{\rb}\right\Vert \right)^{2},
\]
and $\left\{ v_{t}^{2}\right\} _{t=0}^{T}$ is a sequence defined
in (\ref{eq:cvx-v-def}) as
\begin{eqnarray*}
v_{t}^{2}\defeq\frac{\eta_{T}}{\sum_{s=t}^{T}\eta_{s}},\forall t\in\left[T\right] & \text{and} & v_{0}^{2}\defeq v_{1}^{2}.
\end{eqnarray*}
Observe that
\[
\Pr\left[\lasthres_{T}^{\cvx}=\bar{\lasthres}_{T}^{\cvx}\right]\geq\Pr\left[\zeta_{t}=1,\forall t\in\left[T+1\right]\right]\overset{(\ref{eq:cvx-smo-hp-distance})}{\geq}1-\delta,
\]
in which
\begin{align*}
\bar{\lasthres}_{T}^{\cvx} & =8\max_{t\in\left[T\right]}\left(\sum_{s=1}^{t}M_{s}\zeta_{s}\right)^{2}+2\sum_{t=1}^{T}v_{t}^{2}\eta_{t}^{2}\left\Vert \bd_{t}^{\ru}\right\Vert ^{2}\zeta_{t}^{2}+4\left(\sum_{t=1}^{T}v_{t}\eta_{t}\left\Vert \bd_{t}^{\rb}\right\Vert \zeta_{t}\right)^{2},\\
\zeta_{t} & =\1\left[\left\Vert \bx_{\star}-\bx_{t}\right\Vert \leq\bar{D}\right],\forall t\in\left[T\right],
\end{align*}
are introduced in Lemma \ref{lem:cvx-smo-res}. Thus, with probability
at least $1-\delta$,
\begin{equation}
F(\bx_{T+1})-F_{\star}\overset{(\ref{eq:cvx-smo-hp-1})}{\leq}\frac{v_{0}^{2}D^{2}+2\lasthres_{T}^{\cvx}}{\eta_{T}}=\frac{v_{0}^{2}D^{2}+2\bar{\lasthres}_{T}^{\cvx}}{\eta_{T}}.\label{eq:cvx-smo-hp-2}
\end{equation}
Note that $\left\{ \eta_{t}\right\} _{t=1}^{T}$ and $\left\{ \cm_{t}\right\} _{t=1}^{T}$
also satisfy Conditions \ref{enu:cvx-smo-res-1} and \ref{enu:cvx-smo-res-2}
in Lemma \ref{lem:cvx-smo-res}, which implies that $\Pr\left[\bar{\lasthres}_{T}^{\cvx}\leq\bar{\lasthc}_{T}^{\cvx}\right]\geq1-\delta$,
where $\bar{\lasthc}_{T}^{\cvx}$ is a constant satisfying
\begin{equation}
\bar{\lasthc}_{T}^{\cvx}=\O\left(\max_{t\in\left[T\right]}v_{t}^{2}\eta_{t}^{2}\rho^{2}\ln^{2}\frac{3}{\delta}+\sum_{t=1}^{T}\sigma_{\lar}^{\p}v_{t}^{2}\eta_{t}^{2}\cm^{2-\p}+\left(\sum_{t=1}^{T}\frac{\sigma_{\sma}\sigma_{\lar}^{\p-1}v_{t}\eta_{t}}{\cm^{\p-1}}+\min_{\beta\in\left[0,\frac{1}{2}\right]}\frac{\sigma_{\lar}^{\p}\bar{G}^{1-\beta\p}v_{t}\eta_{t}}{\cm^{(1-\beta)\p}}\right)^{2}\right),\label{eq:cvx-smo-hp-C}
\end{equation}
in which $\cm=\cm_{\star}T^{\frac{1}{\p}}$, $\bar{G}=L(\bar{D}+D)+G_{1}\overset{(\ref{eq:cvx-smo-hp-Dbar-value})}{=}\Theta(LD+G_{1})$,
and $\rho=\max\left\{ 2\bar{G},\cm\right\} $. As such, we find that,
with probability at least $1-2\delta$,
\begin{equation}
F(\bx_{T+1})-F_{\star}\leq\frac{v_{0}^{2}D^{2}+2\bar{\lasthc}_{T}^{\cvx}}{\eta_{T}}.\label{eq:cvx-smo-hp-3}
\end{equation}

Finally, by following essentially the same steps in proving Theorem
\ref{thm:cvx-lip-hp}, we conclude that with probability at least
$1-2\delta$,
\[
F(\bx_{T+1})-F_{\star}=\O\left(\frac{\left(\varphi_{\smo}+\ln\frac{3}{\delta}\right)\left(LD+G_{1}\right)D}{T}+\frac{\left(\sigma_{\sma}^{\frac{2}{\p}-1}\sigma_{\lar}^{2-\frac{2}{\p}}+\sigma_{\sma}^{\frac{1}{\p}}\sigma_{\lar}^{1-\frac{1}{\p}}\ln^{1-\frac{1}{\p}}\frac{3}{\delta}\right)D}{T^{1-\frac{1}{\p}}}\right).
\]
\end{proof}

\subsection{Strongly Star-Convex Optimization}

We turn our attention to strongly star-convex optimization, i.e.,
$\mu>0$ in Assumption \ref{assu:obj}.

\subsubsection{Lipschitz Objectives}
\begin{thm}[Restatement of Theorem \ref{thm:main-str-lip-hp}]
\label{thm:str-lip-hp}Under Assumptions \ref{assu:minimizer}, \ref{assu:obj}
(with $\mu>0$), \ref{assu:lip}, and \ref{assu:heavy}, for any $T\geq2$
and $\delta\in\left(0,\frac{1}{2}\right]$, setting $\eta_{t}=\begin{cases}
\frac{6}{\mu t} & t\leq S\\
\frac{T-t+1}{\mu T^{2}} & t>S
\end{cases},\cm_{t}=\begin{cases}
\max\left\{ 2G,\cm_{\star}t^{\frac{1}{\p}}\right\}  & t\leq S\\
\max\left\{ 2G,\cm_{\star}\left\lfloor \frac{T}{2}\right\rfloor ^{\frac{1}{\p}}\right\}  & t>S
\end{cases},\forall t\in\left[T\right]$, where $S\defeq\left\lceil \frac{T}{2}\right\rceil $, then $\ClipSGD$
(Algorithm \ref{alg:Clipped-SGD}) guarantees that, with probability
at least $1-2\delta$,
\[
F(\bx_{T+1})-F_{\star}=\O\left(\frac{\mu D^{2}}{T^{3}}+\frac{\left(\varphi_{\lip}^{2}+\ln^{2}\frac{3}{\delta}\right)G^{2}}{\mu T^{2}}+\frac{G^{2}}{\mu T}+\frac{\sigma_{\sma}^{\frac{4}{\p}-2}\sigma_{\lar}^{4-\frac{4}{\p}}+\sigma_{\sma}^{\frac{2}{\p}}\sigma_{\lar}^{2-\frac{2}{\p}}\ln^{2-\frac{2}{\p}}\frac{3}{\delta}}{\mu T^{2-\frac{2}{\p}}}\right),
\]
in which $\varphi_{\lip}\leq\varphi_{\star}$ is a constant defined
in (\ref{eq:lip-hp-varphi}).
\end{thm}
\begin{proof}
Given $T\in\N$ and $\delta\in\left(0,\frac{1}{2}\right]$, note that
$\left\{ \eta_{t}\right\} _{t=1}^{S}$ satisfies Lemma \ref{lem:str-core}
with $\eta=6$ under the time horizon $S$, and $\left\{ \eta_{t}\right\} _{t=1}^{S}$
and $\left\{ \cm_{t}\right\} _{t=1}^{S}$ satisfy Conditions \ref{enu:str-lip-res-1}
and \ref{enu:str-lip-res-2} in Lemma \ref{lem:str-lip-res}. Therefore,
Lemmas \ref{lem:str-core} ($\eta=6$ and $\omega=1$ now) and \ref{lem:str-lip-res}
($\eta=6$ now) together yield that, with probability at least $1-\delta$,
\begin{equation}
\left\Vert \bx_{\star}-\bx_{S+1}\right\Vert ^{2}\leq\frac{8D^{2}+4\hc_{S}^{\str}+16\sum_{t=1}^{S}\gamma_{t}\eta_{t}^{2}G^{2}}{\gamma_{S}},\label{eq:str-lip-hp-1}
\end{equation}
where $\hc_{S}^{\str}$ is a constant satisfying
\begin{align}
\hc_{S}^{\str}= & \O\left(\max_{t\in\left[S\right]}\gamma_{t}\eta_{t}^{2}\cm_{t}^{2}\ln^{2}\frac{3}{\delta}+\sum_{t=1}^{S}\sigma_{\lar}^{\p}\gamma_{t}\eta_{t}^{2}\cm_{t}^{2-\p}+\sum_{t=1}^{S}\left(\frac{\sigma_{\sma}^{\p}\gamma_{t}\eta_{t}^{2}}{\cm_{t}^{\p-2}}+\frac{\sigma_{\lar}^{\p}G^{2}\gamma_{t}\eta_{t}^{2}}{\cm_{t}^{\p}}\right)\ln\frac{3}{\delta}\right.\nonumber \\
 & \left.\quad+\sum_{t=1}^{S}\left(\frac{\sigma_{\sma}^{2}\sigma_{\lar}^{2\p-2}\gamma_{t}\eta_{t}}{\cm_{t}^{2\p-2}}+\frac{\sigma_{\lar}^{2\p}G^{2}\gamma_{t}\eta_{t}}{\cm_{t}^{2\p}}\right)\frac{1}{\mu}\right),\label{eq:str-lip-hp-C}
\end{align}
and $\left\{ \gamma_{t}\right\} _{t=1}^{S}$ is a sequence defined
in (\ref{eq:str-gamma-def}).

Next, we compute
\[
\gamma_{t}\overset{(\ref{eq:str-gamma-def})}{=}\prod_{s=2}^{t}\frac{1+\mu\eta_{s-1}}{1+\mu\eta_{s}/2}\overset{\eta_{s}=\frac{6}{\mu s}}{=}\prod_{s=2}^{t}\frac{s(s+5)}{(s-1)(s+3)}=\frac{t(t+4)(t+5)}{30},\forall t\in\left[S\right].
\]
By the above equation and (\ref{eq:str-lip-hp-1}), we have
\begin{equation}
\left\Vert \bx_{\star}-\bx_{S+1}\right\Vert ^{2}=\O\left(\frac{D^{2}+\hc_{S}^{\str}+\frac{G^{2}}{\mu^{2}}\sum_{t=1}^{S}t}{S^{3}}\right)\overset{S=\left\lceil \frac{T}{2}\right\rceil }{=}\O\left(\frac{D^{2}+\hc_{S}^{\str}}{T^{3}}+\frac{G^{2}}{\mu^{2}T}\right).\label{eq:str-lip-hp-2}
\end{equation}
By similar steps in proving Theorem 14 of \citep{ICLR2026_bd5b846a},
we can show
\begin{align}
\hc_{S}^{\str} & =\O\left(\frac{\left(\min_{\beta\in\left[0,\frac{1}{2}\right]}\frac{\sigma_{\lar}^{2\p}S^{2\beta}}{G^{2\beta\p}\cm_{\star}^{2(1-\beta)\p}}+\ln^{2}\frac{3}{\delta}\right)G^{2}}{\mu^{2}}S+\frac{\cm_{\star}^{2}\ln^{2}\frac{3}{\delta}+\sigma_{\lar}^{\p}\cm_{\star}^{2-\p}+\frac{\sigma_{\sma}^{2}\sigma_{\lar}^{2\p-2}}{\cm_{\star}^{2\p-2}}}{\mu^{2}}S^{1+\frac{2}{\p}}\right)\nonumber \\
 & \overset{S=\left\lceil \frac{T}{2}\right\rceil ,(\ref{eq:hp-tau-star}),(\ref{eq:lip-hp-varphi})}{=}\O\left(\frac{\left(\varphi_{\lip}^{2}+\ln^{2}\frac{3}{\delta}\right)G^{2}}{\mu^{2}}T+\frac{\sigma_{\sma}^{\frac{4}{\p}-2}\sigma_{\lar}^{4-\frac{4}{\p}}+\sigma_{\sma}^{\frac{2}{\p}}\sigma_{\lar}^{2-\frac{2}{\p}}\ln^{2-\frac{2}{\p}}\frac{3}{\delta}}{\mu^{2}}T^{1+\frac{2}{\p}}\right).\label{eq:str-lip-hp-C-simple}
\end{align}
Hence, we obtain that, with probability at least $1-\delta$,
\begin{equation}
\left\Vert \bx_{\star}-\bx_{S+1}\right\Vert ^{2}\overset{(\ref{eq:str-lip-hp-2}),(\ref{eq:str-lip-hp-C-simple})}{=}\O\left(\frac{D^{2}}{T^{3}}+\frac{\left(\varphi_{\lip}^{2}+\ln^{2}\frac{3}{\delta}\right)G^{2}}{\mu^{2}T^{2}}+\frac{G^{2}}{\mu^{2}T}+\frac{\sigma_{\sma}^{\frac{4}{\p}-2}\sigma_{\lar}^{4-\frac{4}{\p}}+\sigma_{\sma}^{\frac{2}{\p}}\sigma_{\lar}^{2-\frac{2}{\p}}\ln^{2-\frac{2}{\p}}\frac{3}{\delta}}{\mu^{2}T^{2-\frac{2}{\p}}}\right).\label{eq:str-lip-hp-3}
\end{equation}

Finally, when $T\geq2$, we observe that $\bx_{T+1}$ can be viewed
as the output of Algorithm \ref{alg:Clipped-SGD} after $T-S=\left\lfloor \frac{T}{2}\right\rfloor $
iterations with the initial point $\bx_{S+1}$. Therefore, let $\eta_{\star}\defeq\frac{1}{\mu T^{2}}$,
by our choices of $\left\{ \eta_{t}\right\} _{t=S+1}^{T}$ and $\left\{ \cm_{t}\right\} _{t=S+1}^{T}$
and the proof of Theorem \ref{thm:cvx-lip-hp}, we have, with probability
at least $1-\delta$,
\begin{align}
F(\bx_{T+1})-F_{\star}\overset{(\ref{eq:cvx-lip-hp-3})}{=} & \O\left(\frac{\left\Vert \bx_{\star}-\bx_{S+1}\right\Vert ^{2}}{\eta_{\star}\left\lfloor \frac{T}{2}\right\rfloor ^{2}}+\eta_{\star}\left(\varphi_{\lip}^{2}+\ln^{2}\frac{3}{\delta}\right)G^{2}+\eta_{\star}G^{2}\left\lfloor \frac{T}{2}\right\rfloor \right.\nonumber \\
 & \left.\quad+\eta_{\star}\left(\sigma_{\sma}^{\frac{4}{\p}-2}\sigma_{\lar}^{4-\frac{4}{\p}}+\sigma_{\sma}^{\frac{2}{\p}}\sigma_{\lar}^{2-\frac{2}{\p}}\ln^{2-\frac{2}{\p}}\frac{3}{\delta}\right)\left\lfloor \frac{T}{2}\right\rfloor ^{\frac{2}{\p}}\right)\nonumber \\
\overset{\eta_{\star}=\frac{1}{\mu T^{2}}}{=} & \O\left(\mu\left\Vert \bx_{\star}-\bx_{S+1}\right\Vert ^{2}+\frac{\left(\varphi_{\lip}^{2}+\ln^{2}\frac{3}{\delta}\right)G^{2}}{\mu T^{2}}+\frac{G^{2}}{\mu T}\right.\nonumber \\
 & \left.\quad+\frac{\sigma_{\sma}^{\frac{4}{\p}-2}\sigma_{\lar}^{4-\frac{4}{\p}}+\sigma_{\sma}^{\frac{2}{\p}}\sigma_{\lar}^{2-\frac{2}{\p}}\ln^{2-\frac{2}{\p}}\frac{3}{\delta}}{\mu T^{2-\frac{2}{\p}}}\right).\label{eq:str-lip-hp-4}
\end{align}
By combining (\ref{eq:str-lip-hp-3}) and (\ref{eq:str-lip-hp-4}),
we obtain the desired result.
\end{proof}

\subsubsection{Smooth Objectives}
\begin{thm}[Restatement of Theorem \ref{thm:main-str-smo-hp}]
\label{thm:str-smo-hp}Under Assumptions \ref{assu:minimizer}, \ref{assu:obj}
(with $\mu>0$), \ref{assu:smo}, \ref{assu:heavy}, and \ref{assu:compact},
for any $T\geq2$ and $\delta\in\left(0,\frac{1}{2}\right]$, setting
$\eta_{t}=\begin{cases}
\frac{6}{\mu(t+12\kappa)} & t\leq S\\
\frac{T-t+1}{\mu T(T+\kappa)} & t>S
\end{cases},\cm_{t}=\begin{cases}
\max\left\{ 2L\left\Vert \bx_{t}-\bx_{1}\right\Vert +2G_{1},\cm_{\star}\left[T(T+\kappa)\right]^{\frac{1}{2\p}}\right\}  & t\leq S\\
\max\left\{ 2(1+2\kappa)G_{1},\cm_{\star}\left\lfloor \frac{T}{2}\right\rfloor ^{\frac{1}{\p}}\right\}  & t>S
\end{cases},\forall t\in\left[T\right]$, where $S\defeq\left\lceil \frac{T}{2}\right\rceil $ and $\kappa\defeq\frac{L}{\mu}$,
then $\ClipSGD$ (Algorithm \ref{alg:Clipped-SGD}) guarantees that,
with probability at least $1-2\delta$,
\[
F(\bx_{T+1})-F_{\star}=\O\left(\frac{\left(\varphi_{\smo}^{2}+\ln^{2}\frac{3}{\delta}\right)\left(1+\kappa\right)^{2}G_{1}^{2}}{\mu T^{2}}+\frac{\sigma_{\sma}^{\frac{4}{\p}-2}\sigma_{\lar}^{4-\frac{4}{\p}}+\sigma_{\sma}^{\frac{2}{\p}}\sigma_{\lar}^{2-\frac{2}{\p}}\ln^{2-\frac{2}{\p}}\frac{3}{\delta}}{\mu\left[T(T+\kappa)\right]^{1-\frac{1}{\p}}}\right),
\]
in which $\varphi_{\smo}\leq\varphi_{\star}$ is a constant defined
in (\ref{eq:smo-hp-varphi}).
\end{thm}
\begin{proof}
Given $T\in\N$ and $\delta\in\left(0,\frac{1}{2}\right]$, note that
$\left\{ \eta_{t}\right\} _{t=1}^{S}$ satisfies Lemma \ref{lem:str-core}
with $\eta=6$ under the time horizon $S$, $\left\{ \eta_{t}\right\} _{t=1}^{S}$
and $\left\{ \cm_{t}\right\} _{t=1}^{S}$ satisfy Conditions \ref{enu:str-smo-res-1}
and \ref{enu:str-smo-res-2} in Lemma \ref{lem:str-smo-res}, and
$\hres_{S}^{\str}$ in Lemma \ref{lem:str-core} equals $\bar{\hres}_{S}^{\str}$
in Lemma \ref{lem:str-smo-res} under $\bar{D}\defeq\frac{2G_{1}}{\mu}$
due to Assumption \ref{assu:compact}. Therefore, Lemmas \ref{lem:str-core}
($\eta=6$ and $\omega=0$ now) and \ref{lem:str-smo-res} ($\eta=6$
and $\bar{D}=\frac{2G_{1}}{\mu}$ now) together yield that, with probability
at least $1-\delta$,
\begin{equation}
\left\Vert \bx_{\star}-\bx_{S+1}\right\Vert ^{2}\leq\frac{8D^{2}+4\bar{\hc}_{S}^{\str}}{\gamma_{S}},\label{eq:str-smo-hp-1}
\end{equation}
where $\bar{\hc}_{S}^{\str}$ is a constant satisfying
\begin{align}
\bar{\hc}_{S}^{\str}= & \O\left(\max_{t\in\left[S\right]}\gamma_{t}\eta_{t}^{2}\rho^{2}\ln^{2}\frac{3}{\delta}+\sum_{t=1}^{S}\sigma_{\lar}^{\p}\gamma_{t}\eta_{t}^{2}\cm^{2-\p}+\sum_{t=1}^{S}\left(\frac{\sigma_{\sma}^{\p}\gamma_{t}\eta_{t}^{2}\rho}{\cm^{\p-1}}+\min_{\beta\in\left[0,\frac{1}{2}\right]}\frac{\sigma_{\lar}^{\p}\bar{G}^{2-\beta\p}\gamma_{t}\eta_{t}^{2}}{\cm^{(1-\beta)\p}}\right)\ln\frac{3}{\delta}\right.\nonumber \\
 & \left.\quad+\left(\sum_{t=1}^{S}\frac{\sigma_{\sma}^{2}\sigma_{\lar}^{2\p-2}\gamma_{t}\eta_{t}}{\cm^{2\p-2}}+\min_{\beta\in\left[0,\frac{1}{2}\right]}\frac{\sigma_{\lar}^{2\p}\bar{G}^{2-2\beta\p}\gamma_{t}\eta_{t}}{\cm^{2(1-\beta)\p}}\right)\frac{1}{\mu}\right),\label{eq:str-smo-hp-C-1}
\end{align}
 $\cm=\cm_{\star}\left[T(T+\kappa)\right]^{\frac{1}{2\p}}$, $\bar{G}=L(\bar{D}+D)+G_{1}\overset{0\leq D\leq\bar{D}=2G_{1}/\mu}{=}\Theta((1+\kappa)G_{1})$,
$\rho=\max\left\{ 2\bar{G},\cm\right\} $, and $\left\{ \gamma_{t}\right\} _{t=1}^{S}$
is a sequence defined in (\ref{eq:str-gamma-def}).

Next, we compute
\begin{align*}
\gamma_{t} & \overset{(\ref{eq:str-gamma-def})}{=}\prod_{s=2}^{t}\frac{1+\mu\eta_{s-1}}{1+\mu\eta_{s}/2}\overset{\eta_{s}=\frac{6}{\mu(s+12\kappa)}}{=}\prod_{s=2}^{t}\frac{(s+12\kappa)(s+5+12\kappa)}{(s-1+12\kappa)(s+3+12\kappa)}\\
 & =\frac{(t+12\kappa)(t+4+12\kappa)(t+5+12\kappa)}{(1+12\kappa)(5+12\kappa)(6+12\kappa)},\forall t\in\left[S\right].
\end{align*}
By the above equation and (\ref{eq:str-smo-hp-1}), we have
\begin{equation}
\left\Vert \bx_{\star}-\bx_{S+1}\right\Vert ^{2}=\O\left(\frac{\left(1+\kappa\right)^{3}\left(D^{2}+\bar{\hc}_{S}^{\str}\right)}{\left(S+\kappa\right)^{3}}\right)\overset{S=\left\lceil \frac{T}{2}\right\rceil }{=}\O\left(\frac{\left(1+\kappa\right)^{3}\left(D^{2}+\bar{\hc}_{S}^{\str}\right)}{\left(T+\kappa\right)^{3}}\right).\label{eq:str-smo-hp-2}
\end{equation}
Noticing that $\gamma_{t}\eta_{t}=\Theta\left(\frac{(t+\kappa)^{2}}{\mu(1+\kappa)^{3}}\right)$,
$\gamma_{t}\eta_{t}^{2}=\Theta\left(\frac{t+\kappa}{\mu^{2}(1+\kappa)^{3}}\right)$,
and $S=\Theta(T)$, we can write (\ref{eq:str-smo-hp-C-1}) into
\begin{align}
\bar{\hc}_{S}^{\str}= & \O\left(\frac{\rho^{2}(T+\kappa)\ln^{2}\frac{3}{\delta}}{\mu^{2}(1+\kappa)^{3}}+\frac{\sigma_{\lar}^{\p}\cm^{2-\p}(T+\kappa)T}{\mu^{2}(1+\kappa)^{3}}\right.\nonumber \\
 & \left.\quad+\left(\frac{\sigma_{\sma}^{\p}\rho(T+\kappa)T}{\cm^{\p-1}\mu^{2}(1+\kappa)^{3}}+\min_{\beta\in\left[0,\frac{1}{2}\right]}\frac{\sigma_{\lar}^{\p}\bar{G}^{2-\beta\p}(T+\kappa)T}{\cm^{(1-\beta)\p}\mu^{2}(1+\kappa)^{3}}\right)\ln\frac{3}{\delta}\right.\nonumber \\
 & \left.\quad+\frac{\sigma_{\sma}^{2}\sigma_{\lar}^{2\p-2}(T+\kappa)^{2}T}{\cm^{2\p-2}\mu^{2}(1+\kappa)^{3}}+\min_{\beta\in\left[0,\frac{1}{2}\right]}\frac{\sigma_{\lar}^{2\p}\bar{G}^{2-2\beta\p}(T+\kappa)^{2}T}{\cm^{2(1-\beta)\p}\mu^{2}(1+\kappa)^{3}}\right),\nonumber \\
= & \O\left(\frac{\rho^{2}(T+\kappa)\ln^{2}\frac{3}{\delta}}{\mu^{2}(1+\kappa)^{3}}+\frac{\sigma_{\lar}^{\p}\cm^{2-\p}(T+\kappa)T}{\mu^{2}(1+\kappa)^{3}}\right.\nonumber \\
 & \left.\quad+\frac{\sigma_{\sma}^{2}\sigma_{\lar}^{2\p-2}(T+\kappa)^{2}T}{\cm^{2\p-2}\mu^{2}(1+\kappa)^{3}}+\min_{\beta\in\left[0,\frac{1}{2}\right]}\frac{\sigma_{\lar}^{2\p}\bar{G}^{2-2\beta\p}(T+\kappa)^{2}T}{\cm^{2(1-\beta)\p}\mu^{2}(1+\kappa)^{3}}\right),\label{eq:str-smo-hp-C-2}
\end{align}
where the second equation is due to
\begin{align*}
 & \frac{\rho^{2}(T+\kappa)\ln^{2}\frac{3}{\delta}}{\mu^{2}(1+\kappa)^{3}}+\frac{\sigma_{\sma}^{2}\sigma_{\lar}^{2\p-2}(T+\kappa)^{2}T}{\cm^{2\p-2}\mu^{2}(1+\kappa)^{3}}+\min_{\beta\in\left[0,\frac{1}{2}\right]}\frac{\sigma_{\lar}^{2\p}\bar{G}^{2-2\beta\p}(T+\kappa)^{2}T}{\cm^{2(1-\beta)\p}\mu^{2}(1+\kappa)^{3}}\\
\overset{(a)}{=} & \Omega\left(\frac{\sigma_{\sma}\sigma_{\lar}^{\p-1}\rho(T+\kappa)^{\frac{3}{2}}T^{\frac{1}{2}}\ln\frac{3}{\delta}}{\cm^{\p-1}\mu^{2}(1+\kappa)^{3}}+\min_{\beta\in\left[0,\frac{1}{2}\right]}\frac{\sigma_{\lar}^{\p}\bar{G}^{1-\beta\p}\rho(T+\kappa)^{\frac{3}{2}}T^{\frac{1}{2}}\ln\frac{3}{\delta}}{\cm^{(1-\beta)\p}\mu^{2}(1+\kappa)^{3}}\right)\\
\overset{(b)}{=} & \Omega\left(\frac{\sigma_{\sma}^{\p}\rho(T+\kappa)T\ln\frac{3}{\delta}}{\cm^{\p-1}\mu^{2}(1+\kappa)^{3}}+\min_{\beta\in\left[0,\frac{1}{2}\right]}\frac{\sigma_{\lar}^{\p}\bar{G}^{2-\beta\p}(T+\kappa)T\ln\frac{3}{\delta}}{\cm^{(1-\beta)\p}\mu^{2}(1+\kappa)^{3}}\right),
\end{align*}
in which $(a)$ is by AM-GM inequality and $(b)$ is due to $\sigma_{\lar}\geq\sigma_{\sma}$,
$\p\geq1$, $\rho=\max\left\{ 2\bar{G},\cm\right\} \geq\bar{G}$,
and $T+\kappa\geq T$. Now, we use $\rho^{2}\leq4\bar{G}^{2}+\cm^{2}$
and $\cm=\cm_{\star}\left[T(T+\kappa)\right]^{\frac{1}{2\p}}$ to
simplify (\ref{eq:str-smo-hp-C-2}) into
\begin{align}
\bar{\hc}_{S}^{\str} & =\O\left(\frac{\left(\min_{\beta\in\left[0,\frac{1}{2}\right]}\frac{\sigma_{\lar}^{2\p}\left[T(T+\kappa)\right]^{\beta}}{\bar{G}^{2\beta\p}\cm_{\star}^{2(1-\beta)\p}}+\ln^{2}\frac{3}{\delta}\right)\bar{G}^{2}}{\mu^{2}(1+\kappa)^{3}}\left(T+\kappa\right)\right.\nonumber \\
 & \left.\quad+\frac{\cm_{\star}^{2}\ln^{2}\frac{3}{\delta}+\sigma_{\lar}^{\p}\cm_{\star}^{2-\p}+\frac{\sigma_{\sma}^{2}\sigma_{\lar}^{2\p-2}}{\cm_{\star}^{2\p-2}}}{\mu^{2}(1+\kappa)^{3}}T^{\frac{1}{\p}}\left(T+\kappa\right)^{1+\frac{1}{\p}}\right)\nonumber \\
 & \overset{(\ref{eq:hp-tau-star}),(\ref{eq:smo-hp-varphi})}{=}\O\left(\frac{\left(\varphi_{\smo}^{2}\sqrt{1+\frac{\kappa}{T}}+\ln^{2}\frac{3}{\delta}\right)\bar{G}^{2}}{\mu^{2}(1+\kappa)^{3}}\left(T+\kappa\right)\right.\nonumber \\
 & \left.\quad+\frac{\sigma_{\sma}^{\frac{4}{\p}-2}\sigma_{\lar}^{4-\frac{4}{\p}}+\sigma_{\sma}^{\frac{2}{\p}}\sigma_{\lar}^{2-\frac{2}{\p}}\ln^{2-\frac{2}{\p}}\frac{3}{\delta}}{\mu^{2}(1+\kappa)^{3}}T^{\frac{1}{\p}}\left(T+\kappa\right)^{1+\frac{1}{\p}}\right).\label{eq:str-smo-hp-C-simple}
\end{align}
Hence, we obtain that, with probability at least $1-\delta$,
\begin{align}
 & \left\Vert \bx_{\star}-\bx_{S+1}\right\Vert ^{2}\nonumber \\
\overset{(\ref{eq:str-smo-hp-2}),(\ref{eq:str-smo-hp-C-simple})}{=} & \O\left(\frac{\left(1+\kappa\right)^{3}D^{2}}{\left(T+\kappa\right)^{3}}+\frac{\left(\varphi_{\smo}^{2}\sqrt{1+\frac{\kappa}{T}}+\ln^{2}\frac{3}{\delta}\right)\bar{G}^{2}}{\mu^{2}\left(T+\kappa\right)^{2}}+\frac{\sigma_{\sma}^{\frac{4}{\p}-2}\sigma_{\lar}^{4-\frac{4}{\p}}+\sigma_{\sma}^{\frac{2}{\p}}\sigma_{\lar}^{2-\frac{2}{\p}}\ln^{2-\frac{2}{\p}}\frac{3}{\delta}}{\mu^{2}\left(T+\kappa\right)^{2-\frac{1}{\p}}T^{-\frac{1}{\p}}}\right)\nonumber \\
= & \O\left(\frac{\left(\varphi_{\smo}^{2}\sqrt{1+\frac{\kappa}{T}}+\ln^{2}\frac{3}{\delta}\right)\bar{G}^{2}}{\mu^{2}\left(T+\kappa\right)^{2}}+\frac{\sigma_{\sma}^{\frac{4}{\p}-2}\sigma_{\lar}^{4-\frac{4}{\p}}+\sigma_{\sma}^{\frac{2}{\p}}\sigma_{\lar}^{2-\frac{2}{\p}}\ln^{2-\frac{2}{\p}}\frac{3}{\delta}}{\mu^{2}\left(T+\kappa\right)^{2-\frac{1}{\p}}T^{-\frac{1}{\p}}}\right),\label{eq:str-smo-hp-3}
\end{align}
where the second step is due to $\frac{\left(1+\kappa\right)^{3}D^{2}}{\left(T+\kappa\right)^{3}}\leq\frac{\left(1+\kappa\right)^{2}D^{2}}{\left(T+\kappa\right)^{2}}\overset{\text{Assumption }\ref{assu:compact}}{=}\O\left(\frac{\left(1+\kappa\right)^{2}G_{1}^{2}}{\mu^{2}\left(T+\kappa\right)^{2}}\right)=\O\left(\frac{\bar{G}^{2}}{\mu^{2}\left(T+\kappa\right)^{2}}\right)$.

Finally, when $T\geq2$, we observe that $\bx_{T+1}$ can be viewed
as the output of Algorithm \ref{alg:Clipped-SGD} after $T-S=\left\lfloor \frac{T}{2}\right\rfloor $
iterations with the initial point $\bx_{S+1}$. Moreover, note that
for any $t>S$,
\begin{align*}
\left\Vert \nabla f(\bx_{t})\right\Vert  & \leq\left\Vert \nabla f(\bx_{t})-\nabla f(\bx_{1})\right\Vert +\left\Vert \nabla f(\bx_{1})\right\Vert \leq L\left\Vert \bx_{t}-\bx_{1}\right\Vert +\left\Vert \nabla f(\bx_{1})\right\Vert \\
 & \overset{(c)}{\leq}2\kappa G_{1}+G_{1}=(1+2\kappa)G_{1}\leq\frac{\cm_{t}}{2},
\end{align*}
where $(c)$ follows from Assumption \ref{assu:compact}. Therefore,
let $\eta_{\star}\defeq\frac{1}{\mu T(T+\kappa)}$, by our choices
of $\left\{ \eta_{t}\right\} _{t=S+1}^{T}$ and $\left\{ \cm_{t}\right\} _{t=S+1}^{T}$
and the proof of Theorem \ref{thm:cvx-smo-hp} (in which we can directly
set $\bar{D}=\frac{2G_{1}}{\mu}$ due to Assumption \ref{assu:compact},
resulting in $\zeta_{t}=1$ almost surely), we have, for $\bar{G}=(1+2\kappa)G_{1}$,
with probability at least $1-\delta$,
\begin{align}
F(\bx_{T+1})-F_{\star}\overset{(\ref{eq:cvx-smo-hp-C}),(\ref{eq:cvx-smo-hp-3})}{=} & \O\left(\frac{\left\Vert \bx_{\star}-\bx_{S+1}\right\Vert ^{2}}{\eta_{\star}\left\lfloor \frac{T}{2}\right\rfloor ^{2}}+\eta_{\star}\left(\varphi_{\smo}^{2}+\ln^{2}\frac{3}{\delta}\right)\bar{G}^{2}\right.\nonumber \\
 & \left.\quad+\eta_{\star}\left(\sigma_{\sma}^{\frac{4}{\p}-2}\sigma_{\lar}^{4-\frac{4}{\p}}+\sigma_{\sma}^{\frac{2}{\p}}\sigma_{\lar}^{2-\frac{2}{\p}}\ln^{2-\frac{2}{\p}}\frac{3}{\delta}\right)\left\lfloor \frac{T}{2}\right\rfloor ^{\frac{2}{\p}}\right)\nonumber \\
\overset{\eta_{\star}=\frac{1}{\mu T(T+\kappa)}}{=} & \O\left(\left(1+\frac{\kappa}{T}\right)\mu\left\Vert \bx_{\star}-\bx_{S+1}\right\Vert ^{2}+\frac{\left(\varphi_{\smo}^{2}+\ln^{2}\frac{3}{\delta}\right)\bar{G}^{2}}{\mu T(T+\kappa)}\right.\nonumber \\
 & \left.\quad+\frac{\sigma_{\sma}^{\frac{4}{\p}-2}\sigma_{\lar}^{4-\frac{4}{\p}}+\sigma_{\sma}^{\frac{2}{\p}}\sigma_{\lar}^{2-\frac{2}{\p}}\ln^{2-\frac{2}{\p}}\frac{3}{\delta}}{\mu T^{1-\frac{2}{\p}}(T+\kappa)}\right).\label{eq:str-smo-hp-4}
\end{align}
By combining (\ref{eq:str-smo-hp-3}) and (\ref{eq:str-smo-hp-4}),
we obtain the desired result.
\end{proof}

\section{Theoretical Analysis\label{sec:analysis}}

This section contains the missing analysis of each intermediate lemma
used in Section \ref{sec:theorems}. Before proceeding, we recall
that $\F_{t}=\sigma\left(Z_{1},\mydots,Z_{t}\right)$ denotes the
natural filtration. By the structure of Algorithm \ref{alg:Clipped-SGD},
we naturally have $\cm_{t}\in\F_{t-1}$ and $\eta_{t}\in\F_{t}$,
which implies that $\bx_{t}\in\F_{t-1}$ and $\nabla f(\bx_{t})\in\F_{t-1}$.
For each $t\in\left[T\right]$, we define the following standard notation
in the literature on $\ClipSGD$, 
\begin{eqnarray}
\bd_{t}^{\rc}\defeq\bg_{t}^{\rc}-\nabla f(\bx_{t}), & \bd_{t}^{\ru}\defeq\bg_{t}^{\rc}-\E\left[\bg_{t}^{\rc}\mid\F_{t-1}\right], & \bd_{t}^{\rb}\defeq\E\left[\bg_{t}^{\rc}\mid\F_{t-1}\right]-\nabla f(\bx_{t}),\label{eq:decomposition}
\end{eqnarray}
where $\bd_{t}^{\rc}\in\F_{t}$ denotes the difference between the
clipped stochastic gradient and the true gradient, $\bd_{t}^{\ru}$
is the unbiased part in $\bd_{t}^{\rc}$, and $\bd_{t}^{\rb}$ is
the biased part in $\bd_{t}^{\rc}$.

\subsection{General Lemmas}

This subsection provides two general lemmas used in the proof.

First, Lemma \ref{lem:clip-ineq} below gives the clipping error bounds
under Assumption \ref{assu:heavy} recently proved by \citep{ICLR2026_bd5b846a}.
\begin{lem}[Lemma 2 of \citep{ICLR2026_bd5b846a}]
\label{lem:clip-ineq}Under Assumption \ref{assu:heavy}, let $\chi_{t}\defeq\1\left[\cm_{t}\geq2\left\Vert \nabla f(\bx_{t})\right\Vert \right]$,
then there are:
\begin{enumerate}[leftmargin=*]
\item $\left\Vert \bd_{t}^{\ru}\right\Vert \leq2\cm_{t}$.
\item $\E\left[\left\Vert \bd_{t}^{\ru}\right\Vert ^{2}\mid\F_{t-1}\right]\leq4\sigma_{\lar}^{\p}\cm_{t}^{2-\p}$.
\item $\left\Vert \E\left[\bd_{t}^{\ru}\left(\bd_{t}^{\ru}\right)^{\top}\mid\F_{t-1}\right]\right\Vert _{\op}\chi_{t}\leq4\sigma_{\sma}^{\p}\cm_{t}^{2-\p}+2^{\p+1}\sigma_{\lar}^{\p}\left\Vert \nabla f(\bx_{t})\right\Vert ^{2}\cm_{t}^{-\p}$.
\item $\left\Vert \bd_{t}^{\rb}\right\Vert \chi_{t}\leq\sigma_{\sma}\sigma_{\lar}^{\p-1}\cm_{t}^{1-\p}+2^{\p-1}\sigma_{\lar}^{\p}\left\Vert \nabla f(\bx_{t})\right\Vert \cm_{t}^{-\p}$.
\end{enumerate}
\end{lem}
\begin{rem}
The original Lemma 2 of \citep{ICLR2026_bd5b846a} is stated more
generally and includes two additional inequalities. Here, we present
a simplified version containing only the ones used in this paper.
\end{rem}
Next, we give a one-step descent inequality. Because it is standard
in the literature, we omit the proof to save space.
\begin{lem}
\label{lem:descent}Under Assumptions \ref{assu:obj} and \ref{assu:lip}
(resp. \ref{assu:smo}), for any $\by\in\X$, $t\in\N$, and $\eta_{t}>0$
(resp. $0<\eta_{t}\leq\frac{1}{2L}$), $\ClipSGD$ (Algorithm \ref{alg:Clipped-SGD})
guarantees:
\[
F(\bx_{t+1})-F(\by)\leq\frac{\left\Vert \by-\bx_{t}\right\Vert ^{2}-\left\Vert \by-\bx_{t+1}\right\Vert ^{2}}{2\eta_{t}}+\left\langle \bd_{t}^{\rc},\by-\bx_{t}\right\rangle +\eta_{t}\left\Vert \bd_{t}^{\rc}\right\Vert ^{2}+4\eta_{t}G^{2}\omega,
\]
where $\omega\defeq\begin{cases}
1 & \text{under Assumption \ref{assu:lip}}\\
0 & \text{under Assumption \ref{assu:smo}}
\end{cases}$.
\end{lem}

\subsection{Lemmas for General Convex Optimization}

In this subsection, we will prove the lemmas used in general convex
optimization, i.e., $\mu=0$ in Assumption \ref{assu:obj}.

Following \citep{ICLR2024_b1239ceb} and \citep{doi:10.1137/24M1717762},
we first introduce an auxiliary sequence $\left\{ \bz_{t}\right\} _{t=0}^{T}$
defined as follows
\begin{eqnarray}
\bz_{0}\defeq\bx_{\star} & \text{and} & \bz_{t}\defeq\frac{v_{t-1}^{2}}{v_{t}^{2}}\bz_{t-1}+\left(1-\frac{v_{t-1}^{2}}{v_{t}^{2}}\right)\bx_{t},\forall t\in\left[T\right],\label{eq:cvx-z-def}
\end{eqnarray}
where $\bx_{\star}$ is the optimal solution given in Assumption \ref{assu:minimizer},
and the sequence $\left\{ v_{t}^{2}\right\} _{t=0}^{T}$ is defined
as
\begin{eqnarray}
v_{t}^{2}\defeq\frac{\eta_{T}}{\sum_{s=t}^{T}\eta_{s}},\forall t\in\left[T\right] & \text{and} & v_{0}^{2}\defeq v_{1}^{2}.\label{eq:cvx-v-def}
\end{eqnarray}
Equivalently, there is
\begin{equation}
\bz_{t}=\frac{v_{0}^{2}}{v_{t}^{2}}\bx_{\star}+\sum_{s=1}^{t}\frac{v_{s}^{2}-v_{s-1}^{2}}{v_{t}^{2}}\bx_{s},\forall t\in\left[T\right],\label{eq:cvx-z-equiv}
\end{equation}
which implies that $\bz_{t}$ is a convex combination of $\bx_{\star}$
and $\left\{ \bx_{s}\right\} _{s=1}^{t}$ (since $\left\{ v_{t}^{2}\right\} _{t=0}^{T}$
is positive and nondecreasing), and hence falls into $\X$.

Next, inspired by \citep{ICLR2026_bd5b846a}, we introduce another
two sequences $\left\{ R_{t}\right\} _{t=1}^{T}$ and $\left\{ M_{t}\right\} _{t=1}^{T}$
as follows
\begin{eqnarray}
R_{t}\defeq\max_{s\in\left[t\right]}v_{s-1}\left\Vert \bz_{s-1}-\bx_{s}\right\Vert  & \text{and} & M_{t}\defeq\left\langle v_{t-1}\eta_{t}\bd_{t}^{\ru},\frac{v_{t-1}\left(\bz_{t-1}-\bx_{t}\right)}{R_{t}}\right\rangle .\label{eq:cvx-R-M-def}
\end{eqnarray}
One more step, like \citep{liu2023stochastic}, we define two more
sequences $\left\{ D_{t}\right\} _{t=1}^{T}$ and $\left\{ N_{t}\right\} _{t=1}^{T}$
in the following
\begin{eqnarray}
D_{t}\defeq\max_{s\in\left[t\right]}\left\Vert \bx_{\star}-\bx_{s}\right\Vert  & \text{and} & N_{t}\defeq\left\langle \eta_{t}\bd_{t}^{\ru},\frac{\bx_{\star}-\bx_{t}}{D_{t}}\right\rangle .\label{eq:cvx-D-N-def}
\end{eqnarray}
As one may observe, $\left\{ D_{t}\right\} _{t=1}^{T}$ (resp. $\left\{ N_{t}\right\} _{t=1}^{T}$)
corresponds to $\left\{ R_{t}\right\} _{t=1}^{T}$ (resp. $\left\{ M_{t}\right\} _{t=1}^{T}$)
when $v_{t}=1,\forall t\in\left\{ 0\right\} \cup\left[T\right]$.

Now, we are ready to prove the following Lemma \ref{lem:cvx-core},
which not only characterizes the last-iterate progress made by $\ClipSGD$
after $T$ iterations, but also provides a bound on $\left\Vert \bx_{\star}-\bx_{t+1}\right\Vert ^{2}$
for any $t\in\left[T\right]$.
\begin{lem}
\label{lem:cvx-core}Under Assumptions \ref{assu:minimizer}, \ref{assu:obj}
(with $\mu=0$), and \ref{assu:lip} (resp. \ref{assu:smo}), for
any $T\in\N$ and $\left\{ \eta_{t}\right\} _{t=1}^{T}$ satisfying
$\eta_{t}>0,\forall t\in\left[T\right]$ (resp. $0<\eta_{t}\leq\frac{1}{2L},\forall t\in\left[T\right]$),
$\ClipSGD$ (Algorithm \ref{alg:Clipped-SGD}) guarantees
\begin{align*}
F(\bx_{T+1})-F_{\star} & \leq\frac{v_{0}^{2}D^{2}+2\lasthres_{T}^{\cvx}+8\sum_{t=1}^{T}v_{t}^{2}\eta_{t}^{2}G^{2}\omega}{\eta_{T}},\\
\left\Vert \bx_{\star}-\bx_{t+1}\right\Vert ^{2} & \leq2D^{2}+4\hres_{t}^{\cvx}+16\sum_{s=1}^{t}\eta_{s}^{2}G^{2}\omega,\forall t\in\left[T\right],
\end{align*}
where 
\begin{align*}
\lasthres_{T}^{\cvx} & =8\max_{t\in\left[T\right]}\left(\sum_{s=1}^{t}M_{s}\right)^{2}+2\sum_{t=1}^{T}v_{t}^{2}\eta_{t}^{2}\left\Vert \bd_{t}^{\ru}\right\Vert ^{2}+4\left(\sum_{t=1}^{T}v_{t}\eta_{t}\left\Vert \bd_{t}^{\rb}\right\Vert \right)^{2},\\
\hres_{t}^{\cvx} & =8\max_{t^{\prime}\in\left[t\right]}\left(\sum_{s=1}^{t^{\prime}}N_{s}\right)^{2}+2\sum_{s=1}^{t}\eta_{s}^{2}\left\Vert \bd_{s}^{\ru}\right\Vert ^{2}+4\left(\sum_{s=1}^{t}\eta_{s}\left\Vert \bd_{s}^{\rb}\right\Vert \right)^{2},\forall t\in\left[T\right],
\end{align*}
and $\omega=\begin{cases}
1 & \text{under Assumption \ref{assu:lip}}\\
0 & \text{under Assumption \ref{assu:smo}}
\end{cases}$ is introduced in Lemma \ref{lem:descent}.
\end{lem}
\begin{proof}
We invoke Lemma \ref{lem:descent} with $\by=\bz_{t}$ (defined in
(\ref{eq:cvx-z-def})) and replace the subscript $t$ with $s$ to
have, for any $s\in\left[T\right]$,
\[
F(\bx_{s+1})-F(\bz_{s})\leq\frac{\left\Vert \bz_{s}-\bx_{s}\right\Vert ^{2}-\left\Vert \bz_{s}-\bx_{s+1}\right\Vert ^{2}}{2\eta_{s}}+\left\langle \bd_{s}^{\rc},\bz_{s}-\bx_{s}\right\rangle +\eta_{s}\left\Vert \bd_{s}^{\rc}\right\Vert ^{2}+4\eta_{s}G^{2}\omega,
\]
where $\omega=\begin{cases}
1 & \text{under Assumption }\ref{assu:lip}\\
0 & \text{under Assumption }\ref{assu:smo}
\end{cases}$. Multiply both sides of the above inequality by $v_{s}^{2}\eta_{s}$
and use $\bz_{s}-\bx_{s}\overset{(\ref{eq:cvx-z-def})}{=}\frac{v_{s-1}^{2}}{v_{s}^{2}}\left(\bz_{s-1}-\bx_{s}\right)$
and $v_{s-1}^{2}\leq v_{s}^{2}$ to obtain
\begin{align*}
v_{s}^{2}\eta_{s}\left(F(\bx_{s+1})-F(\bz_{s})\right)\leq & \frac{v_{s-1}^{2}\left\Vert \bz_{s-1}-\bx_{s}\right\Vert ^{2}-v_{s}^{2}\left\Vert \bz_{s}-\bx_{s+1}\right\Vert ^{2}}{2}\\
 & +\left\langle v_{s-1}\eta_{s}\bd_{s}^{\rc},v_{s-1}\left(\bz_{s-1}-\bx_{s}\right)\right\rangle +v_{s}^{2}\eta_{s}^{2}\left\Vert \bd_{s}^{\rc}\right\Vert ^{2}+4v_{s}^{2}\eta_{s}^{2}G^{2}\omega,
\end{align*}
sum up which from $s=1$ to $t$ where $t\in\left[T\right]$, use
$\bz_{0}\overset{(\ref{eq:cvx-z-def})}{=}\bx_{\star}$ and $D=\left\Vert \bx_{\star}-\bx_{1}\right\Vert $,
and rearrange terms to have
\begin{align}
 & \frac{v_{t}^{2}\left\Vert \bz_{t}-\bx_{t+1}\right\Vert ^{2}}{2}+\sum_{s=1}^{t}v_{s}^{2}\eta_{s}\left(F(\bx_{s+1})-F(\bz_{s})\right)\nonumber \\
\leq & \frac{v_{0}^{2}D^{2}}{2}+\sum_{s=1}^{t}\left\langle v_{s-1}\eta_{s}\bd_{s}^{\rc},v_{s-1}\left(\bz_{s-1}-\bx_{s}\right)\right\rangle +\sum_{s=1}^{t}v_{s}^{2}\eta_{s}^{2}\left\Vert \bd_{s}^{\rc}\right\Vert ^{2}+4\sum_{s=1}^{t}v_{s}^{2}\eta_{s}^{2}G^{2}\omega.\label{eq:cvx-core-1}
\end{align}

Next, we apply the convexity of $F$ to upper bound
\[
F(\bz_{s})-F_{\star}\overset{(\ref{eq:cvx-z-equiv})}{\leq}\sum_{\ell=1}^{s}\frac{v_{\ell}^{2}-v_{\ell-1}^{2}}{v_{s}^{2}}\left(F(\bx_{\ell})-F_{\star}\right)=\sum_{\ell=2}^{s}\frac{v_{\ell}^{2}-v_{\ell-1}^{2}}{v_{s}^{2}}\left(F(\bx_{\ell})-F_{\star}\right),
\]
where the last equation is due to $v_{1}\overset{(\ref{eq:cvx-v-def})}{=}v_{0}$.
Hence, we know
\begin{align}
 & \sum_{s=1}^{t}v_{s}^{2}\eta_{s}\left(F(\bx_{s+1})-F(\bz_{s})\right)\nonumber \\
\geq & \sum_{s=1}^{t}v_{s}^{2}\eta_{s}\left(F(\bx_{s+1})-F_{\star}-\sum_{\ell=2}^{s}\frac{v_{\ell}^{2}-v_{\ell-1}^{2}}{v_{s}^{2}}\left(F(\bx_{\ell})-F_{\star}\right)\right)\nonumber \\
= & \sum_{s=1}^{t}v_{s}^{2}\eta_{s}\left(F(\bx_{s+1})-F_{\star}\right)-\sum_{s=2}^{t}\sum_{\ell=2}^{s}\eta_{s}(v_{\ell}^{2}-v_{\ell-1}^{2})\left(F(\bx_{\ell})-F_{\star}\right)\nonumber \\
= & v_{t}^{2}\eta_{t}\left(F(\bx_{t+1})-F_{\star}\right)+\sum_{s=2}^{t}\left[v_{s-1}^{2}\eta_{s-1}-\left(\sum_{\ell=s}^{t}\eta_{\ell}\right)(v_{s}^{2}-v_{s-1}^{2})\right]\left(F(\bx_{s})-F_{\star}\right)\nonumber \\
\overset{(a)}{=} & v_{t}^{2}\eta_{t}\left(F(\bx_{t+1})-F_{\star}\right)+\sum_{s=2}^{t}\left(v_{s}^{2}-v_{s-1}^{2}\right)\left(\sum_{\ell=t+1}^{T}\eta_{\ell}\right)\left(F(\bx_{s})-F_{\star}\right)\nonumber \\
\overset{(b)}{\geq} & v_{t}^{2}\eta_{t}\left(F(\bx_{t+1})-F_{\star}\right),\label{eq:cvx-core-2}
\end{align}
where $(a)$ is by
\[
v_{s-1}^{2}\eta_{s-1}-\left(\sum_{\ell=s}^{t}\eta_{\ell}\right)(v_{s}^{2}-v_{s-1}^{2})=v_{s-1}^{2}\left(\sum_{\ell=s-1}^{t}\eta_{\ell}\right)-v_{s}^{2}\left(\sum_{\ell=s}^{t}\eta_{\ell}\right)\overset{(\ref{eq:cvx-v-def})}{=}\left(v_{s}^{2}-v_{s-1}^{2}\right)\left(\sum_{\ell=t+1}^{T}\eta_{\ell}\right),
\]
and $(b)$ is due to $v_{s}^{2}\geq v_{s-1}^{2}\geq0$ and $F(\bx_{s})-F_{\star}\geq0$.
We plug (\ref{eq:cvx-core-2}) back into (\ref{eq:cvx-core-1}) to
obtain
\begin{align}
 & \frac{v_{t}^{2}\left\Vert \bz_{t}-\bx_{t+1}\right\Vert ^{2}}{2}+v_{t}^{2}\eta_{t}\left(F(\bx_{t+1})-F_{\star}\right)\nonumber \\
\leq & \frac{v_{0}^{2}D^{2}}{2}+\sum_{s=1}^{t}\left\langle v_{s-1}\eta_{s}\bd_{s}^{\rc},v_{s-1}\left(\bz_{s-1}-\bx_{s}\right)\right\rangle +\sum_{s=1}^{t}v_{s}^{2}\eta_{s}^{2}\left\Vert \bd_{s}^{\rc}\right\Vert ^{2}+4\sum_{s=1}^{t}v_{s}^{2}\eta_{s}^{2}G^{2}\omega.\label{eq:cvx-core-3}
\end{align}

Moreover, by $\bd_{s}^{\rc}\overset{(\ref{eq:decomposition})}{=}\bd_{s}^{\ru}+\bd_{s}^{\rb}$,
we have the decomposition
\begin{align*}
 & \left\langle v_{s-1}\eta_{s}\bd_{s}^{\rc},v_{s-1}\left(\bz_{s-1}-\bx_{s}\right)\right\rangle \overset{(\ref{eq:cvx-R-M-def})}{=}R_{s}M_{s}+\left\langle v_{s-1}\eta_{s}\bd_{s}^{\rb},v_{s-1}\left(\bz_{s-1}-\bx_{s}\right)\right\rangle \\
\leq & R_{s}M_{s}+v_{s-1}\eta_{s}\left\Vert \bd_{s}^{\rb}\right\Vert \cdot v_{s-1}\left\Vert \bz_{s-1}-\bx_{s}\right\Vert \overset{(\ref{eq:cvx-R-M-def}),v_{s-1}\leq v_{s}}{\leq}R_{s}M_{s}+R_{s}v_{s}\eta_{s}\left\Vert \bd_{s}^{\rb}\right\Vert ,
\end{align*}
which implies that
\begin{align*}
 & \sum_{s=1}^{t}\left\langle v_{s-1}\eta_{s}\bd_{s}^{\rc},v_{s-1}\left(\bz_{s-1}-\bx_{s}\right)\right\rangle \leq\sum_{s=1}^{t}R_{s}M_{s}+\sum_{s=1}^{t}R_{s}v_{s}\eta_{s}\left\Vert \bd_{s}^{\rb}\right\Vert \\
\overset{(c)}{\leq} & 2R_{t}\max_{t^{\prime}\in\left[t\right]}\left|\sum_{s=1}^{t^{\prime}}M_{s}\right|+R_{t}\sum_{s=1}^{t}v_{s}\eta_{s}\left\Vert \bd_{s}^{\rb}\right\Vert \overset{(d)}{\leq}\frac{R_{t}^{2}}{4}+8\max_{t^{\prime}\in\left[t\right]}\left(\sum_{s=1}^{t^{\prime}}M_{s}\right)^{2}+2\left(\sum_{s=1}^{t}v_{s}\eta_{s}\left\Vert \bd_{s}^{\rb}\right\Vert \right)^{2},
\end{align*}
where $(c)$ is due to Lemma \ref{lem:Abel}, and $(d)$ is by AM-GM
inequality. Combine the above inequality and (\ref{eq:cvx-core-3})
to know
\begin{align}
 & \frac{v_{t}^{2}\left\Vert \bz_{t}-\bx_{t+1}\right\Vert ^{2}}{2}+v_{t}^{2}\eta_{t}\left(F(\bx_{t+1})-F_{\star}\right)\nonumber \\
\leq & \frac{v_{0}^{2}D^{2}}{2}+\frac{R_{t}^{2}}{4}+8\max_{t^{\prime}\in\left[t\right]}\left(\sum_{s=1}^{t^{\prime}}M_{s}\right)^{2}+2\left(\sum_{s=1}^{t}v_{s}\eta_{s}\left\Vert \bd_{s}^{\rb}\right\Vert \right)^{2}+\sum_{s=1}^{t}v_{s}^{2}\eta_{s}^{2}\left\Vert \bd_{s}^{\rc}\right\Vert ^{2}+4\sum_{s=1}^{t}v_{s}^{2}\eta_{s}^{2}G^{2}\omega\nonumber \\
\leq & \frac{R_{t}^{2}}{4}+\frac{v_{0}^{2}D^{2}}{2}+\underbrace{8\max_{t^{\prime}\in\left[t\right]}\left(\sum_{s=1}^{t^{\prime}}M_{s}\right)^{2}+2\sum_{s=1}^{t}v_{s}^{2}\eta_{s}^{2}\left\Vert \bd_{s}^{\ru}\right\Vert ^{2}+4\left(\sum_{s=1}^{t}v_{s}\eta_{s}\left\Vert \bd_{s}^{\rb}\right\Vert \right)^{2}}_{\defeq\lasthres_{t}^{\cvx}}+4\sum_{s=1}^{t}v_{s}^{2}\eta_{s}^{2}G^{2}\omega,\label{eq:cvx-core-4}
\end{align}
where the last step is by
\begin{align*}
\left\Vert \bd_{s}^{\rc}\right\Vert ^{2} & \overset{(\ref{eq:decomposition})}{=}\left\Vert \bd_{s}^{\ru}+\bd_{s}^{\rb}\right\Vert ^{2}\leq2\left\Vert \bd_{s}^{\ru}\right\Vert ^{2}+2\left\Vert \bd_{s}^{\rb}\right\Vert ^{2}\\
\Rightarrow\sum_{s=1}^{t}v_{s}^{2}\eta_{s}^{2}\left\Vert \bd_{s}^{\rc}\right\Vert ^{2} & \leq2\sum_{s=1}^{t}v_{s}^{2}\eta_{s}^{2}\left\Vert \bd_{s}^{\ru}\right\Vert ^{2}+2\sum_{s=1}^{t}v_{s}^{2}\eta_{s}^{2}\left\Vert \bd_{s}^{\rb}\right\Vert ^{2}\leq2\sum_{s=1}^{t}v_{s}^{2}\eta_{s}^{2}\left\Vert \bd_{s}^{\ru}\right\Vert ^{2}+2\left(\sum_{s=1}^{t}v_{s}\eta_{s}\left\Vert \bd_{s}^{\rb}\right\Vert \right)^{2}.
\end{align*}

Now, let $a_{t}\defeq\frac{v_{t-1}^{2}\left\Vert \bz_{t-1}-\bx_{t}\right\Vert ^{2}}{2}$,
$b_{t}\defeq v_{t}^{2}\eta_{t}\left(F(\bx_{t+1})-F_{\star}\right)$,
and $c_{t}\defeq\frac{v_{0}^{2}D^{2}}{2}+\lasthres_{t-1}^{\cvx}+4\sum_{s=1}^{t-1}v_{s}^{2}\eta_{s}^{2}G^{2}\omega$.
Then, (\ref{eq:cvx-core-4}) is saying that
\[
a_{t+1}+b_{t}\leq\frac{\max_{s\in\left[t\right]}a_{s}}{2}+c_{t+1},\forall t\in\left[T\right].
\]
Note that $b_{t}\geq0$, $c_{t}$ is nondecreasing, and $a_{1}=\frac{v_{0}^{2}D^{2}}{2}\leq v_{0}^{2}D^{2}=2c_{1}$.
Thus, by Lemma \ref{lem:algebra}, we have
\[
a_{t+1}+b_{t}\leq2c_{t+1},\forall t\in\left[T\right],
\]
implying that
\[
v_{T}^{2}\eta_{T}\left(F(\bx_{T+1})-F_{\star}\right)=b_{T}\leq a_{T+1}+b_{T}\leq2c_{T+1}=v_{0}^{2}D^{2}+2\lasthres_{T}^{\cvx}+8\sum_{t=1}^{T}v_{t}^{2}\eta_{t}^{2}G^{2}\omega.
\]
Finally, we divide both sides by $v_{T}^{2}\eta_{T}$ and use $v_{T}^{2}\overset{(\ref{eq:cvx-v-def})}{=}1$
to complete the proof for the first inequality.

To prove the inequality for $\left\Vert \bx_{\star}-\bx_{t+1}\right\Vert ^{2}$,
we can simply repeat the above analysis, up to replacing $\by=\bz_{s}$
with $\by=\bx_{\star}$. One new quantity defined in the proof is
\[
\hres_{t}^{\cvx}\defeq8\max_{t^{\prime}\in\left[t\right]}\left(\sum_{s=1}^{t^{\prime}}N_{s}\right)^{2}+2\sum_{s=1}^{t}\eta_{s}^{2}\left\Vert \bd_{s}^{\ru}\right\Vert ^{2}+4\left(\sum_{s=1}^{t}\eta_{s}\left\Vert \bd_{s}^{\rb}\right\Vert \right)^{2},\forall t\in\left[T\right].
\]
 For simplicity, we omit the details.
\end{proof}

\subsubsection{Lipschitz Objectives}

In this part, we consider Lipschitz objectives (i.e., Assumption \ref{assu:lip}).
To prove the last-iterate rate, it is sufficient to bound the residual
term $\lasthres_{T}^{\cvx}$ in Lemma \ref{lem:cvx-core}. To do so,
we rely on the following Lemma \ref{lem:cvx-lip-res}.
\begin{lem}
\label{lem:cvx-lip-res}Under Assumptions \ref{assu:lip} and \ref{assu:heavy},
and the following two conditions:
\begin{enumerate}[leftmargin=*]
\item \label{enu:cvx-lip-res-1}$\eta_{t}$ and $\cm_{t}$ are deterministic
for all $t\in\left[T\right]$.
\item \label{enu:cvx-lip-res-2}$\cm_{t}\geq2G$ for all $t\in\left[T\right]$. 
\end{enumerate}
We have for any $\delta\in\left(0,1\right]$, with probability at
least $1-\delta$, $\lasthres_{T}^{\cvx}\leq\lasthc_{T}^{\cvx}$,
where $\lasthres_{T}^{\cvx}$ is introduced in Lemma \ref{lem:cvx-core}
and $\lasthc_{T}^{\cvx}$ is a constant in the order of
\[
\O\left(\max_{t\in\left[T\right]}v_{t}^{2}\eta_{t}^{2}\cm_{t}^{2}\ln^{2}\frac{3}{\delta}+\sum_{t=1}^{T}\frac{\sigma_{\lar}^{\p}v_{t}^{2}\eta_{t}^{2}}{\cm_{t}^{\p-2}}+\left(\sum_{t=1}^{T}\frac{\sigma_{\sma}\sigma_{\lar}^{\p-1}v_{t}\eta_{t}}{\cm_{t}^{\p-1}}+\frac{\sigma_{\lar}^{\p}Gv_{t}\eta_{t}}{\cm_{t}^{\p}}\right)^{2}\right).
\]
\end{lem}
\begin{proof}
Note that $\cm_{t}\overset{\text{Condition }\ref{enu:cvx-lip-res-2}}{\geq}2G\overset{\text{Assumption }\ref{assu:lip}}{\geq}2\left\Vert \nabla f(\bx_{t})\right\Vert ,\forall t\in\left[T\right]$
holds, implying that $\chi_{t}$ in Lemma \ref{lem:clip-ineq} equals
$1$ for all $t\in\left[T\right]$. Then Lemma \ref{lem:clip-ineq}
and Assumption \ref{assu:lip} together yield that the following collection
of inequalities holds for any $t\in\left[T\right]$:
\begin{align}
\left\Vert \bd_{t}^{\ru}\right\Vert  & \leq2\cm_{t},\label{eq:cvx-lip-res-u-norm}\\
\E\left[\left\Vert \bd_{t}^{\ru}\right\Vert ^{2}\mid\F_{t-1}\right] & \leq\frac{4\sigma_{\lar}^{\p}}{\cm_{t}^{\p-2}},\label{eq:cvx-lip-res-u-var}\\
\left\Vert \E\left[\bd_{t}^{\ru}\left(\bd_{t}^{\ru}\right)^{\top}\mid\F_{t-1}\right]\right\Vert _{\op} & \leq\frac{4\sigma_{\sma}^{\p}}{\cm_{t}^{\p-2}}+\frac{2^{\p+1}\sigma_{\lar}^{\p}G^{2}}{\cm_{t}^{\p}},\label{eq:cvx-lip-res-u-op}\\
\left\Vert \bd_{t}^{\rb}\right\Vert  & \leq\frac{\sigma_{\sma}\sigma_{\lar}^{\p-1}}{\cm_{t}^{\p-1}}+\frac{2^{\p-1}\sigma_{\lar}^{\p}G}{\cm_{t}^{\p}}.\label{eq:cvx-lip-res-b-norm}
\end{align}

Now, we are ready to bound $\lasthres_{T}^{\cvx}$.

\begin{itemize}[leftmargin=*]

\item We first observe that the sequence $\left\{ v_{t}\right\} _{t=1}^{T}$
is now deterministic by Condition \ref{enu:cvx-lip-res-1} and its
definition in (\ref{eq:cvx-v-def}), which further implies that $\bz_{t}\in\F_{t-1}$
and $R_{t}\in\F_{t-1}$ (see (\ref{eq:cvx-z-equiv}) and (\ref{eq:cvx-R-M-def})
for their definitions). Therefore, we know
\[
\E\left[M_{t}\mid\F_{t-1}\right]\overset{(\ref{eq:cvx-R-M-def})}{=}\left\langle v_{t-1}\eta_{t}\E\left[\bd_{t}^{\ru}\mid\F_{t-1}\right],\frac{v_{t-1}\left(\bz_{t-1}-\bx_{t}\right)}{R_{t}}\right\rangle =0,\forall t\in\left[T\right],
\]
meaning that $\left\{ M_{t}\right\} _{t=1}^{T}$ is a real-valued
MDS. By Cauchy-Schwarz inequality, we can bound
\begin{align*}
\left|M_{t}\right| & \overset{(\ref{eq:cvx-R-M-def})}{=}\left|\left\langle v_{t-1}\eta_{t}\bd_{t}^{\ru},\frac{v_{t-1}\left(\bz_{t-1}-\bx_{t}\right)}{R_{t}}\right\rangle \right|\leq v_{t-1}\eta_{t}\left\Vert \bd_{t}^{\ru}\right\Vert \cdot\frac{v_{t-1}\left\Vert \bz_{t-1}-\bx_{t}\right\Vert }{R_{t}}\\
 & \overset{(\ref{eq:cvx-R-M-def}),v_{t-1}\leq v_{t}}{\leq}v_{t}\eta_{t}\left\Vert \bd_{t}^{\ru}\right\Vert \overset{(\ref{eq:cvx-lip-res-u-norm})}{\leq}2v_{t}\eta_{t}\cm_{t}\leq2\max_{t\in\left[T\right]}v_{t}\eta_{t}\cm_{t}.
\end{align*}
In addition, the conditional variance of $M_{t}$ can be controlled
by
\begin{align*}
\E\left[M_{t}^{2}\mid\F_{t-1}\right] & =\left(\frac{v_{t-1}\left(\bz_{t-1}-\bx_{t}\right)}{R_{t}}\right)^{\top}v_{t-1}^{2}\eta_{t}^{2}\E\left[\bd_{t}^{\ru}\left(\bd_{t}^{\ru}\right)^{\top}\mid\F_{t-1}\right]\frac{v_{t-1}\left(\bz_{t-1}-\bx_{t}\right)}{R_{t}}\\
 & \overset{(\ref{eq:cvx-R-M-def}),v_{t-1}\leq v_{t}}{\leq}v_{t-1}^{2}\eta_{t}^{2}\left\Vert \E\left[\bd_{t}^{\ru}\left(\bd_{t}^{\ru}\right)^{\top}\mid\F_{t-1}\right]\right\Vert _{\op}\overset{(\ref{eq:cvx-lip-res-u-op})}{\leq}\frac{4\sigma_{\sma}^{\p}v_{t}^{2}\eta_{t}^{2}}{\cm_{t}^{\p-2}}+\frac{2^{\p+1}\sigma_{\lar}^{\p}G^{2}v_{t}^{2}\eta_{t}^{2}}{\cm_{t}^{\p}}.
\end{align*}
Therefore, by Freedman's inequality (Lemma \ref{lem:Freedman}), with
probability at least $1-2\delta/3$,
\[
\left|\sum_{s=1}^{t}M_{s}\right|\leq\frac{2}{3}\max_{s\in\left[T\right]}v_{s}\eta_{s}\cm_{s}\ln\frac{3}{\delta}+\sqrt{8\sum_{s=1}^{T}\left(\frac{\sigma_{\sma}^{\p}v_{s}^{2}\eta_{s}^{2}}{\cm_{s}^{\p-2}}+\frac{2^{\p-1}\sigma_{\lar}^{\p}G^{2}v_{s}^{2}\eta_{s}^{2}}{\cm_{s}^{\p}}\right)\ln\frac{3}{\delta}},\forall t\in\left[T\right],
\]
which implies that
\begin{equation}
\max_{t\in\left[T\right]}\left(\sum_{s=1}^{t}M_{s}\right)^{2}\leq\frac{8}{9}\max_{t\in\left[T\right]}v_{t}^{2}\eta_{t}^{2}\cm_{t}^{2}\ln^{2}\frac{3}{\delta}+16\sum_{t=1}^{T}\left(\frac{\sigma_{\sma}^{\p}v_{t}^{2}\eta_{t}^{2}}{\cm_{t}^{\p-2}}+\frac{2^{\p-1}\sigma_{\lar}^{\p}G^{2}v_{t}^{2}\eta_{t}^{2}}{\cm_{t}^{\p}}\right)\ln\frac{3}{\delta}.\label{eq:cvx-lip-res-1}
\end{equation}

\item Next, note that $\left\{ v_{t}\eta_{t}\left\Vert \bd_{t}^{\ru}\right\Vert \right\} _{t=1}^{T}$
is a sequence of random variables satisfying
\begin{eqnarray*}
v_{t}\eta_{t}\left\Vert \bd_{t}^{\ru}\right\Vert \overset{(\ref{eq:cvx-lip-res-u-norm})}{\leq}2v_{t}\eta_{t}\cm_{t}\leq2\max_{t\in\left[T\right]}v_{t}\eta_{t}\cm_{t} & \text{and} & \E\left[v_{t}^{2}\eta_{t}^{2}\left\Vert \bd_{t}^{\ru}\right\Vert ^{2}\mid\F_{t-1}\right]\overset{(\ref{eq:cvx-lip-res-u-var})}{\leq}\frac{4\sigma_{\lar}^{\p}v_{t}^{2}\eta_{t}^{2}}{\cm_{t}^{\p-2}}.
\end{eqnarray*}
Therefore, by Lemma \ref{lem:Freedman-square}, with probability at
least $1-\delta/3$,
\begin{align}
\sum_{t=1}^{T}v_{t}^{2}\eta_{t}^{2}\left\Vert \bd_{t}^{\ru}\right\Vert ^{2} & \leq\frac{10}{3}\max_{t\in\left[T\right]}v_{t}^{2}\eta_{t}^{2}\cm_{t}^{2}\ln\frac{3}{\delta}+8\sum_{t=1}^{T}\frac{\sigma_{\lar}^{\p}v_{t}^{2}\eta_{t}^{2}}{\cm_{t}^{\p-2}}\nonumber \\
 & \overset{\ln\frac{3}{\delta}\geq1}{\leq}\frac{10}{3}\max_{t\in\left[T\right]}v_{t}^{2}\eta_{t}^{2}\cm_{t}^{2}\ln^{2}\frac{3}{\delta}+8\sum_{t=1}^{T}\frac{\sigma_{\lar}^{\p}v_{t}^{2}\eta_{t}^{2}}{\cm_{t}^{\p-2}}.\label{eq:cvx-lip-res-2}
\end{align}

\item Lastly, there is
\begin{equation}
\sum_{t=1}^{T}v_{t}\eta_{t}\left\Vert \bd_{t}^{\rb}\right\Vert \overset{(\ref{eq:cvx-lip-res-b-norm})}{\leq}\sum_{t=1}^{T}\frac{\sigma_{\sma}\sigma_{\lar}^{\p-1}v_{t}\eta_{t}}{\cm_{t}^{\p-1}}+\frac{2^{\p-1}\sigma_{\lar}^{\p}Gv_{t}\eta_{t}}{\cm_{t}^{\p}}.\label{eq:cvx-lip-res-3}
\end{equation}

\end{itemize}

Combine (\ref{eq:cvx-lip-res-1}), (\ref{eq:cvx-lip-res-2}), and
(\ref{eq:cvx-lip-res-3}) to have with probability at least $1-\delta$,
\[
\lasthres_{T}^{\cvx}=8\max_{t\in\left[T\right]}\left(\sum_{s=1}^{t}M_{s}\right)^{2}+2\sum_{t=1}^{T}v_{t}^{2}\eta_{t}^{2}\left\Vert \bd_{t}^{\ru}\right\Vert ^{2}+4\left(\sum_{t=1}^{T}v_{t}\eta_{t}\left\Vert \bd_{t}^{\rb}\right\Vert \right)^{2}\leq\lasthc_{T}^{\cvx},
\]
where
\begin{align*}
\lasthc_{T}^{\cvx}\defeq & \frac{124}{9}\max_{t\in\left[T\right]}v_{t}^{2}\eta_{t}^{2}\cm_{t}^{2}\ln^{2}\frac{3}{\delta}+128\sum_{t=1}^{T}\left(\frac{\sigma_{\sma}^{\p}v_{t}^{2}\eta_{t}^{2}}{\cm_{t}^{\p-2}}+\frac{2^{\p-1}\sigma_{\lar}^{\p}G^{2}v_{t}^{2}\eta_{t}^{2}}{\cm_{t}^{\p}}\right)\ln\frac{3}{\delta}\\
 & +16\sum_{t=1}^{T}\frac{\sigma_{\lar}^{\p}v_{t}^{2}\eta_{t}^{2}}{\cm_{t}^{\p-2}}+4\left(\sum_{t=1}^{T}\frac{\sigma_{\sma}\sigma_{\lar}^{\p-1}v_{t}\eta_{t}}{\cm_{t}^{\p-1}}+\frac{2^{\p-1}\sigma_{\lar}^{\p}Gv_{t}\eta_{t}}{\cm_{t}^{\p}}\right)^{2}\\
= & \O\left(\max_{t\in\left[T\right]}v_{t}^{2}\eta_{t}^{2}\cm_{t}^{2}\ln^{2}\frac{3}{\delta}+\sum_{t=1}^{T}\frac{\sigma_{\lar}^{\p}v_{t}^{2}\eta_{t}^{2}}{\cm_{t}^{\p-2}}+\left(\sum_{t=1}^{T}\frac{\sigma_{\sma}\sigma_{\lar}^{\p-1}v_{t}\eta_{t}}{\cm_{t}^{\p-1}}+\frac{\sigma_{\lar}^{\p}Gv_{t}\eta_{t}}{\cm_{t}^{\p}}\right)^{2}\right),
\end{align*}
in which the second equation is by noticing that
\begin{align*}
 & \max_{t\in\left[T\right]}v_{t}^{2}\eta_{t}^{2}\cm_{t}^{2}\ln^{2}\frac{3}{\delta}+\left(\sum_{t=1}^{T}\frac{\sigma_{\sma}\sigma_{\lar}^{\p-1}v_{t}\eta_{t}}{\cm_{t}^{\p-1}}+\frac{\sigma_{\lar}^{\p}Gv_{t}\eta_{t}}{\cm_{t}^{\p}}\right)^{2}\\
\overset{(a)}{\geq} & 2\left(\max_{t\in\left[T\right]}v_{t}\eta_{t}\cm_{t}\right)\sum_{t=1}^{T}\left(\frac{\sigma_{\sma}\sigma_{\lar}^{\p-1}v_{t}\eta_{t}}{\cm_{t}^{\p-1}}+\frac{\sigma_{\lar}^{\p}Gv_{t}\eta_{t}}{\cm_{t}^{\p}}\right)\ln\frac{3}{\delta}\\
\geq & 2\sum_{t=1}^{T}\left(\frac{\sigma_{\sma}\sigma_{\lar}^{\p-1}v_{t}^{2}\eta_{t}^{2}}{\cm_{t}^{\p-2}}+\frac{\sigma_{\lar}^{\p}Gv_{t}^{2}\eta_{t}^{2}}{\cm_{t}^{\p-1}}\right)\ln\frac{3}{\delta}\\
\overset{(b)}{=} & \Omega\left(\sum_{t=1}^{T}\left(\frac{\sigma_{\sma}^{\p}v_{t}^{2}\eta_{t}^{2}}{\cm_{t}^{\p-2}}+\frac{\sigma_{\lar}^{\p}G^{2}v_{t}^{2}\eta_{t}^{2}}{\cm_{t}^{\p}}\right)\ln\frac{3}{\delta}\right),
\end{align*}
where $(a)$ is due to AM-GM inequality and $(b)$ is by $\sigma_{\lar}\geq\sigma_{\sma}$,
$\p\geq1$, and $\cm_{t}\overset{\text{Condition }\ref{enu:cvx-lip-res-2}}{\geq}2G$.
\end{proof}

\subsubsection{Smooth Objectives}

In this part, we consider smooth objectives (i.e., Assumption \ref{assu:smo}).
One may still think that, like the Lipschitz case, we only need to
bound the residual term $\lasthres_{T}^{\cvx}$ in Lemma \ref{lem:cvx-core}.
This is not a wrong thought. However, an obstacle will arise when
considering smooth objectives, that is, how to bound $\left\Vert \nabla f(\bx_{t})\right\Vert $
on the R.H.S. of inequalities in Lemma \ref{lem:clip-ineq}. Why is
this important? Because, to apply Freedman's inequality (Lemma \ref{lem:Freedman}),
the conditional variance of the martingale should be upper bounded
by a constant.

To handle this challenge, our solution is to first bound $\bar{\lasthres}_{T}^{\cvx}$
in the following Lemma \ref{lem:cvx-smo-res}, a variant of $\lasthres_{T}^{\cvx}$
under the condition that $\left\Vert \bx_{\star}-\bx_{t}\right\Vert $
is uniformly bounded for all $t\in\left[T\right]$. Then, Lemma \ref{lem:cvx-smo-distance}
provides a sufficient condition to ensure that $\left\Vert \bx_{\star}-\bx_{t}\right\Vert $
is uniformly bounded.
\begin{lem}
\label{lem:cvx-smo-res}Under Assumptions \ref{assu:minimizer}, \ref{assu:smo},
\ref{assu:heavy}, \ref{assu:G1}, and the following two conditions:
\begin{enumerate}[leftmargin=*]
\item \label{enu:cvx-smo-res-1}$\eta_{t}$ is deterministic for all $t\in\left[T\right]$.
\item \label{enu:cvx-smo-res-2}$\cm_{t}=\max\left\{ 2G_{t},\cm\right\} $
for all $t\in\left[T\right]$, where $G_{t}=L\left\Vert \bx_{t}-\bx_{1}\right\Vert +G_{1}$
and $\cm\geq0$ is a constant.
\end{enumerate}
We have for any $\delta\in\left(0,1\right]$, with probability at
least $1-\delta$, $\bar{\lasthres}_{T}^{\cvx}\leq\bar{\lasthc}_{T}^{\cvx}$,
where
\begin{align*}
\bar{\lasthres}_{T}^{\cvx} & \defeq8\max_{t\in\left[T\right]}\left(\sum_{s=1}^{t}M_{s}\zeta_{s}\right)^{2}+2\sum_{t=1}^{T}v_{t}^{2}\eta_{t}^{2}\left\Vert \bd_{t}^{\ru}\right\Vert ^{2}\zeta_{t}^{2}+4\left(\sum_{t=1}^{T}v_{t}\eta_{t}\left\Vert \bd_{t}^{\rb}\right\Vert \zeta_{t}\right)^{2},\\
\zeta_{t} & \defeq\1\left[\left\Vert \bx_{\star}-\bx_{t}\right\Vert \leq\bar{D}\right],\forall t\in\left[T\right],
\end{align*}
$\bar{D}\geq0$ can be any constant, and $\bar{\lasthc}_{T}^{\cvx}$
is a constant in the order of
\[
\O\left(\max_{t\in\left[T\right]}v_{t}^{2}\eta_{t}^{2}\rho^{2}\ln^{2}\frac{3}{\delta}+\sum_{t=1}^{T}\frac{\sigma_{\lar}^{\p}v_{t}^{2}\eta_{t}^{2}}{\cm^{\p-2}}+\left(\sum_{t=1}^{T}\frac{\sigma_{\sma}\sigma_{\lar}^{\p-1}v_{t}\eta_{t}}{\cm^{\p-1}}+\min_{\beta\in\left[0,\frac{1}{2}\right]}\frac{\sigma_{\lar}^{\p}\bar{G}^{1-\beta\p}v_{t}\eta_{t}}{\cm^{(1-\beta)\p}}\right)^{2}\right),
\]
in which $\bar{G}=L(\bar{D}+D)+G_{1}$ and $\rho=\max\left\{ 2\bar{G},\cm\right\} $.
\end{lem}
\begin{proof}
By $L$-smoothness (Assumption \ref{assu:smo}) and $\left\Vert \nabla f(\bx_{1})\right\Vert \leq G_{1}$
(Assumption \ref{assu:G1}), we have
\begin{equation}
\left\Vert \nabla f(\bx_{t})\right\Vert \leq\left\Vert \nabla f(\bx_{t})-\nabla f(\bx_{1})\right\Vert +\left\Vert \nabla f(\bx_{1})\right\Vert \leq L\left\Vert \bx_{t}-\bx_{1}\right\Vert +G_{1}=G_{t},\forall t\in\left[T\right].\label{eq:cvx-smo-res-Gt}
\end{equation}
Moreover, under Assumption \ref{assu:minimizer}, we note that $G_{t}\leq L\left\Vert \bx_{\star}-\bx_{t}\right\Vert +L\left\Vert \bx_{\star}-\bx_{1}\right\Vert +G_{1}$,
implying that
\begin{equation}
\left\Vert \nabla f(\bx_{t})\right\Vert \zeta_{t}\overset{(\ref{eq:cvx-smo-res-Gt})}{\leq}G_{t}\zeta_{t}\leq L(\bar{D}+D)+G_{1}\defeq\bar{G},\forall t\in\left[T\right],\label{eq:cvx-smo-res-Gbar}
\end{equation}
and
\begin{equation}
\cm_{t}\zeta_{t}=\max\left\{ 2G_{t},\cm\right\} \zeta_{t}\overset{\zeta_{t}\leq1}{\leq}\max\left\{ 2G_{t}\zeta_{t},\cm\right\} \overset{(\ref{eq:cvx-smo-res-Gbar})}{\leq}\max\left\{ 2\bar{G},\cm\right\} \defeq\rho,\forall t\in\left[T\right].\label{eq:cvx-smo-res-rho}
\end{equation}
We observe that
\begin{equation}
2\left\Vert \nabla f(\bx_{t})\right\Vert \overset{(\ref{eq:cvx-smo-res-Gt})}{\leq}2G_{t}\leq\max\left\{ 2G_{t},\cm\right\} =\cm_{t},\forall t\in\left[T\right],\label{eq:cvx-smo-res-tau}
\end{equation}
which implies that $\chi_{t}$ in Lemma \ref{lem:clip-ineq} equals
$1$ for all $t\in\left[T\right]$. Then Lemma \ref{lem:clip-ineq}
and (\ref{eq:cvx-smo-res-Gt}) together yield that the following collection
of inequalities holds for any $t\in\left[T\right]$:
\begin{align}
\left\Vert \bd_{t}^{\ru}\right\Vert  & \leq2\cm_{t},\label{eq:cvx-smo-res-u-norm}\\
\E\left[\left\Vert \bd_{t}^{\ru}\right\Vert ^{2}\mid\F_{t-1}\right] & \leq\frac{4\sigma_{\lar}^{\p}}{\cm_{t}^{\p-2}},\label{eq:cvx-smo-res-u-var}\\
\left\Vert \E\left[\bd_{t}^{\ru}\left(\bd_{t}^{\ru}\right)^{\top}\mid\F_{t-1}\right]\right\Vert _{\op} & \leq\frac{4\sigma_{\sma}^{\p}}{\cm_{t}^{\p-2}}+\frac{2^{\p+1}\sigma_{\lar}^{\p}G_{t}^{2}}{\cm_{t}^{\p}},\label{eq:cvx-smo-res-u-op}\\
\left\Vert \bd_{t}^{\rb}\right\Vert  & \leq\frac{\sigma_{\sma}\sigma_{\lar}^{\p-1}}{\cm_{t}^{\p-1}}+\frac{2^{\p-1}\sigma_{\lar}^{\p}G_{t}}{\cm_{t}^{\p}}.\label{eq:cvx-smo-res-b-norm}
\end{align}

Now, we are ready to bound $\bar{\lasthres}_{T}^{\cvx}$.

\begin{itemize}[leftmargin=*]

\item We first observe that the sequence $\left\{ v_{t}\right\} _{t=1}^{T}$
is now deterministic by Condition \ref{enu:cvx-smo-res-1} and its
definition in (\ref{eq:cvx-v-def}), which further implies that $\bz_{t}\in\F_{t-1}$
and $R_{t}\in\F_{t-1}$ (see (\ref{eq:cvx-z-equiv}) and (\ref{eq:cvx-R-M-def})
for their definitions). Therefore, using $\zeta_{t}\in\F_{t-1}$,
we know
\[
\E\left[M_{t}\zeta_{t}\mid\F_{t-1}\right]\overset{(\ref{eq:cvx-R-M-def})}{=}\left\langle v_{t-1}\eta_{t}\E\left[\bd_{t}^{\ru}\mid\F_{t-1}\right],\frac{v_{t-1}\left(\bz_{t-1}-\bx_{t}\right)}{R_{t}}\right\rangle \zeta_{t}=0,\forall t\in\left[T\right],
\]
meaning that $\left\{ M_{t}\zeta_{t}\right\} _{t=1}^{T}$ is a real-valued
MDS. By Cauchy-Schwarz inequality, we can bound
\begin{align*}
\left|M_{t}\zeta_{t}\right| & \overset{(\ref{eq:cvx-R-M-def})}{=}\left|\left\langle v_{t-1}\eta_{t}\bd_{t}^{\ru},\frac{v_{t-1}\left(\bz_{t-1}-\bx_{t}\right)}{R_{t}}\right\rangle \right|\zeta_{t}\leq v_{t-1}\eta_{t}\left\Vert \bd_{t}^{\ru}\right\Vert \cdot\frac{v_{t-1}\left\Vert \bz_{t-1}-\bx_{t}\right\Vert }{R_{t}}\cdot\zeta_{t}\\
 & \overset{(\ref{eq:cvx-R-M-def}),v_{t-1}\leq v_{t}}{\leq}v_{t}\eta_{t}\left\Vert \bd_{t}^{\ru}\right\Vert \zeta_{t}\overset{(\ref{eq:cvx-smo-res-u-norm})}{\leq}2v_{t}\eta_{t}\cm_{t}\zeta_{t}\overset{(\ref{eq:cvx-smo-res-rho})}{\leq}2v_{t}\eta_{t}\rho\leq2\max_{t\in\left[T\right]}v_{t}\eta_{t}\rho.
\end{align*}
In addition, the conditional variance of $M_{t}\zeta_{t}$ can be
controlled by
\begin{align*}
\E\left[\left(M_{t}\zeta_{t}\right)^{2}\mid\F_{t-1}\right] & =\left(\frac{v_{t-1}\left(\bz_{t-1}-\bx_{t}\right)}{R_{t}}\right)^{\top}v_{t-1}^{2}\eta_{t}^{2}\E\left[\bd_{t}^{\ru}\left(\bd_{t}^{\ru}\right)^{\top}\mid\F_{t-1}\right]\frac{v_{t-1}\left(\bz_{t-1}-\bx_{t}\right)}{R_{t}}\zeta_{t}^{2}\\
 & \overset{(\ref{eq:cvx-R-M-def}),v_{t-1}\leq v_{t}}{\leq}v_{t}^{2}\eta_{t}^{2}\left\Vert \E\left[\bd_{t}^{\ru}\left(\bd_{t}^{\ru}\right)^{\top}\mid\F_{t-1}\right]\right\Vert _{\op}\zeta_{t}^{2}\overset{(\ref{eq:cvx-smo-res-u-op})}{\leq}\frac{4\sigma_{\sma}^{\p}v_{t}^{2}\eta_{t}^{2}\zeta_{t}^{2}}{\cm_{t}^{\p-2}}+\frac{2^{\p+1}\sigma_{\lar}^{\p}G_{t}^{2}v_{t}^{2}\eta_{t}^{2}\zeta_{t}^{2}}{\cm_{t}^{\p}}\\
 & \leq\frac{4\sigma_{\sma}^{\p}v_{t}^{2}\eta_{t}^{2}\rho}{\cm^{\p-1}}+\min_{\beta\in\left[0,\frac{1}{2}\right]}\frac{2^{\p+1}\sigma_{\lar}^{\p}\bar{G}^{2-\beta\p}v_{t}^{2}\eta_{t}^{2}}{2^{\beta\p}\cm^{(1-\beta)\p}},
\end{align*}
where the last step is due to
\[
\frac{\zeta_{t}^{2}}{\cm_{t}^{\p-2}}=\frac{\cm_{t}\zeta_{t}^{2}}{\cm_{t}^{\p-1}}\overset{\cm_{t}\geq\cm,\p\geq1}{\leq}\frac{\cm_{t}\zeta_{t}^{2}}{\cm^{\p-1}}\overset{\zeta_{t}\leq1,(\ref{eq:cvx-smo-res-rho})}{\leq}\frac{\rho}{\cm^{\p-1}},
\]
and, for any $\beta\in\left[0,\frac{1}{2}\right]$,
\[
\frac{G_{t}^{2}\zeta_{t}^{2}}{\cm_{t}^{\p}}\overset{\cm_{t}\geq\cm,(\ref{eq:cvx-smo-res-tau})}{\leq}\frac{G_{t}^{2-\beta\p}\zeta_{t}^{2}}{2^{\beta\p}\cm^{(1-\beta)\p}}\overset{\zeta_{t}\leq1,\beta\p\leq2,(\ref{eq:cvx-smo-res-Gbar})}{\leq}\frac{\bar{G}^{2-\beta\p}}{2^{\beta\p}\cm^{(1-\beta)\p}}.
\]
Therefore, by Freedman's inequality (Lemma \ref{lem:Freedman}), with
probability at least $1-2\delta/3$,
\[
\left|\sum_{s=1}^{t}M_{s}\zeta_{s}\right|\leq\frac{2}{3}\max_{s\in\left[T\right]}v_{s}\eta_{s}\rho\ln\frac{3}{\delta}+\sqrt{8\sum_{s=1}^{T}\left(\frac{\sigma_{\sma}^{\p}v_{s}^{2}\eta_{s}^{2}\rho}{\cm^{\p-1}}+\min_{\beta\in\left[0,\frac{1}{2}\right]}\frac{2^{\p-1}\sigma_{\lar}^{\p}\bar{G}^{2-\beta\p}v_{s}^{2}\eta_{s}^{2}}{2^{\beta\p}\cm^{(1-\beta)\p}}\right)\ln\frac{3}{\delta}},
\]
which implies that
\begin{equation}
\max_{t\in\left[T\right]}\left(\sum_{s=1}^{t}M_{s}\zeta_{s}\right)^{2}\leq\frac{8}{9}\max_{t\in\left[T\right]}v_{t}^{2}\eta_{t}^{2}\rho^{2}\ln^{2}\frac{3}{\delta}+16\sum_{t=1}^{T}\left(\frac{\sigma_{\sma}^{\p}v_{t}^{2}\eta_{t}^{2}\rho}{\cm^{\p-1}}+\min_{\beta\in\left[0,\frac{1}{2}\right]}\frac{2^{\p-1}\sigma_{\lar}^{\p}\bar{G}^{2-\beta\p}v_{t}^{2}\eta_{t}^{2}}{2^{\beta\p}\cm^{(1-\beta)\p}}\right)\ln\frac{3}{\delta}.\label{eq:cvx-smo-res-1}
\end{equation}

\item Next, note that $\left\{ v_{t}\eta_{t}\left\Vert \bd_{t}^{\ru}\right\Vert \zeta_{t}\right\} _{t=1}^{T}$
is a sequence of random variables satisfying
\[
v_{t}\eta_{t}\left\Vert \bd_{t}^{\ru}\right\Vert \zeta_{t}\overset{(\ref{eq:cvx-smo-res-u-norm})}{\leq}2v_{t}\eta_{t}\cm_{t}\zeta_{t}\overset{(\ref{eq:cvx-smo-res-rho})}{\leq}2v_{t}\eta_{t}\rho\leq2\max_{t\in\left[T\right]}v_{t}\eta_{t}\rho,
\]
and
\[
\E\left[v_{t}^{2}\eta_{t}^{2}\left\Vert \bd_{t}^{\ru}\right\Vert ^{2}\zeta_{t}^{2}\mid\F_{t-1}\right]\overset{(\ref{eq:cvx-smo-res-u-var})}{\leq}\frac{4\sigma_{\lar}^{\p}v_{t}^{2}\eta_{t}^{2}\zeta_{t}^{2}}{\cm_{t}^{\p-2}}\leq\frac{4\sigma_{\lar}^{\p}v_{t}^{2}\eta_{t}^{2}}{\cm^{\p-2}}+\min_{\beta\in\left[0,\frac{1}{2}\right]}\frac{4\sigma_{\lar}^{\p}(2\bar{G})^{2-\beta\p}v_{t}^{2}\eta_{t}^{2}}{\cm^{(1-\beta)\p}},
\]
where the second step is due to, for any $\beta\in\left[0,\frac{1}{2}\right]$,
\begin{align*}
\frac{\zeta_{t}^{2}}{\cm_{t}^{\p-2}} & =\frac{\cm_{t}^{2}\zeta_{t}^{2}}{\cm_{t}^{\p}}\overset{\text{Condition }\ref{enu:cvx-smo-res-2}}{\leq}\frac{\left(\cm^{2}+4G_{t}^{2}\right)\zeta_{t}^{2}}{\cm_{t}^{\p}}\overset{\cm_{t}\geq\cm,\zeta_{t}\leq1}{\leq}\frac{1}{\cm^{\p-2}}+\frac{4G_{t}^{2}\zeta_{t}^{2}}{\cm_{t}^{\p}}\\
 & \overset{\cm_{t}\geq\cm,(\ref{eq:cvx-smo-res-tau})}{\leq}\frac{1}{\cm^{\p-2}}+\frac{(2G_{t})^{2-\beta\p}\zeta_{t}^{2}}{\cm^{(1-\beta)\p}}\overset{\zeta_{t}\leq1,\beta\p\leq2,(\ref{eq:cvx-smo-res-Gbar})}{\leq}\frac{1}{\cm^{\p-2}}+\frac{(2\bar{G})^{2-\beta\p}}{\cm^{(1-\beta)\p}}.
\end{align*}
Therefore, by Lemma \ref{lem:Freedman-square}, with probability at
least $1-\delta/3$,
\begin{align}
\sum_{t=1}^{T}v_{t}^{2}\eta_{t}^{2}\left\Vert \bd_{t}^{\ru}\right\Vert ^{2}\zeta_{t}^{2} & \leq\frac{10}{3}\max_{t\in\left[T\right]}v_{t}^{2}\eta_{t}^{2}\rho^{2}\ln\frac{3}{\delta}+8\sum_{t=1}^{T}\left(\frac{\sigma_{\lar}^{\p}v_{t}^{2}\eta_{t}^{2}}{\cm^{\p-2}}+\min_{\beta\in\left[0,\frac{1}{2}\right]}\frac{\sigma_{\lar}^{\p}(2\bar{G})^{2-\beta\p}v_{t}^{2}\eta_{t}^{2}}{\cm^{(1-\beta)\p}}\right)\nonumber \\
 & \overset{\ln\frac{3}{\delta}\geq1}{\leq}\frac{10}{3}\max_{t\in\left[T\right]}v_{t}^{2}\eta_{t}^{2}\rho^{2}\ln^{2}\frac{3}{\delta}+8\sum_{t=1}^{T}\left(\frac{\sigma_{\lar}^{\p}v_{t}^{2}\eta_{t}^{2}}{\cm^{\p-2}}+\min_{\beta\in\left[0,\frac{1}{2}\right]}\frac{\sigma_{\lar}^{\p}(2\bar{G})^{2-\beta\p}v_{t}^{2}\eta_{t}^{2}}{\cm^{(1-\beta)\p}}\right).\label{eq:cvx-smo-res-2}
\end{align}

\item Lastly, there is
\[
v_{t}\eta_{t}\left\Vert \bd_{t}^{\rb}\right\Vert \zeta_{t}\overset{(\ref{eq:cvx-smo-res-b-norm})}{\leq}\frac{\sigma_{\sma}\sigma_{\lar}^{\p-1}v_{t}\eta_{t}\zeta_{t}}{\cm_{t}^{\p-1}}+\frac{2^{\p-1}\sigma_{\lar}^{\p}G_{t}v_{t}\eta_{t}\zeta_{t}}{\cm_{t}^{\p}}\leq\frac{\sigma_{\sma}\sigma_{\lar}^{\p-1}v_{t}\eta_{t}}{\cm^{\p-1}}+\min_{\beta\in\left[0,\frac{1}{2}\right]}\frac{2^{\p-1}\sigma_{\lar}^{\p}\bar{G}^{1-\beta\p}v_{t}\eta_{t}}{2^{\beta\p}\cm^{(1-\beta)\p}},
\]
where the second step is due to
\[
\frac{\zeta_{t}}{\cm_{t}^{\p-1}}\overset{\cm_{t}\geq\cm,\p\geq1}{\leq}\frac{\zeta_{t}}{\cm^{\p-1}}\overset{\zeta_{t}\leq1}{\leq}\frac{1}{\cm^{\p-1}},
\]
and, for any $\beta\in\left[0,\frac{1}{2}\right]$,
\[
\frac{G_{t}\zeta_{t}}{\cm_{t}^{\p}}\overset{\cm_{t}\geq\cm,(\ref{eq:cvx-smo-res-tau})}{\leq}\frac{G_{t}^{1-\beta\p}\zeta_{t}}{2^{\beta\p}\cm^{(1-\beta)\p}}\overset{\zeta_{t}\leq1,\beta\p\leq1,(\ref{eq:cvx-smo-res-Gbar})}{\leq}\frac{\bar{G}^{1-\beta\p}}{2^{\beta\p}\cm^{(1-\beta)\p}}.
\]
Hence, we obtain
\begin{equation}
\sum_{t=1}^{T}v_{t}\eta_{t}\left\Vert \bd_{t}^{\rb}\right\Vert \zeta_{t}\leq\sum_{t=1}^{T}\frac{\sigma_{\sma}\sigma_{\lar}^{\p-1}v_{t}\eta_{t}}{\cm^{\p-1}}+\min_{\beta\in\left[0,\frac{1}{2}\right]}\frac{2^{\p-1}\sigma_{\lar}^{\p}\bar{G}^{1-\beta\p}v_{t}\eta_{t}}{2^{\beta\p}\cm^{(1-\beta)\p}}.\label{eq:cvx-smo-res-3}
\end{equation}

\end{itemize}

Combine (\ref{eq:cvx-smo-res-1}), (\ref{eq:cvx-smo-res-2}), and
(\ref{eq:cvx-smo-res-3}) to have with probability at least $1-\delta$,
\[
\bar{\lasthres}_{T}^{\cvx}=8\max_{t\in\left[T\right]}\left(\sum_{s=1}^{t}M_{s}\zeta_{s}\right)^{2}+2\sum_{t=1}^{T}v_{t}^{2}\eta_{t}^{2}\left\Vert \bd_{t}^{\ru}\right\Vert ^{2}\zeta_{t}^{2}+4\left(\sum_{t=1}^{T}v_{t}\eta_{t}\left\Vert \bd_{t}^{\rb}\right\Vert \zeta_{t}\right)^{2}\leq\bar{\lasthc}_{T}^{\cvx},
\]
where
\begin{align}
\bar{\lasthc}_{T}^{\cvx}\defeq & \frac{124}{9}\max_{t\in\left[T\right]}v_{t}^{2}\eta_{t}^{2}\rho^{2}\ln^{2}\frac{3}{\delta}+128\sum_{t=1}^{T}\left(\frac{\sigma_{\sma}^{\p}v_{t}^{2}\eta_{t}^{2}\rho}{\cm^{\p-1}}+\min_{\beta\in\left[0,\frac{1}{2}\right]}\frac{2^{\p-1}\sigma_{\lar}^{\p}\bar{G}^{2-\beta\p}v_{t}^{2}\eta_{t}^{2}}{2^{\beta\p}\cm^{(1-\beta)\p}}\right)\ln\frac{3}{\delta}\nonumber \\
 & +16\sum_{t=1}^{T}\left(\frac{\sigma_{\lar}^{\p}v_{t}^{2}\eta_{t}^{2}}{\cm^{\p-2}}+\min_{\beta\in\left[0,\frac{1}{2}\right]}\frac{\sigma_{\lar}^{\p}(2\bar{G})^{2-\beta\p}v_{t}^{2}\eta_{t}^{2}}{\cm^{(1-\beta)\p}}\right)\nonumber \\
 & +4\left(\sum_{t=1}^{T}\frac{\sigma_{\sma}\sigma_{\lar}^{\p-1}v_{t}\eta_{t}}{\cm^{\p-1}}+\min_{\beta\in\left[0,\frac{1}{2}\right]}\frac{2^{\p-1}\sigma_{\lar}^{\p}\bar{G}^{1-\beta\p}v_{t}\eta_{t}}{2^{\beta\p}\cm^{(1-\beta)\p}}\right)^{2}\nonumber \\
= & \O\left(\max_{t\in\left[T\right]}v_{t}^{2}\eta_{t}^{2}\rho^{2}\ln^{2}\frac{3}{\delta}+\sum_{t=1}^{T}\frac{\sigma_{\lar}^{\p}v_{t}^{2}\eta_{t}^{2}}{\cm^{\p-2}}+\left(\sum_{t=1}^{T}\frac{\sigma_{\sma}\sigma_{\lar}^{\p-1}v_{t}\eta_{t}}{\cm^{\p-1}}+\min_{\beta\in\left[0,\frac{1}{2}\right]}\frac{\sigma_{\lar}^{\p}\bar{G}^{1-\beta\p}v_{t}\eta_{t}}{\cm^{(1-\beta)\p}}\right)^{2}\right),\label{eq:cvx-smo-res-4}
\end{align}
in which the second equation is by noticing that
\begin{align*}
 & \max_{t\in\left[T\right]}v_{t}^{2}\eta_{t}^{2}\rho^{2}\ln^{2}\frac{3}{\delta}+\left(\sum_{t=1}^{T}\frac{\sigma_{\sma}\sigma_{\lar}^{\p-1}v_{t}\eta_{t}}{\cm^{\p-1}}+\min_{\beta\in\left[0,\frac{1}{2}\right]}\frac{2^{\p-1}\sigma_{\lar}^{\p}\bar{G}^{1-\beta\p}v_{t}\eta_{t}}{2^{\beta\p}\cm^{(1-\beta)\p}}\right)^{2}\\
\overset{(a)}{\geq} & 2\left(\max_{t\in\left[T\right]}v_{t}\eta_{t}\rho\right)\sum_{t=1}^{T}\left(\frac{\sigma_{\sma}\sigma_{\lar}^{\p-1}v_{t}\eta_{t}}{\cm^{\p-1}}+\min_{\beta\in\left[0,\frac{1}{2}\right]}\frac{2^{\p-1}\sigma_{\lar}^{\p}\bar{G}^{1-\beta\p}v_{t}\eta_{t}}{2^{\beta\p}\cm^{(1-\beta)\p}}\right)\ln\frac{3}{\delta}\\
\geq & 2\sum_{t=1}^{T}\left(\frac{\sigma_{\sma}\sigma_{\lar}^{\p-1}v_{t}^{2}\eta_{t}^{2}\rho}{\cm^{\p-1}}+\min_{\beta\in\left[0,\frac{1}{2}\right]}\frac{2^{\p-1}\sigma_{\lar}^{\p}\bar{G}^{1-\beta\p}v_{t}^{2}\eta_{t}^{2}\rho}{2^{\beta\p}\cm^{(1-\beta)\p}}\right)\ln\frac{3}{\delta}\\
\overset{(b)}{=} & \Omega\left(\sum_{t=1}^{T}\left(\frac{\sigma_{\sma}^{\p}v_{t}^{2}\eta_{t}^{2}\rho}{\cm^{\p-1}}+\min_{\beta\in\left[0,\frac{1}{2}\right]}\frac{2^{\p-1}\sigma_{\lar}^{\p}\bar{G}^{2-\beta\p}v_{t}^{2}\eta_{t}^{2}}{2^{\beta\p}\cm^{(1-\beta)\p}}\right)\ln\frac{3}{\delta}+\sum_{t=1}^{T}\min_{\beta\in\left[0,\frac{1}{2}\right]}\frac{\sigma_{\lar}^{\p}(2\bar{G})^{2-\beta\p}v_{t}^{2}\eta_{t}^{2}}{\cm^{(1-\beta)\p}}\right),
\end{align*}
where $(a)$ is due to AM-GM inequality and $(b)$ is by $\sigma_{\lar}\geq\sigma_{\sma}$,
$\p\geq1$, $\rho\overset{(\ref{eq:cvx-smo-res-rho})}{\geq}2\bar{G}$,
and $\ln\frac{3}{\delta}\geq1$.
\end{proof}

To find a sufficient condition for uniformly upper bounded $\left\Vert \bx_{\star}-\bx_{t}\right\Vert $,
we employ the idea introduced before in Section \ref{sec:idea}.
\begin{lem}
\label{lem:cvx-smo-distance}Under Assumptions \ref{assu:minimizer},
\ref{assu:obj} (with $\mu=0$), \ref{assu:smo}, \ref{assu:heavy},
\ref{assu:G1}, and the following two conditions:
\begin{enumerate}[leftmargin=*]
\item \label{enu:cvx-smo-distance-1}$\eta_{t}$ is deterministic and satisfies
$0<\eta_{t}\leq\frac{1}{2L}$ for all $t\in\left[T\right]$.
\item \label{enu:cvx-smo-distance-2}$\cm_{t}=\max\left\{ 2G_{t},\cm\right\} $
for all $t\in\left[T\right]$, where $G_{t}=L\left\Vert \bx_{t}-\bx_{1}\right\Vert +G_{1}$
and $\cm\geq0$ is a constant.
\end{enumerate}
$\ClipSGD$ (Algorithm \ref{alg:Clipped-SGD}) guarantees that, for
any $\delta\in\left(0,1\right]$, with probability at least $1-\delta$,
\[
\left\Vert \bx_{\star}-\bx_{t}\right\Vert \leq\bar{D},\forall t\in\left[T+1\right],
\]
where $\bar{D}>0$ is any constant satisfying (\ref{eq:cvx-smo-distance-ineq}).
\end{lem}
\begin{proof}
Let $\bar{D}>0$ be any constant. We first upper bound the following
term:
\begin{equation}
\bar{\hres}_{T}^{\cvx}\defeq8\max_{t\in\left[T\right]}\left(\sum_{s=1}^{t}N_{s}\zeta_{s}\right)^{2}+2\sum_{t=1}^{T}\eta_{t}^{2}\left\Vert \bd_{t}^{\ru}\right\Vert ^{2}\zeta_{t}^{2}+4\left(\sum_{t=1}^{T}\eta_{t}\left\Vert \bd_{t}^{\rb}\right\Vert \zeta_{t}\right)^{2},\label{eq:cvx-smo-distance-1}
\end{equation}
where $\zeta_{t}\defeq\1\left[\left\Vert \bx_{\star}-\bx_{t}\right\Vert \leq\bar{D}\right],\forall t\in\left[T+1\right]$.
Following similar steps of proving (\ref{eq:cvx-smo-res-4}), we can
show that, for any $\delta\in\left(0,1\right]$,
\begin{align}
\Pr\left[E\right] & \geq1-\delta,\text{ where }E\defeq\left\{ \bar{\hres}_{T}^{\cvx}\leq\bar{\hc}_{T}^{\cvx}\right\} \text{ and }\label{eq:cvx-smo-distance-2}\\
\bar{\hc}_{T}^{\cvx} & \defeq\constant\left[\max_{t\in\left[T\right]}\eta_{t}^{2}\rho^{2}\ln^{2}\frac{3}{\delta}+\sum_{t=1}^{T}\sigma_{\lar}^{\p}\eta_{t}^{2}\cm^{2-\p}+\left(\sum_{t=1}^{T}\frac{\sigma_{\sma}\sigma_{\lar}^{\p-1}\eta_{t}}{\cm^{\p-1}}+\min_{\beta\in\left[0,\frac{1}{2}\right]}\frac{\sigma_{\lar}^{\p}\bar{G}^{1-\beta\p}\eta_{t}}{\cm^{(1-\beta)\p}}\right)^{2}\right],\nonumber 
\end{align}
in which $\bar{G}=L(\bar{D}+D)+G_{1}$, $\rho=\max\left\{ 2\bar{G},\cm\right\} $,
and $\constant\geq1$ is a numerical constant.

Next, suppose $\bar{D}$ satisfies the following inequality
\begin{equation}
2D^{2}+4\bar{\hc}_{T}^{\cvx}\leq\bar{D}^{2}.\label{eq:cvx-smo-distance-ineq}
\end{equation}
We will show
\[
\Pr\left[\left\Vert \bx_{\star}-\bx_{t}\right\Vert \leq\bar{D},\forall t\in\left[T+1\right]\right]\geq1-\delta.
\]
Let us define a group of events $E_{t}\defeq\left\{ \zeta_{s}=1,\forall s\in\left[t\right]\right\} ,\forall t\in\left[T+1\right]$.
We claim that there is always
\begin{equation}
E\cap E_{t}\subseteq\left\{ \zeta_{t+1}=1\right\} ,\forall t\in\left[T\right].\label{eq:cvx-smo-distance-3}
\end{equation}
Note that, by Condition \ref{enu:cvx-smo-distance-1} and Lemma \ref{lem:cvx-core}
($\omega=0$ now), every sample path satisfies
\[
\left\Vert \bx_{\star}-\bx_{t+1}\right\Vert ^{2}\leq2D^{2}+4\left[8\max_{t^{\prime}\in\left[t\right]}\left(\sum_{s=1}^{t^{\prime}}N_{s}\right)^{2}+2\sum_{s=1}^{t}\eta_{s}^{2}\left\Vert \bd_{s}^{\ru}\right\Vert ^{2}+4\left(\sum_{s=1}^{t}\eta_{s}\left\Vert \bd_{s}^{\rb}\right\Vert \right)^{2}\right],\forall t\in\left[T\right].
\]
Given a fixed $t\in\left[T\right]$ and a sample path in $E\cap E_{t}$,
we know $\zeta_{s}=1,\forall s\in\left[t\right]$ and $\bar{\hres}_{T}^{\cvx}\leq\bar{\hc}_{T}^{\cvx}$,
implying that
\begin{align*}
\left\Vert \bx_{\star}-\bx_{t+1}\right\Vert ^{2} & \leq2D^{2}+4\left[8\max_{t^{\prime}\in\left[t\right]}\left(\sum_{s=1}^{t^{\prime}}N_{s}\right)^{2}+2\sum_{s=1}^{t}\eta_{s}^{2}\left\Vert \bd_{s}^{\ru}\right\Vert ^{2}+4\left(\sum_{s=1}^{t}\eta_{s}\left\Vert \bd_{s}^{\rb}\right\Vert \right)^{2}\right]\\
 & =2D^{2}+4\left[8\max_{t^{\prime}\in\left[t\right]}\left(\sum_{s=1}^{t^{\prime}}N_{s}\zeta_{s}\right)^{2}+2\sum_{s=1}^{t}\eta_{s}^{2}\left\Vert \bd_{s}^{\ru}\right\Vert ^{2}\zeta_{s}^{2}+4\left(\sum_{s=1}^{t}\eta_{s}\left\Vert \bd_{s}^{\rb}\right\Vert \zeta_{s}\right)^{2}\right]\\
 & \leq2D^{2}+4\left[8\max_{t\in\left[T\right]}\left(\sum_{s=1}^{t}N_{s}\zeta_{s}\right)^{2}+2\sum_{t=1}^{T}\eta_{t}^{2}\left\Vert \bd_{t}^{\ru}\right\Vert ^{2}\zeta_{t}^{2}+4\left(\sum_{t=1}^{T}\eta_{t}\left\Vert \bd_{t}^{\rb}\right\Vert \zeta_{t}\right)^{2}\right]\\
 & \overset{(\ref{eq:cvx-smo-distance-1})}{=}2D^{2}+4\bar{\hres}_{T}^{\cvx}\leq2D^{2}+4\bar{\hc}_{T}^{\cvx}\overset{(\ref{eq:cvx-smo-distance-ineq})}{\leq}\bar{D}^{2}\Rightarrow\zeta_{t+1}=1.
\end{align*}
Therefore, (\ref{eq:cvx-smo-distance-3}) holds for all $t\in\left[T\right]$,
which implies that
\[
E\cap E_{t+1}=E\cap E_{t}\cap\left\{ \zeta_{t+1}=1\right\} =E\cap E_{t},\forall t\in\left[T\right].
\]
Subsequently, we obtain
\begin{equation}
E_{T+1}\supseteq E\cap E_{T+1}=E\cap E_{1}=E,\label{eq:cvx-smo-distance-4}
\end{equation}
where the last step holds because $E_{1}=\left\{ \zeta_{1}=1\right\} =\left\{ \left\Vert \bx_{\star}-\bx_{1}\right\Vert \leq\bar{D}\right\} $
is the set that contains all possible sample paths, since $\left\Vert \bx_{\star}-\bx_{1}\right\Vert =D\leq\sqrt{2}D\overset{(\ref{eq:cvx-smo-distance-ineq})}{\leq}\bar{D}.$
Finally, we have
\[
\Pr\left[\left\Vert \bx_{\star}-\bx_{t}\right\Vert \leq\bar{D},\forall t\in\left[T+1\right]\right]=\Pr\left[E_{T+1}\right]\overset{(\ref{eq:cvx-smo-distance-4})}{\geq}\Pr\left[E\right]\overset{(\ref{eq:cvx-smo-distance-2})}{\geq}1-\delta.
\]
\end{proof}

\subsection{Lemmas for Strongly Star-Convex Optimization}

In this subsection, we will prove the lemmas used in strongly star-convex
optimization, i.e., $\mu>0$ in Assumption \ref{assu:obj}.

Before proceeding with the analysis, we need some new notation. Inspired
by \citep{ICLR2026_bd5b846a}, we introduce an auxiliary sequence
$\left\{ \gamma_{t}\right\} _{t=1}^{T}$ defined as follows
\begin{equation}
\gamma_{t}\defeq\prod_{s=2}^{t}\frac{1+\mu\eta_{s-1}}{1+\mu\eta_{s}/2},\forall t\in\left[T\right].\label{eq:str-gamma-def}
\end{equation}
With $\left\{ \gamma_{t}\right\} _{t=1}^{T}$ in hand, we redefine
$\left\{ D_{t}\right\} _{t=1}^{T}$ and $\left\{ N_{t}\right\} _{t=1}^{T}$
in the following
\begin{align}
D_{t} & \defeq\max_{s\in\left[t\right]}\sqrt{\gamma_{s}(1+\mu\eta_{s}/2)}\left\Vert \bx_{\star}-\bx_{s}\right\Vert ,\label{eq:str-D-def}\\
N_{t} & \defeq\left\langle \sqrt{\frac{\gamma_{t}}{1+\mu\eta_{t}/2}}\eta_{t}\bd_{t}^{\ru},\frac{\sqrt{\gamma_{t}(1+\mu\eta_{t}/2)}\left(\bx_{\star}-\bx_{t}\right)}{D_{t}}\right\rangle .\label{eq:str-N-def}
\end{align}
As one can see, these two new definitions degenerate to the old ones
in (\ref{eq:cvx-D-N-def}) when $\mu=0$.

Although our goal is to prove last-iterate convergence in terms of
function value gap, we still need the following Lemma \ref{lem:str-core},
which provides a bound on $\left\Vert \bx_{\star}-\bx_{T+1}\right\Vert ^{2}$.
\begin{lem}
\label{lem:str-core}Under Assumptions \ref{assu:minimizer}, \ref{assu:obj}
(with $\mu>0$), and \ref{assu:lip} (resp. \ref{assu:smo}), for
any $T\in\N$ and $\left\{ \eta_{t}\right\} _{t=1}^{T}$ satisfying
$0<\eta_{t}\leq\frac{\eta}{\mu},\forall t\in\left[T\right]$ (resp.
$0<\eta_{t}\leq\min\left\{ \frac{1}{2L},\frac{\eta}{\mu}\right\} ,\forall t\in\left[T\right]$)
where $\eta>0$ is a constant, $\ClipSGD$ (Algorithm \ref{alg:Clipped-SGD})
guarantees
\[
\left\Vert \bx_{\star}-\bx_{T+1}\right\Vert ^{2}\leq\frac{(2+\eta)D^{2}+4\hres_{T}^{\str}+16\sum_{t=1}^{T}\gamma_{t}\eta_{t}^{2}G^{2}\omega}{\gamma_{T}},
\]
where 
\[
I_{T}^{\str}=4\max_{t\in\left[T\right]}\left(\sum_{s=1}^{t}N_{s}\right)^{2}+2\sum_{t=1}^{T}\gamma_{t}\eta_{t}^{2}\left\Vert \bd_{t}^{\ru}\right\Vert ^{2}+\frac{1+2\eta}{\mu}\sum_{t=1}^{T}\gamma_{t}\eta_{t}\left\Vert \bd_{t}^{\rb}\right\Vert ^{2},
\]
 and $\omega=\begin{cases}
1 & \text{under Assumption \ref{assu:lip}}\\
0 & \text{under Assumption \ref{assu:smo}}
\end{cases}$ is introduced in Lemma \ref{lem:descent}.
\end{lem}
\begin{proof}
We invoke Lemma \ref{lem:descent} with $\by=\bx_{\star}$ and replace
the subscript $t$ with $s$ to have, for any $s\in\left[T\right]$,
\[
F(\bx_{s+1})-F_{\star}\leq\frac{\left\Vert \bx_{\star}-\bx_{s}\right\Vert ^{2}-\left\Vert \bx_{\star}-\bx_{s+1}\right\Vert ^{2}}{2\eta_{s}}+\left\langle \bd_{s}^{\rc},\bx_{\star}-\bx_{s}\right\rangle +\eta_{s}\left\Vert \bd_{s}^{\rc}\right\Vert ^{2}+4\eta_{s}G^{2}\omega,
\]
where $\omega=\begin{cases}
1 & \text{under Assumption }\ref{assu:lip}\\
0 & \text{under Assumption }\ref{assu:smo}
\end{cases}$. Note that $F_{\star}=F(\bx_{\star})=\inf_{\bx\in\X}F(\bx)$, which
implies that there exists $\nabla F(\bx_{\star})\in\partial F(\bx_{\star})$
such that $\left\langle \nabla F(\bx_{\star}),\bx-\bx_{\star}\right\rangle \geq0,\forall\bx\in\X$.
Therefore, we know
\[
F(\bx_{s+1})-F_{\star}\overset{\text{Lemma }\ref{lem:str-star-cvx}}{\geq}\frac{\mu}{2}\left\Vert \bx_{s+1}-\bx_{\star}\right\Vert ^{2}.
\]
As such, there is
\[
0\leq\frac{\left\Vert \bx_{\star}-\bx_{s}\right\Vert ^{2}-(1+\mu\eta_{s})\left\Vert \bx_{\star}-\bx_{s+1}\right\Vert ^{2}}{2\eta_{s}}+\left\langle \bd_{s}^{\rc},\bx_{\star}-\bx_{s}\right\rangle +\eta_{s}\left\Vert \bd_{s}^{\rc}\right\Vert ^{2}+4\eta_{s}G^{2}\omega.
\]
With the above inequality in hand, the following proof is identical
to the analysis for Lemma 7 of \citep{ICLR2026_bd5b846a}. Therefore,
we omit all details to save space, but only remind the reader of the
following newly defined quantity:
\[
\hres_{t}^{\str}\defeq4\max_{t^{\prime}\in\left[t\right]}\left(\sum_{s=1}^{t^{\prime}}N_{s}\right)^{2}+2\sum_{s=1}^{t}\gamma_{s}\eta_{s}^{2}\left\Vert \bd_{s}^{\ru}\right\Vert ^{2}+\frac{1+2\eta}{\mu}\sum_{s=1}^{t}\gamma_{s}\eta_{s}\left\Vert \bd_{s}^{\rb}\right\Vert ^{2},\forall t\in\left[T\right].
\]
\end{proof}

\subsubsection{Lipschitz Objectives}

In this part, we consider Lipschitz objectives (i.e., Assumption \ref{assu:lip}).
Our goal is to bound the residual term $I_{T}^{\str}$ in Lemma \ref{lem:str-core}.
Fortunately, this has been done previously by \citep{ICLR2026_bd5b846a}.
\begin{lem}[Lemma 9 of \citep{ICLR2026_bd5b846a}]
\label{lem:str-lip-res}Under Assumptions \ref{assu:lip} and \ref{assu:heavy},
and the following two conditions:
\begin{enumerate}[leftmargin=*]
\item \label{enu:str-lip-res-1}$\eta_{t}$ and $\cm_{t}$ are deterministic
for all $t\in\left[T\right]$.
\item \label{enu:str-lip-res-2}$\cm_{t}\geq2G$ for all $t\in\left[T\right]$.
\end{enumerate}
We have for any $\delta\in\left(0,1\right]$, with probability at
least $1-\delta$, $\hres_{T}^{\str}\leq\hc_{T}^{\str}$, where $\hres_{T}^{\str}$
is introduced in Lemma \ref{lem:str-core} and $\hc_{T}^{\str}$ is
a constant in the order of
\begin{align*}
\O & \left(\max_{t\in\left[T\right]}\gamma_{t}\eta_{t}^{2}\cm_{t}^{2}\ln^{2}\frac{3}{\delta}+\sum_{t=1}^{T}\frac{\sigma_{\lar}^{\p}\gamma_{t}\eta_{t}^{2}}{\cm_{t}^{\p-2}}+\sum_{t=1}^{T}\left(\frac{\sigma_{\sma}^{\p}\gamma_{t}\eta_{t}^{2}}{\cm_{t}^{\p-2}}+\frac{\sigma_{\lar}^{\p}G^{2}\gamma_{t}\eta_{t}^{2}}{\cm_{t}^{\p}}\right)\ln\frac{3}{\delta}\right.\\
 & \left.\quad+\sum_{t=1}^{T}\left(\frac{\sigma_{\sma}^{2}\sigma_{\lar}^{2\p-2}\gamma_{t}\eta_{t}}{\cm_{t}^{2\p-2}}+\frac{\sigma_{\lar}^{2\p}G^{2}\gamma_{t}\eta_{t}}{\cm_{t}^{2\p}}\right)\frac{1+\eta}{\mu}\right).
\end{align*}
\end{lem}
\begin{rem}
Condition \ref{enu:str-lip-res-2} in the original Lemma 9 of \citep{ICLR2026_bd5b846a}
is stated in a more general form, i.e., $\cm_{t}\geq\frac{G}{1-\alpha},\forall t\in\left[T\right]$
for a constant $\alpha\in\left(0,1\right)$. Here, we set $\alpha=\frac{1}{2}$
to accommodate the present paper.
\end{rem}

\subsubsection{Smooth Objectives}

In this part, we consider smooth objectives (i.e., Assumption \ref{assu:smo}).
As discussed before, the difficulty in the smooth case is to bound
$\left\Vert \nabla f(\bx_{t})\right\Vert $ properly. Luckily, on
a compact domain (i.e., Assumption \ref{assu:compact}), this is relatively
easy to achieve. However, to maintain generality, we prove the following
lemma in a manner analogous to Lemma \ref{lem:cvx-smo-res} under
the weaker Assumption \ref{assu:G1}.
\begin{lem}
\label{lem:str-smo-res}Under Assumptions \ref{assu:minimizer}, \ref{assu:smo},
\ref{assu:heavy}, \ref{assu:G1}, and the following two conditions:
\begin{enumerate}[leftmargin=*]
\item \label{enu:str-smo-res-1}$\eta_{t}$ is deterministic for all $t\in\left[T\right]$.
\item \label{enu:str-smo-res-2}$\cm_{t}=\max\left\{ 2G_{t},\cm\right\} $
holds for all $t\in\left[T\right]$, where $G_{t}=L\left\Vert \bx_{t}-\bx_{1}\right\Vert +G_{1}$
and $\cm\geq0$ is deterministic.
\end{enumerate}
We have for any $\delta\in\left(0,1\right]$, with probability at
least $1-\delta$, $\bar{\hres}_{T}^{\str}\leq\bar{\hc}_{T}^{\str}$,
where
\begin{align*}
\bar{\hres}_{T}^{\str} & \defeq4\max_{t\in\left[T\right]}\left(\sum_{s=1}^{t}N_{s}\zeta_{s}\right)^{2}+2\sum_{t=1}^{T}\gamma_{t}\eta_{t}^{2}\left\Vert \bd_{t}^{\ru}\right\Vert ^{2}\zeta_{t}^{2}+\frac{1+2\eta}{\mu}\sum_{t=1}^{T}\gamma_{t}\eta_{t}\left\Vert \bd_{t}^{\rb}\right\Vert ^{2}\zeta_{t}^{2},\\
\zeta_{t} & \defeq\1\left[\left\Vert \bx_{\star}-\bx_{t}\right\Vert \leq\bar{D}\right],\forall t\in\left[T\right],
\end{align*}
$\bar{D}\geq0$ can be any constant, and $\bar{\hc}_{T}^{\str}$ is
a constant in the order of
\begin{align*}
 & \O\left(\max_{t\in\left[T\right]}\gamma_{t}\eta_{t}^{2}\rho^{2}\ln^{2}\frac{3}{\delta}+\sum_{t=1}^{T}\frac{\sigma_{\lar}^{\p}\gamma_{t}\eta_{t}^{2}}{\cm^{\p-2}}+\sum_{t=1}^{T}\left(\frac{\sigma_{\sma}^{\p}\gamma_{t}\eta_{t}^{2}\rho}{\cm^{\p-1}}+\min_{\beta\in\left[0,\frac{1}{2}\right]}\frac{\sigma_{\lar}^{\p}\bar{G}^{2-\beta\p}\gamma_{t}\eta_{t}^{2}}{\cm^{(1-\beta)\p}}\right)\ln\frac{3}{\delta}\right.\\
 & \left.\quad+\left(\sum_{t=1}^{T}\frac{\sigma_{\sma}^{2}\sigma_{\lar}^{2\p-2}\gamma_{t}\eta_{t}}{\cm^{2\p-2}}+\min_{\beta\in\left[0,\frac{1}{2}\right]}\frac{\sigma_{\lar}^{2\p}\bar{G}^{2-2\beta\p}\gamma_{t}\eta_{t}}{\cm^{2(1-\beta)\p}}\right)\frac{1+\eta}{\mu}\right),
\end{align*}
in which $\bar{G}=L(\bar{D}+D)+G_{1}$ and $\rho=\max\left\{ 2\bar{G},\cm\right\} $.
\end{lem}
\begin{proof}
Similar to the proof of Lemma \ref{lem:cvx-smo-res}, let $\bar{G}\defeq L(\bar{D}+D)+G_{1}$
and $\rho\defeq\max\left\{ 2\bar{G},\cm\right\} $, then the following
inequalities hold for any $t\in\left[T\right]$:
\begin{eqnarray}
2\left\Vert \nabla f(\bx_{t})\right\Vert \leq2G_{t}\leq\cm_{t}, & \text{\ensuremath{\left\Vert \nabla f(\bx_{t})\right\Vert \zeta_{t}}}\leq\bar{G}, & \cm_{t}\zeta_{t}\leq\rho.\label{eq:str-smo-res-ineq}
\end{eqnarray}
In addition, $\chi_{t}$ in Lemma \ref{lem:clip-ineq} equals $1$
for all $t\in\left[T\right]$. Then Lemma \ref{lem:clip-ineq} and
(\ref{eq:str-smo-res-ineq}) together yield that the following collection
of inequalities holds for any $t\in\left[T\right]$:
\begin{align}
\left\Vert \bd_{t}^{\ru}\right\Vert  & \leq2\cm_{t},\label{eq:str-smo-res-u-norm}\\
\E\left[\left\Vert \bd_{t}^{\ru}\right\Vert ^{2}\mid\F_{t-1}\right] & \leq\frac{4\sigma_{\lar}^{\p}}{\cm_{t}^{\p-2}},\label{eq:str-smo-res-u-var}\\
\left\Vert \E\left[\bd_{t}^{\ru}\left(\bd_{t}^{\ru}\right)^{\top}\mid\F_{t-1}\right]\right\Vert _{\op} & \leq\frac{4\sigma_{\sma}^{\p}}{\cm_{t}^{\p-2}}+\frac{2^{\p+1}\sigma_{\lar}^{\p}G_{t}^{2}}{\cm_{t}^{\p}},\label{eq:str-smo-res-u-op}\\
\left\Vert \bd_{t}^{\rb}\right\Vert  & \leq\frac{\sigma_{\sma}\sigma_{\lar}^{\p-1}}{\cm_{t}^{\p-1}}+\frac{2^{\p-1}\sigma_{\lar}^{\p}G_{t}}{\cm_{t}^{\p}}.\label{eq:str-smo-res-b-norm}
\end{align}

Now, we are ready to bound $\bar{\hres}_{T}^{\str}$.

\begin{itemize}[leftmargin=*]

\item Similar to (\ref{eq:cvx-smo-res-1}), we can show with probability
at least $1-2\delta/3$,
\begin{equation}
\max_{t\in\left[T\right]}\left(\sum_{s=1}^{t}N_{s}\zeta_{s}\right)^{2}\leq\frac{8}{9}\max_{t\in\left[T\right]}\gamma_{t}\eta_{t}^{2}\rho^{2}\ln^{2}\frac{3}{\delta}+16\sum_{t=1}^{T}\left(\frac{\sigma_{\sma}^{\p}\gamma_{t}\eta_{t}^{2}\rho}{\cm^{\p-1}}+\min_{\beta\in\left[0,\frac{1}{2}\right]}\frac{2^{\p-1}\sigma_{\lar}^{\p}\bar{G}^{2-\beta\p}\gamma_{t}\eta_{t}^{2}}{2^{\beta\p}\cm^{(1-\beta)\p}}\right)\ln\frac{3}{\delta}.\label{eq:str-smo-res-1}
\end{equation}

\item Similar to (\ref{eq:cvx-smo-res-2}), we can show with probability
at least $1-\delta/3$,
\begin{align}
\sum_{t=1}^{T}\gamma_{t}\eta_{t}^{2}\left\Vert \bd_{t}^{\ru}\right\Vert ^{2}\zeta_{t}^{2} & \leq\frac{10}{3}\max_{t\in\left[T\right]}\gamma_{t}\eta_{t}^{2}\rho^{2}\ln\frac{3}{\delta}+8\sum_{t=1}^{T}\left(\frac{\sigma_{\lar}^{\p}\gamma_{t}\eta_{t}^{2}}{\cm^{\p-2}}+\min_{\beta\in\left[0,\frac{1}{2}\right]}\frac{\sigma_{\lar}^{\p}(2\bar{G})^{2-\beta\p}\gamma_{t}\eta_{t}^{2}}{\cm^{(1-\beta)\p}}\right)\nonumber \\
 & \overset{\ln\frac{3}{\delta}\geq1}{\leq}\frac{10}{3}\max_{t\in\left[T\right]}\gamma_{t}\eta_{t}^{2}\rho^{2}\ln^{2}\frac{3}{\delta}+8\sum_{t=1}^{T}\left(\frac{\sigma_{\lar}^{\p}\gamma_{t}\eta_{t}^{2}}{\cm^{\p-2}}+\min_{\beta\in\left[0,\frac{1}{2}\right]}\frac{\sigma_{\lar}^{\p}(2\bar{G})^{2-\beta\p}\gamma_{t}\eta_{t}^{2}}{\cm^{(1-\beta)\p}}\right).\label{eq:str-smo-res-2}
\end{align}

\item Similar to (\ref{eq:cvx-smo-res-3}), we can show
\begin{equation}
\sum_{t=1}^{T}\gamma_{t}\eta_{t}\left\Vert \bd_{t}^{\rb}\right\Vert ^{2}\zeta_{t}^{2}\leq2\sum_{t=1}^{T}\left(\frac{\sigma_{\sma}^{2}\sigma_{\lar}^{2\p-2}\gamma_{t}\eta_{t}}{\cm^{2\p-2}}+\min_{\beta\in\left[0,\frac{1}{2}\right]}\frac{2^{2\p-2}\sigma_{\lar}^{2\p}\bar{G}^{2-2\beta\p}\gamma_{t}\eta_{t}}{2^{2\beta\p}\cm^{2(1-\beta)\p}}\right).\label{eq:str-smo-res-3}
\end{equation}

\end{itemize}

Combine (\ref{eq:str-smo-res-1}), (\ref{eq:str-smo-res-2}), and
(\ref{eq:str-smo-res-3}) to have with probability at least $1-\delta$,
\[
\bar{\hres}_{T}^{\str}=4\max_{t\in\left[T\right]}\left(\sum_{s=1}^{t}N_{s}\zeta_{s}\right)^{2}+2\sum_{t=1}^{T}\gamma_{t}\eta_{t}^{2}\left\Vert \bd_{t}^{\ru}\right\Vert ^{2}\zeta_{t}^{2}+\frac{1+2\eta}{\mu}\sum_{t=1}^{T}\gamma_{t}\eta_{t}\left\Vert \bd_{t}^{\rb}\right\Vert ^{2}\zeta_{t}^{2}\leq\bar{\hc}_{T}^{\str},
\]
where
\begin{align*}
\bar{\hc}_{T}^{\str}\defeq & \frac{92}{9}\max_{t\in\left[T\right]}\gamma_{t}\eta_{t}^{2}\rho^{2}\ln^{2}\frac{3}{\delta}+64\sum_{t=1}^{T}\left(\frac{\sigma_{\sma}^{\p}\gamma_{t}\eta_{t}^{2}\rho}{\cm^{\p-1}}+\min_{\beta\in\left[0,\frac{1}{2}\right]}\frac{2^{\p-1}\sigma_{\lar}^{\p}\bar{G}^{2-\beta\p}\gamma_{t}\eta_{t}^{2}}{2^{\beta\p}\cm^{(1-\beta)\p}}\right)\ln\frac{3}{\delta}\\
 & +16\sum_{t=1}^{T}\left(\frac{\sigma_{\lar}^{\p}\gamma_{t}\eta_{t}^{2}}{\cm^{\p-2}}+\min_{\beta\in\left[0,\frac{1}{2}\right]}\frac{\sigma_{\lar}^{\p}(2\bar{G})^{2-\beta\p}\gamma_{t}\eta_{t}^{2}}{\cm^{(1-\beta)\p}}\right)\\
 & +2\left(\sum_{t=1}^{T}\frac{\sigma_{\sma}^{2}\sigma_{\lar}^{2\p-2}\gamma_{t}\eta_{t}}{\cm^{2\p-2}}+\min_{\beta\in\left[0,\frac{1}{2}\right]}\frac{2^{2\p-2}\sigma_{\lar}^{2\p}\bar{G}^{2-2\beta\p}\gamma_{t}\eta_{t}}{2^{2\beta\p}\cm^{2(1-\beta)\p}}\right)\frac{1+2\eta}{\mu}\\
= & \O\left(\max_{t\in\left[T\right]}\gamma_{t}\eta_{t}^{2}\rho^{2}\ln^{2}\frac{3}{\delta}+\sum_{t=1}^{T}\frac{\sigma_{\lar}^{\p}\gamma_{t}\eta_{t}^{2}}{\cm^{\p-2}}+\sum_{t=1}^{T}\left(\frac{\sigma_{\sma}^{\p}\gamma_{t}\eta_{t}^{2}\rho}{\cm^{\p-1}}+\min_{\beta\in\left[0,\frac{1}{2}\right]}\frac{\sigma_{\lar}^{\p}\bar{G}^{2-\beta\p}\gamma_{t}\eta_{t}^{2}}{\cm^{(1-\beta)\p}}\right)\ln\frac{3}{\delta}\right.\\
 & \left.\quad+\left(\sum_{t=1}^{T}\frac{\sigma_{\sma}^{2}\sigma_{\lar}^{2\p-2}\gamma_{t}\eta_{t}}{\cm^{2\p-2}}+\min_{\beta\in\left[0,\frac{1}{2}\right]}\frac{\sigma_{\lar}^{2\p}\bar{G}^{2-2\beta\p}\gamma_{t}\eta_{t}}{\cm^{2(1-\beta)\p}}\right)\frac{1+\eta}{\mu}\right),
\end{align*}
in which the second equation is by $\ln\frac{3}{\delta}\geq1$.
\end{proof}

\section{Technical Lemmas}

In this section, we list some technical lemmas used in the analysis.
For results already existing in the literature, we omit their proofs
to save space.

First, Lemma \ref{lem:Freedman} is known as Freedman's inequality,
a useful tool to obtain a high-probability concentration bound for
a real-valued MDS.
\begin{lem}[Freedman's inequality \citep{10.1214/aop/1176996452}]
\label{lem:Freedman}Suppose $\left\{ X_{t}\right\} _{t=1}^{T}$
is a real-valued MDS adapted to the filtration $\left\{ \F_{t}\right\} _{t=0}^{T}$
satisfying for any $t\in\left[T\right]$, $X_{t}\leq b$ and $\E\left[X_{t}^{2}\mid\F_{t-1}\right]\leq\sigma_{t}^{2}$
almost surely, where $b\geq0$ and $\sigma_{t}^{2}$ are both constant,
then for any $\delta\in\left(0,1\right]$, there is
\[
\Pr\left[\sum_{s=1}^{t}X_{s}\leq\frac{b}{3}\ln\frac{1}{\delta}+\sqrt{2\sum_{s=1}^{T}\sigma_{s}^{2}\ln\frac{1}{\delta}},\forall t\in\left[T\right]\right]\geq1-\delta.
\]
\end{lem}
Next, we introduce another concentration inequality in Lemma \ref{lem:Freedman-square}.
\begin{lem}
\label{lem:Freedman-square}Suppose $\left\{ X_{t}\right\} _{t=1}^{T}$
is a sequence of random variables adapted to the filtration $\left\{ \F_{t}\right\} _{t=0}^{T}$
satisfying for any $t\in\left[T\right]$, $\left|X_{t}\right|\leq b$
and $\E\left[X_{t}^{2}\mid\F_{t-1}\right]\leq\sigma_{t}^{2}$ almost
surely, where $b\geq0$ is a constant and $\sigma_{t}^{2}\in\F_{t-1}$
is predictable, then for any $\delta\in\left(0,1\right]$, there is
\[
\Pr\left[\sum_{t=1}^{T}X_{t}^{2}\leq\frac{5b^{2}}{6}\ln\frac{1}{\delta}+2\sum_{t=1}^{T}\sigma_{t}^{2}\right]\geq1-\delta.
\]
\end{lem}
\begin{rem}
At first glance, this inequality seems to exist in the prior literature
(e.g., \citep{NEURIPS2021_26901deb,ICLR2026_bd5b846a,liu2023stochastic,NEURIPS2022_349956de}).
However, we emphasize that our lemma is in fact stronger than existing
ones, as ours allows $\sigma_{t}^{2}$ to be a random variable rather
than a constant assumed in prior work. Although its proof is not difficult,
we believe it can provide some fresh insights to the community.
\end{rem}
\begin{proof}
We first have
\begin{equation}
\sum_{t=1}^{T}X_{t}^{2}=\sum_{t=1}^{T}\underbrace{X_{t}^{2}-\E\left[X_{t}^{2}\mid\F_{t-1}\right]}_{\defeq Y_{t}}+\sum_{t=1}^{T}\E\left[X_{t}^{2}\mid\F_{t-1}\right]\leq\sum_{t=1}^{T}Y_{t}+\sum_{t=1}^{T}\sigma_{t}^{2}.\label{eq:Freedman-square-1}
\end{equation}
Note that $\left\{ Y_{t}\right\} _{t=1}^{T}$ is a real-valued MDS
adapted to the filtration $\left\{ \F_{t}\right\} _{t=0}^{T}$ satisfying
\begin{eqnarray}
Y_{t}\leq X_{t}^{2}\leq b^{2} & \text{and} & \E\left[Y_{t}^{2}\mid\F_{t-1}\right]\leq\E\left[X_{t}^{4}\mid\F_{t-1}\right]\leq b^{2}\sigma_{t}^{2}.\label{eq:Freedman-square-2}
\end{eqnarray}

Next, let $h(x)\defeq\begin{cases}
\frac{\exp(x)-1-x}{x^{2}} & x\neq0\\
\frac{1}{2} & x=0
\end{cases}$. It is easy to verify that $h$ is continuous and increasing. Therefore,
given $\lambda>0$, we have for any $t\in\left[T\right]$,
\begin{align}
\exp(\lambda Y_{t}) & =1+\lambda Y_{t}+(\lambda Y_{t})^{2}h(\lambda Y_{t})\overset{(\ref{eq:Freedman-square-2})}{\leq}1+\lambda Y_{t}+(\lambda Y_{t})^{2}h(\lambda b^{2})\nonumber \\
\Rightarrow\E\left[\exp(\lambda Y_{t})\mid\F_{t-1}\right] & \leq1+\lambda\E\left[Y_{t}\mid\F_{t-1}\right]+\lambda^{2}\E\left[Y_{t}^{2}\mid\F_{t-1}\right]h(\lambda b^{2})\nonumber \\
 & \overset{(\ref{eq:Freedman-square-2})}{\leq}1+\lambda^{2}b^{2}\sigma_{t}^{2}h(\lambda b^{2})\leq\exp\left(\lambda^{2}b^{2}\sigma_{t}^{2}h(\lambda b^{2})\right).\label{eq:Freedman-square-3}
\end{align}
As such, we know
\begin{align}
\Pr\left[\sum_{t=1}^{T}Y_{t}-\lambda b^{2}h(\lambda b^{2})\sum_{t=1}^{T}\sigma_{t}^{2}\geq\frac{\ln\frac{1}{\delta}}{\lambda}\right] & =\Pr\left[\exp\left(\lambda\sum_{t=1}^{T}Y_{t}-\lambda^{2}b^{2}h(\lambda b^{2})\sum_{t=1}^{T}\sigma_{t}^{2}\right)\geq\frac{1}{\delta}\right]\nonumber \\
 & \overset{(a)}{\leq}\delta\E\left[\exp\left(\lambda\sum_{t=1}^{T}Y_{t}-\lambda^{2}b^{2}h(\lambda b^{2})\sum_{t=1}^{T}\sigma_{t}^{2}\right)\right]\overset{(b)}{\leq}\delta,\label{eq:Freedman-square-4}
\end{align}
where $(a)$ is due to Markov's inequality and $(b)$ is by applying
(\ref{eq:Freedman-square-3}) recursively.

Finally, by (\ref{eq:Freedman-square-1}) and (\ref{eq:Freedman-square-4}),
for any given $\lambda>0$,
\[
\Pr\left[\sum_{t=1}^{T}X_{t}^{2}\leq\frac{\ln\frac{1}{\delta}}{\lambda}+\left(\lambda b^{2}h(\lambda b^{2})+1\right)\sum_{t=1}^{T}\sigma_{t}^{2}\right]\geq1-\delta.
\]
We pick $\lambda=\frac{6}{5b^{2}}$ and use $\frac{6}{5}h(\frac{6}{5})\leq1$
to conclude.
\end{proof}

Moreover, our analysis relies on the following inequality proposed
in \citep{pmlr-v202-ivgi23a}.
\begin{lem}[Lemma C.2 of \citep{pmlr-v202-ivgi23a}]
\label{lem:Abel}Let $\left\{ a_{t}\right\} _{t=1}^{T}$ and $\left\{ b_{t}\right\} _{t=1}^{T}$
be two sequences in $\R$ such that $a_{t}$ is nonnegative and nondecreasing,
then there is
\[
\left|\sum_{s=1}^{t}a_{s}b_{s}\right|\leq2a_{t}\max_{t^{\prime}\in\left[t\right]}\left|\sum_{s=1}^{t^{\prime}}b_{s}\right|,\forall t\in\left[T\right].
\]
\end{lem}
Finally, we need an algebraic inequality proved in \citep{ICLR2026_bd5b846a}.
As stated in \citep{ICLR2026_bd5b846a}, Lemma \ref{lem:algebra}
(or the same idea) has been used in previous works such as \citep{pmlr-v202-ivgi23a}
and \citep{liu2023stochastic}.
\begin{lem}[Lemma 14 of \citep{ICLR2026_bd5b846a}]
\label{lem:algebra}Let $\left\{ a_{t}\right\} _{t=1}^{T+1}$, $\left\{ b_{t}\right\} _{t=1}^{T}$,
and $\left\{ c_{t}\right\} _{t=1}^{T+1}$ be three sequences in $\R$
such that $b_{t}$ is nonnegative and $c_{t}$ is nondecreasing. If
$a_{1}\leq2c_{1}$ and
\[
a_{t+1}+b_{t}\leq\frac{\max_{s\in\left[t\right]}a_{s}}{2}+c_{t+1},\forall t\in\left[T\right],
\]
then there is
\[
a_{t+1}+b_{t}\leq2c_{t+1},\forall t\in\left[T\right].
\]
\end{lem}
\begin{rem}
Although the above lemma differs slightly from the original one presented
in \citep{ICLR2026_bd5b846a}, it follows essentially the same proof
given in \citep{ICLR2026_bd5b846a}.
\end{rem}

\clearpage

\section*{NeurIPS Paper Checklist}

\begin{enumerate}

\item {\bf Claims}
    \item[] Question: Do the main claims made in the abstract and introduction accurately reflect the paper's contributions and scope?
    \item[] Answer: \answerYes{} % Replace by \answerYes{}, \answerNo{}, or \answerNA{}.
    \item[] Justification: The main claims made in the abstract and introduction accurately reflect the paper's contributions and scope.
    \item[] Guidelines:
    \begin{itemize}
        \item The answer \answerNA{} means that the abstract and introduction do not include the claims made in the paper.
        \item The abstract and/or introduction should clearly state the claims made, including the contributions made in the paper and important assumptions and limitations. A \answerNo{} or \answerNA{} answer to this question will not be perceived well by the reviewers. 
        \item The claims made should match theoretical and experimental results, and reflect how much the results can be expected to generalize to other settings. 
        \item It is fine to include aspirational goals as motivation as long as it is clear that these goals are not attained by the paper. 
    \end{itemize}

\item {\bf Limitations}
    \item[] Question: Does the paper discuss the limitations of the work performed by the authors?
    \item[] Answer: \answerYes{} % Replace by \answerYes{}, \answerNo{}, or \answerNA{}.
    \item[] Justification: In Section \ref{sec:conclusion}, we discuss the limitations of the work.
    \item[] Guidelines:
    \begin{itemize}
        \item The answer \answerNA{} means that the paper has no limitation while the answer \answerNo{} means that the paper has limitations, but those are not discussed in the paper. 
        \item The authors are encouraged to create a separate ``Limitations'' section in their paper.
        \item The paper should point out any strong assumptions and how robust the results are to violations of these assumptions (e.g., independence assumptions, noiseless settings, model well-specification, asymptotic approximations only holding locally). The authors should reflect on how these assumptions might be violated in practice and what the implications would be.
        \item The authors should reflect on the scope of the claims made, e.g., if the approach was only tested on a few datasets or with a few runs. In general, empirical results often depend on implicit assumptions, which should be articulated.
        \item The authors should reflect on the factors that influence the performance of the approach. For example, a facial recognition algorithm may perform poorly when image resolution is low or images are taken in low lighting. Or a speech-to-text system might not be used reliably to provide closed captions for online lectures because it fails to handle technical jargon.
        \item The authors should discuss the computational efficiency of the proposed algorithms and how they scale with dataset size.
        \item If applicable, the authors should discuss possible limitations of their approach to address problems of privacy and fairness.
        \item While the authors might fear that complete honesty about limitations might be used by reviewers as grounds for rejection, a worse outcome might be that reviewers discover limitations that aren't acknowledged in the paper. The authors should use their best judgment and recognize that individual actions in favor of transparency play an important role in developing norms that preserve the integrity of the community. Reviewers will be specifically instructed to not penalize honesty concerning limitations.
    \end{itemize}

\item {\bf Theory assumptions and proofs}
    \item[] Question: For each theoretical result, does the paper provide the full set of assumptions and a complete (and correct) proof?
    \item[] Answer: \answerYes{} % Replace by \answerYes{}, \answerNo{}, or \answerNA{}.
    \item[] Justification: For each theoretical result, the paper provides the full set of assumptions and a complete (and correct) proof.
    \item[] Guidelines:
    \begin{itemize}
        \item The answer \answerNA{} means that the paper does not include theoretical results. 
        \item All the theorems, formulas, and proofs in the paper should be numbered and cross-referenced.
        \item All assumptions should be clearly stated or referenced in the statement of any theorems.
        \item The proofs can either appear in the main paper or the supplemental material, but if they appear in the supplemental material, the authors are encouraged to provide a short proof sketch to provide intuition. 
        \item Inversely, any informal proof provided in the core of the paper should be complemented by formal proofs provided in appendix or supplemental material.
        \item Theorems and Lemmas that the proof relies upon should be properly referenced. 
    \end{itemize}

    \item {\bf Experimental result reproducibility}
    \item[] Question: Does the paper fully disclose all the information needed to reproduce the main experimental results of the paper to the extent that it affects the main claims and/or conclusions of the paper (regardless of whether the code and data are provided or not)?
    \item[] Answer: \answerNA{} % Replace by \answerYes{}, \answerNo{}, or \answerNA{}.
    \item[] Justification: The paper does not include experiments.
    \item[] Guidelines:
    \begin{itemize}
        \item The answer \answerNA{} means that the paper does not include experiments.
        \item If the paper includes experiments, a \answerNo{} answer to this question will not be perceived well by the reviewers: Making the paper reproducible is important, regardless of whether the code and data are provided or not.
        \item If the contribution is a dataset and\slash or model, the authors should describe the steps taken to make their results reproducible or verifiable. 
        \item Depending on the contribution, reproducibility can be accomplished in various ways. For example, if the contribution is a novel architecture, describing the architecture fully might suffice, or if the contribution is a specific model and empirical evaluation, it may be necessary to either make it possible for others to replicate the model with the same dataset, or provide access to the model. In general. releasing code and data is often one good way to accomplish this, but reproducibility can also be provided via detailed instructions for how to replicate the results, access to a hosted model (e.g., in the case of a large language model), releasing of a model checkpoint, or other means that are appropriate to the research performed.
        \item While NeurIPS does not require releasing code, the conference does require all submissions to provide some reasonable avenue for reproducibility, which may depend on the nature of the contribution. For example
        \begin{enumerate}
            \item If the contribution is primarily a new algorithm, the paper should make it clear how to reproduce that algorithm.
            \item If the contribution is primarily a new model architecture, the paper should describe the architecture clearly and fully.
            \item If the contribution is a new model (e.g., a large language model), then there should either be a way to access this model for reproducing the results or a way to reproduce the model (e.g., with an open-source dataset or instructions for how to construct the dataset).
            \item We recognize that reproducibility may be tricky in some cases, in which case authors are welcome to describe the particular way they provide for reproducibility. In the case of closed-source models, it may be that access to the model is limited in some way (e.g., to registered users), but it should be possible for other researchers to have some path to reproducing or verifying the results.
        \end{enumerate}
    \end{itemize}

\item {\bf Open access to data and code}
    \item[] Question: Does the paper provide open access to the data and code, with sufficient instructions to faithfully reproduce the main experimental results, as described in supplemental material?
    \item[] Answer: \answerNA{} % Replace by \answerYes{}, \answerNo{}, or \answerNA{}.
    \item[] Justification: The paper does not include experiments requiring code.
    \item[] Guidelines:
    \begin{itemize}
        \item The answer \answerNA{} means that paper does not include experiments requiring code.
        \item Please see the NeurIPS code and data submission guidelines (\url{https://neurips.cc/public/guides/CodeSubmissionPolicy}) for more details.
        \item While we encourage the release of code and data, we understand that this might not be possible, so \answerNo{} is an acceptable answer. Papers cannot be rejected simply for not including code, unless this is central to the contribution (e.g., for a new open-source benchmark).
        \item The instructions should contain the exact command and environment needed to run to reproduce the results. See the NeurIPS code and data submission guidelines (\url{https://neurips.cc/public/guides/CodeSubmissionPolicy}) for more details.
        \item The authors should provide instructions on data access and preparation, including how to access the raw data, preprocessed data, intermediate data, and generated data, etc.
        \item The authors should provide scripts to reproduce all experimental results for the new proposed method and baselines. If only a subset of experiments are reproducible, they should state which ones are omitted from the script and why.
        \item At submission time, to preserve anonymity, the authors should release anonymized versions (if applicable).
        \item Providing as much information as possible in supplemental material (appended to the paper) is recommended, but including URLs to data and code is permitted.
    \end{itemize}

\item {\bf Experimental setting/details}
    \item[] Question: Does the paper specify all the training and test details (e.g., data splits, hyperparameters, how they were chosen, type of optimizer) necessary to understand the results?
    \item[] Answer: \answerNA{} % Replace by \answerYes{}, \answerNo{}, or \answerNA{}.
    \item[] Justification: The paper does not include experiments.
    \item[] Guidelines:
    \begin{itemize}
        \item The answer \answerNA{} means that the paper does not include experiments.
        \item The experimental setting should be presented in the core of the paper to a level of detail that is necessary to appreciate the results and make sense of them.
        \item The full details can be provided either with the code, in appendix, or as supplemental material.
    \end{itemize}

\item {\bf Experiment statistical significance}
    \item[] Question: Does the paper report error bars suitably and correctly defined or other appropriate information about the statistical significance of the experiments?
    \item[] Answer: \answerNA{} % Replace by \answerYes{}, \answerNo{}, or \answerNA{}.
    \item[] Justification: The paper does not include experiments.
    \item[] Guidelines:
    \begin{itemize}
        \item The answer \answerNA{} means that the paper does not include experiments.
        \item The authors should answer \answerYes{} if the results are accompanied by error bars, confidence intervals, or statistical significance tests, at least for the experiments that support the main claims of the paper.
        \item The factors of variability that the error bars are capturing should be clearly stated (for example, train/test split, initialization, random drawing of some parameter, or overall run with given experimental conditions).
        \item The method for calculating the error bars should be explained (closed form formula, call to a library function, bootstrap, etc.)
        \item The assumptions made should be given (e.g., Normally distributed errors).
        \item It should be clear whether the error bar is the standard deviation or the standard error of the mean.
        \item It is OK to report 1-sigma error bars, but one should state it. The authors should preferably report a 2-sigma error bar than state that they have a 96\% CI, if the hypothesis of Normality of errors is not verified.
        \item For asymmetric distributions, the authors should be careful not to show in tables or figures symmetric error bars that would yield results that are out of range (e.g., negative error rates).
        \item If error bars are reported in tables or plots, the authors should explain in the text how they were calculated and reference the corresponding figures or tables in the text.
    \end{itemize}

\item {\bf Experiments compute resources}
    \item[] Question: For each experiment, does the paper provide sufficient information on the computer resources (type of compute workers, memory, time of execution) needed to reproduce the experiments?
    \item[] Answer: \answerNA{} % Replace by \answerYes{}, \answerNo{}, or \answerNA{}.
    \item[] Justification: The paper does not include experiments.
    \item[] Guidelines:
    \begin{itemize}
        \item The answer \answerNA{} means that the paper does not include experiments.
        \item The paper should indicate the type of compute workers CPU or GPU, internal cluster, or cloud provider, including relevant memory and storage.
        \item The paper should provide the amount of compute required for each of the individual experimental runs as well as estimate the total compute. 
        \item The paper should disclose whether the full research project required more compute than the experiments reported in the paper (e.g., preliminary or failed experiments that didn't make it into the paper). 
    \end{itemize}
    
\item {\bf Code of ethics}
    \item[] Question: Does the research conducted in the paper conform, in every respect, with the NeurIPS Code of Ethics \url{https://neurips.cc/public/EthicsGuidelines}?
    \item[] Answer: \answerYes{} % Replace by \answerYes{}, \answerNo{}, or \answerNA{}.
    \item[] Justification: The research conducted in the paper conforms, in every respect, with the NeurIPS Code of Ethics.
    \item[] Guidelines:
    \begin{itemize}
        \item The answer \answerNA{} means that the authors have not reviewed the NeurIPS Code of Ethics.
        \item If the authors answer \answerNo, they should explain the special circumstances that require a deviation from the Code of Ethics.
        \item The authors should make sure to preserve anonymity (e.g., if there is a special consideration due to laws or regulations in their jurisdiction).
    \end{itemize}

\item {\bf Broader impacts}
    \item[] Question: Does the paper discuss both potential positive societal impacts and negative societal impacts of the work performed?
    \item[] Answer: \answerNA{} % Replace by \answerYes{}, \answerNo{}, or \answerNA{}.
    \item[] Justification: There is no societal impact of the work performed.
    \item[] Guidelines:
    \begin{itemize}
        \item The answer \answerNA{} means that there is no societal impact of the work performed.
        \item If the authors answer \answerNA{} or \answerNo, they should explain why their work has no societal impact or why the paper does not address societal impact.
        \item Examples of negative societal impacts include potential malicious or unintended uses (e.g., disinformation, generating fake profiles, surveillance), fairness considerations (e.g., deployment of technologies that could make decisions that unfairly impact specific groups), privacy considerations, and security considerations.
        \item The conference expects that many papers will be foundational research and not tied to particular applications, let alone deployments. However, if there is a direct path to any negative applications, the authors should point it out. For example, it is legitimate to point out that an improvement in the quality of generative models could be used to generate Deepfakes for disinformation. On the other hand, it is not needed to point out that a generic algorithm for optimizing neural networks could enable people to train models that generate Deepfakes faster.
        \item The authors should consider possible harms that could arise when the technology is being used as intended and functioning correctly, harms that could arise when the technology is being used as intended but gives incorrect results, and harms following from (intentional or unintentional) misuse of the technology.
        \item If there are negative societal impacts, the authors could also discuss possible mitigation strategies (e.g., gated release of models, providing defenses in addition to attacks, mechanisms for monitoring misuse, mechanisms to monitor how a system learns from feedback over time, improving the efficiency and accessibility of ML).
    \end{itemize}
    
\item {\bf Safeguards}
    \item[] Question: Does the paper describe safeguards that have been put in place for responsible release of data or models that have a high risk for misuse (e.g., pre-trained language models, image generators, or scraped datasets)?
    \item[] Answer: \answerNA{} % Replace by \answerYes{}, \answerNo{}, or \answerNA{}.
    \item[] Justification: The paper poses no such risks.
    \item[] Guidelines:
    \begin{itemize}
        \item The answer \answerNA{} means that the paper poses no such risks.
        \item Released models that have a high risk for misuse or dual-use should be released with necessary safeguards to allow for controlled use of the model, for example by requiring that users adhere to usage guidelines or restrictions to access the model or implementing safety filters. 
        \item Datasets that have been scraped from the Internet could pose safety risks. The authors should describe how they avoided releasing unsafe images.
        \item We recognize that providing effective safeguards is challenging, and many papers do not require this, but we encourage authors to take this into account and make a best faith effort.
    \end{itemize}

\item {\bf Licenses for existing assets}
    \item[] Question: Are the creators or original owners of assets (e.g., code, data, models), used in the paper, properly credited and are the license and terms of use explicitly mentioned and properly respected?
    \item[] Answer: \answerNA{} % Replace by \answerYes{}, \answerNo{}, or \answerNA{}.
    \item[] Justification: The paper does not use existing assets.
    \item[] Guidelines:
    \begin{itemize}
        \item The answer \answerNA{} means that the paper does not use existing assets.
        \item The authors should cite the original paper that produced the code package or dataset.
        \item The authors should state which version of the asset is used and, if possible, include a URL.
        \item The name of the license (e.g., CC-BY 4.0) should be included for each asset.
        \item For scraped data from a particular source (e.g., website), the copyright and terms of service of that source should be provided.
        \item If assets are released, the license, copyright information, and terms of use in the package should be provided. For popular datasets, \url{paperswithcode.com/datasets} has curated licenses for some datasets. Their licensing guide can help determine the license of a dataset.
        \item For existing datasets that are re-packaged, both the original license and the license of the derived asset (if it has changed) should be provided.
        \item If this information is not available online, the authors are encouraged to reach out to the asset's creators.
    \end{itemize}

\item {\bf New assets}
    \item[] Question: Are new assets introduced in the paper well documented and is the documentation provided alongside the assets?
    \item[] Answer: \answerNA{} % Replace by \answerYes{}, \answerNo{}, or \answerNA{}.
    \item[] Justification: The paper does not release new assets.
    \item[] Guidelines:
    \begin{itemize}
        \item The answer \answerNA{} means that the paper does not release new assets.
        \item Researchers should communicate the details of the dataset\slash code\slash model as part of their submissions via structured templates. This includes details about training, license, limitations, etc. 
        \item The paper should discuss whether and how consent was obtained from people whose asset is used.
        \item At submission time, remember to anonymize your assets (if applicable). You can either create an anonymized URL or include an anonymized zip file.
    \end{itemize}

\item {\bf Crowdsourcing and research with human subjects}
    \item[] Question: For crowdsourcing experiments and research with human subjects, does the paper include the full text of instructions given to participants and screenshots, if applicable, as well as details about compensation (if any)? 
    \item[] Answer: \answerNA{} % Replace by \answerYes{}, \answerNo{}, or \answerNA{}.
    \item[] Justification: The paper does not involve crowdsourcing nor research with human subjects.
    \item[] Guidelines:
    \begin{itemize}
        \item The answer \answerNA{} means that the paper does not involve crowdsourcing nor research with human subjects.
        \item Including this information in the supplemental material is fine, but if the main contribution of the paper involves human subjects, then as much detail as possible should be included in the main paper. 
        \item According to the NeurIPS Code of Ethics, workers involved in data collection, curation, or other labor should be paid at least the minimum wage in the country of the data collector. 
    \end{itemize}

\item {\bf Institutional review board (IRB) approvals or equivalent for research with human subjects}
    \item[] Question: Does the paper describe potential risks incurred by study participants, whether such risks were disclosed to the subjects, and whether Institutional Review Board (IRB) approvals (or an equivalent approval/review based on the requirements of your country or institution) were obtained?
    \item[] Answer: \answerNA{} % Replace by \answerYes{}, \answerNo{}, or \answerNA{}.
    \item[] Justification: The paper does not involve crowdsourcing nor research with human subjects.
    \item[] Guidelines:
    \begin{itemize}
        \item The answer \answerNA{} means that the paper does not involve crowdsourcing nor research with human subjects.
        \item Depending on the country in which research is conducted, IRB approval (or equivalent) may be required for any human subjects research. If you obtained IRB approval, you should clearly state this in the paper. 
        \item We recognize that the procedures for this may vary significantly between institutions and locations, and we expect authors to adhere to the NeurIPS Code of Ethics and the guidelines for their institution. 
        \item For initial submissions, do not include any information that would break anonymity (if applicable), such as the institution conducting the review.
    \end{itemize}

\item {\bf Declaration of LLM usage}
    \item[] Question: Does the paper describe the usage of LLMs if it is an important, original, or non-standard component of the core methods in this research? Note that if the LLM is used only for writing, editing, or formatting purposes and does \emph{not} impact the core methodology, scientific rigor, or originality of the research, declaration is not required.
    %this research? 
    \item[] Answer: \answerNA{} % Replace by \answerYes{}, \answerNo{}, or \answerNA{}.
    \item[] Justification: The core method development in this research does not involve LLMs as any important, original, or non-standard components.
    \item[] Guidelines:
    \begin{itemize}
        \item The answer \answerNA{} means that the core method development in this research does not involve LLMs as any important, original, or non-standard components.
        \item Please refer to our LLM policy in the NeurIPS handbook for what should or should not be described.
    \end{itemize}

\end{enumerate}
\end{document}